\documentclass[12pt]{amsart}
\usepackage{epsfig,color}
\usepackage{blindtext}
\usepackage{hyperref}
\usepackage{graphicx}
\usepackage{enumitem} 
\usepackage{url}
\usepackage{amssymb}
\usepackage{graphicx,import}
\usepackage{comment}
\usepackage{ulem}
\usepackage{outlines}
\usepackage{esint}
\usepackage{verbatim}
\usepackage{mathrsfs}
\usepackage{mathtools}
\usepackage{amsmath}
\usepackage{dsfont}

\makeatletter
\def\l@subsection{\@tocline{2}{0pt}{2.8pc}{5pc}{}}
\def\l@subsubsection{\@tocline{2}{0pt}{5.6pc}{5pc}{}}

\theoremstyle{definition}

\newtheorem{theorem}{Theorem}[section]
\newtheorem{definition}[theorem]{Definition}

\newtheorem{lemma}[theorem]{Lemma}
\newtheorem{remark}[theorem]{Remark}

\newtheorem{example}[theorem]{Example}

\newtheorem*{remark*}{Remark}

\numberwithin{equation}{section}

\newcommand{\mf}{\mathbf}
\newcommand{\mc}{\mathcal}
\newcommand{\mb}{\mathbb}
\newcommand{\mr}{\mathrm}
\newcommand{\mfk}{\mathfrak}
\newcommand{\spt}{\mathrm{spt}}
\newcommand{\dist}{\mathrm{dist}}
\newcommand{\bH}{\mathbf{H}}
\newcommand{\pa}{\partial}

\DeclareMathOperator{\vol}{Vol}
\DeclareMathOperator{\II}{II}

\title[On the topology of Gaussian random zero sets]{On the topology of Gaussian random zero sets}

\date{\today}

\author{Zhengjiang Lin}
\address{Courant Institute of Mathematical Sciences, New York University, 251 Mercer Street, New York, NY10012, USA}
\email{malin at nyu.edu}

\begin{document}

\begin{abstract}
	We study the asymptotic laws for the number, Betti numbers, and isotopy classes of connected components of zero sets of real Gaussian random fields, where the random zero sets almost surely consist of submanifolds of codimension greater than or equal to one. Our results include `random knots' as a special case.
    Our work is closely related to a series of questions posed by Berry in~\cite{berry2001knotted,berry2000phase}; in particular, our results apply to the ensembles of random knots that appear in the complex arithmetic random waves (Example~\ref{Example: complex random wave}), the Bargmann-Fock model (Example~\ref{Example: Bargmann-Fock}), Black-Body radiation (Example~\ref{Example: Black-Body radiation}), and Berry's monochromatic random waves.
    Our proofs combine techniques introduced for level sets of random scalar-valued functions with methods from differential geometry and differential topology.
\end{abstract}

\maketitle

\tableofcontents

\pagebreak

\section{Introduction}

\subsection{Background}

We consider a Gaussian random field $F: \mathbb{R}^n \times (\Omega, \mfk{S},\mc{P}) \to \mathbb{R}^m$ with $n > m$, where $(\Omega, \mfk{S},\mc{P})$ is a probability space. That is, for every $k \in \mb{Z}_+$ and every given $x_1,x_2,\dots,x_k \in \mb{R}^n$, $(F(x_1, \cdot), F(x_2, \cdot), \dots, F(x_k, \cdot) )$ is a joint Gaussian random vector; and for every fixed $\omega \in \Omega$, $F(\cdot, \omega)$ is a map (vector field) from $\mb{R}^n$ to $\mb{R}^m$. 
See Example~\ref{Example: Bargmann-Fock} and Example~\ref{Example: Black-Body radiation} for Gaussian fields defined on $\mb{R}^n$, and also Example~\ref{Example: complex random wave} for Gaussian random fields defined on a manifold.
For this map $F$, one can define its random zero set as $Z(F) \equiv \{x : F(x) = 0 \in \mathbb{R}^m \}$. 
After some regularity assumptions on $F$, $Z(F)$ a.s.~consists of many submanifolds, either closed or open, of codimension $m$, as proved later in Lemma~\ref{L: Quantitative Bulinskaya}.

For example, when $n=3,m=2$, $Z(F)$ possibly contains many knots and links, which is a very interesting and deep branch in algebraic topology. 
On the other hand, `random knots' have appeared in many applied sciences like quantum physics~\cite{berry2001knotted,berry2000phase}, cosmic strings~\cite{vilenkin1994cosmic}, classical fluid flows~\cite{kleckner2013creation}, and superfluid flows~\cite{kleckner2016superfluid}, for example. 
In particular, in~\cite{taylor2016vortex}, numerical experiments showed that random knots are prevailing in many physical models, where those knots are realized by zeros of random complex scalar wavefunctions.

The term `random knots' has its own pure mathematical interest. The random zero set $Z(F)$ actually gives a probability distribution on knots. 
We can probably ask, how many trefoil knots are in $Z(F) \cap B_1 \subset \mathbb{R}^3$? 
Because $Z(F)$ is a random set, the number of trefoil knots is also a random variable. Is this number always a nontrivial (nonzero) random variable? 
What is the expectation or what are the higher moments of such a random variable?
We can ask similar questions on other types of knots and obtain a random empirical distribution on all types of knots. 
Among all types of knots, what are the most typical knots?
We study similar questions for more general $n,m$ and for $F$ defined on more general topological spaces like manifolds.

The zero sets of random scalar-valued functions have been considered by many authors; a selective list of mathematical work on this topic includes~\cite{NS09, NS16, SW19, S16, W21}. 
In considering the zero sets of vector-valued functions, we must deal with the fact that the geometry and topology of submanifolds of codimension greater than one is much richer and more complex than the codimension-one setting.

\subsection{Main Results}\label{Subsection: main results}

\subsubsection{Stationary Gaussian Fields on Euclidean Spaces}

We consider a $C^{3-}$-smooth Gaussian field $F: \mathbb{R}^n \times (\Omega, \mfk{S},\mc{P}) \to \mathbb{R}^{m}$, $m \in(0,n)$. Here, $C^{3-} \equiv \cap_{0 < \beta < 1} C^{2+ \beta}$.
If we write $F=(f_1, f_2,\dots, f_m)$, then if the two point covariance matrix $K(x,y)$ with elements $ K_{ij}(x,y) \equiv \mc{E}(f_i(x)f_j(y))$ is $C^3$-smooth, we can know that $F$ is a.s.~$C^{3-}$-smooth.
Here, we use $\mathcal{E}(\cdot)$ to denote expectations. On the other hand, Kolmogorov's theorem says that this $F$ is uniquely determined by $K$ up to an equivalence.

We give two examples first, and readers can see~\cite{dalmao20193} for more. Although $F$'s in these examples have i.i.d.~components $\{f_i{\}}$, we do not need such an assumption in any of our theorems.
\begin{example}[Bargmann-Fock model]\label{Example: Bargmann-Fock}
    Let ${\{f_i\}}_{i=1} ^m$ be i.i.d.~scalar-valued Gaussian random functions defined on $\mb{R}^n$ with the same two point function $ k(x,y) \equiv \mc{E}(f_1(x)f_1(y)) = e^{-\frac{{||x-y||}^2}{2}}$.
    Here, $||\cdot||$ is the standard $L^2$-distance. Notice that $k(x,y)$ is a function of $x-y$, which we can write as $k(x-y)$. The Fourier transform of $k(z)$ is called the \textit{spectral measure} of $f_1$, which is ${(2\pi)}^{-n/2} e^{-{||z||}^2/2}$.
\end{example}

\begin{example}[Black-Body radiation]\label{Example: Black-Body radiation}
    Let $n=3$ and $f_1,f_2$ be i.i.d.~scalar-valued Gaussian random functions defined on $\mb{R}^3$ with the same two point function $ k(x,y) \equiv \mc{E}(f_1(x)f_1(y)) = \frac{c_1}{{||x-y||}^2} - \frac{c_2 ||x-y|| \cosh{(||x-y||)}}{{\sinh{(||x-y||)}}^2}$.
    Here $c_1,c_2$ are two fixed constants. Again, $k(x,y) = k(x-y)$. The Fourier transform of $k(z)$ is $\frac{c ||z||}{e^{||z||}-1}$ with another constant $c$.
\end{example}
In Example~\ref{Example: Bargmann-Fock} and Example~\ref{Example: Black-Body radiation}, the $k(x,y)$'s are functions of $x-y$.
We say that a random field $F$ is \textit{stationary} if its covariance matrix $K(x,y) = K(x-y)$.


Recall that $Z(F)$ a.s.~consists of many submanifolds, either closed or open. We use $N(R;F)$ to denote the number of connected components of $Z(F)$ fully contained in the open cube $C_R  \equiv {(-R,R)}^n $ of $\mathbb{R}^n$.
Similarly, for each $l = 0, 1, \dots, n-m$, we define $\beta_l(R;F)$ as the sum of $l$-th Betti numbers over $\mathbb{R}$ for connected components of $Z(F)$ fully contained in $C_R$, i.e.,
    \begin{equation}
        \beta_l(R;F) \equiv \sum_{\gamma \subset Z(F) \cap C_R} \beta_l(\gamma),
    \end{equation}
where the summation counts closed components $\gamma$ fully contained in $C_R$ and $\beta_l(\gamma)$ is the $l$-th Betti number over $\mathbb{R}$ of $\gamma$.

We introduce $\mc{C}(n-m)$ as the set of $C^1$-isotopy classes of closed $(n-m)$-dimensional manifolds that can be embedded into $\mathbb{R}^n$ so that the embeddings have trivial normal bundles. 
That is, for any two $(n-m)$-dimensional manifolds $M_1,M_2$ embedded in $\mathbb{R}^n$ with trivial normal bundles, if $M_1$ is $C^1$-isotopic to $M_2$, then they are in the same class $c \in \mathcal{C}$.

This trivial normal bundle assumption is necessary in this paper. 
It is not always correct that for any closed $(n-m)$-dimensional submanifold $M \subset \mathbb{R}^n$, there is a $G \in C^1(\mathbb{R}^n, \mathbb{R}^m)$ such that $M$ is a non-degenerate connected component of $Z(G)$, which means that $\nabla G$ is of full rank on $M$. 
But if the normal bundle of $M$ in $\mathbb{R}^n$ is trivial, then it is possible to find such a $G$. 
We also notice that $\mc{C}(n-m)$ is a countable set, because any $G$ can be approximated on a compact set by polynomials.

Embeddings of $\mb{S}^1$ always have trivial normal bundles in any $\mb{R}^n$. So, when $n=3,m=2$, $\mc{C}(3-2)$ consists of the classes of all knots. One can also replace $\mc{C}(3-2)$ with all possible links in the following theorems, but for simplicity, we keep this definition.

In the following theorems, we let $N(R;F,c)$ be the number of connected components of $Z(F)$ of class $c \in \mc{C}(n-m)$ which are fully contained in $C_R$.
Example~\ref{Example: Bargmann-Fock}, Example~\ref{Example: Black-Body radiation}, and Berry's monochromatic random waves model all satisfy the assumptions in the following theorems.

\begin{theorem}\label{Main Results: Local}
	\textit{
    Assume that $F:\mathbb{R}^n \to \mathbb{R}^m$, $m \in (0,n)$, is a $C^{3-}$-smooth centered stationary Gaussian random field such that the joint distribution of $(F(0),\nabla F(0))$ is non-degenerate and mean-zero.
    \begin{itemize}
        \item[(1)] For each $c \in \mc{C}(n-m)$, there exists a number $\nu_{F,c} \geq 0$ so that
                    \begin{equation}
                        \mathcal{E}(N(R;F,c)) = \nu_{F,c} \cdot |C_R| + o(R^n),
                    \end{equation} 
                    as $R \to +\infty$. Here $|C_R|$ is the volume of $C_R$.
        \item[(2)] If the translation action of $\mathbb{R}^n$ is ergodic, then
                    \begin{equation}
                        \lim_{R \to \infty}\frac{N(R;F,c)}{|C_R|} = \nu_{F,c} \ \text{almost surely} \quad \text{and} \quad \lim_{R  \to \infty}\mathcal{E} \bigg| \frac{N(R;F,c)}{|C_R|} - \nu_{F,c} \bigg| = 0 .
                    \end{equation}
    \end{itemize}
	(1) and (2) also hold true if one replaces $N(R;F,c)$ with $N(R;F)$ and replaces $\nu_{F,c}$ with $\nu_F \geq 0$.
    \begin{itemize}
	\item[(3)]
	    \begin{equation}
		\nu_F = \sum_{c \in \mc{C}(n-m)} \nu_{F,c} .
	    \end{equation} 
    \end{itemize}
    }
\end{theorem}
See more details in Section~\ref{Section: Stationary Random Field} about the translation action of $\mathbb{R}^n$ and the definition of ergodicity. 

A natural question is when $\nu_F$ and $\nu_{F,c}$'s are positive. 
In the special case when $F$ consists of independent ${\{f_i\}}_{i=1} ^m$, if for each $i = 1, \dots, m$, the spectral measure $\rho_i$ associated with $f_i$ is the surface measure on the unit sphere $\mb{S}^{n-1}$ (Berry's monochromatic random waves), or if its support has a nonempty interior, then one can get that $\nu_F >0$ and $\nu_{F,c}>0$ for all $c \in \mc{C}(n-m)$.
See also~\cite{NS16,SW19, W21} for the scalar-valued ($m=1$) case. We will also discuss some weaker assumptions in Appendix~\ref{APP: Positivity}. 

For the Betti numbers $\beta_l(R;F)$, we have the following counterpart.

\begin{theorem}\label{Main Results: Local Betti}
	\textit{
    Assume that $F:\mathbb{R}^n \to \mathbb{R}^m$, $m \in (0,n)$, is a $C^{3-}$-smooth centered stationary Gaussian random field such that the joint distribution of $(F(0),\nabla F(0))$ is non-degenerate and mean-zero. Then, for each $l = 0,1, \dots ,n-m$, we have the following results.
    \begin{itemize}
        \item[(1)] There exists a number $\nu_{l;F} \geq 0$ so that
                    \begin{equation}
                        \mathcal{E}(\beta_l(R;F)) = \nu_{l;F} \cdot |C_R| + o (R^n),
                    \end{equation} 
                    as $R \to +\infty$.
        \item[(2)] If the translation action of $\mathbb{R}^n$ is ergodic, then
                    \begin{equation}
                        \lim_{R  \to \infty}\mathcal{E} \bigg| \frac{\beta_l(R;F)}{|C_R|} - \nu_{l;F} \bigg| = 0 .
                    \end{equation}
    \end{itemize}
    }
\end{theorem}
Under the assumptions we discussed after Theorem~\ref{Main Results: Local} (or under weaker assumptions discussed in Appendix~\ref{APP: Positivity}), these $\nu_{l;F}$ are positive numbers.


\subsubsection{Gaussian Fields Ensembles on Manifolds}

We turn to some families of Gaussian random fields $\{F_L{\}}$ defined on $X^n$, a $n$-dimensional closed manifold, where each $F_L$ takes values in $\mb{R}^m$ with $n>m$.
The results that we discuss here are closely related to questions raised by Berry at the end of his paper~\cite{berry2001knotted} and section 7 of another paper~\cite{berry2000phase}.
We first give an example of a family of Gaussian random fields, called complex arithmetic random waves, see for example~\cite{dalmao2019phase}. This model also appeared in many numerical experiments on random knots, see for example~\cite{taylor2016vortex}.

\begin{example}\label{Example: complex random wave}
	We choose the $n$-dimensional torus $X^n = \mb{T}^n = \mb{R}^n / \mb{Z}^n$. Let $\mc{H}_L$ be the subspace of $L^2(X)$ that consists of trigonometric polynomials of degree $\lambda \in \mb{Z}^n$ with $ |\lambda|_2 = L$. 
 $L$ takes values in the set $\mc{L} \equiv \{L = {(\sum_{i=1} ^n \lambda_i ^2)}^{1/2}  : \lambda_1,\lambda_2, \dots,\lambda_n \in \mb{Z} {\}}  $.
	We can define a complex Gaussian distribution on $\mc{H}_L$ by
		\begin{equation}
			F_L(x) = \sum_{\lambda \in \mb{Z}^n, |\lambda|_2 = L} (\xi_{\lambda} + \eta_{\lambda} \cdot i ) e^{(2\pi i \langle \lambda , x \rangle)},
		\end{equation}
	where $\xi_\lambda,\eta_\lambda$ are i.i.d.~standard real Gaussian random variables, $\langle \lambda , x \rangle$ is the standard inner product in $\mb{R}^n$.
	We call the sequence of complex random fields $\{F_L{\}}$ the ensemble of complex arithmetic random waves.
	It is not hard to see that $\mr{Re}( F_L(x))$ and $\mr{Im}( F_L(x))$ are two independent real Gaussian random functions and they have the same distribution. Hence, $Z(F_L) \equiv \{ x \in X : F_L(x) = 0 {\}}$ should consist of submanifolds of codimension $2$.
	In particular, when $n=3$, it is the random knot model studied in~\cite{taylor2016vortex} and is closely related to Berry's work~\cite{berry2001knotted, berry2000phase}.
\end{example}

In Section~\ref{Section: Local Double Limit and Global Limit}, we will give another example of $\{F_L\}$ defined on $\mb{S}^{n}$, Example~\ref{Example}, called Kostlan's ensemble, to illustrate more technical details.
For those $\{F_L\}$ satisfying the technical but natural assumptions in Section~\ref{Section: Local Double Limit and Global Limit}, which include complex arithmetic random waves and the Kostlan's ensemble, we say that $\{F_L\}$ is \textit{tame}.
We first state the results for $\{F_L\}$ defined on an open set of $\mb{R}^n$ and then state it on general closed Riemannian manifolds.

\begin{theorem}\label{Main Results: Global 1}
	\textit{
	If ${\{F_L\}}_{L \in \mathcal{L}}$ is tame on an open set $U \subset \mathbb{R}^n$, then we have the following:
    \begin{itemize}
        \item [(1)] There is a measurable locally bounded function $x \mapsto \bar{\nu}(x )$ on $U$, such that
                    for every sequence of connected component counting measures $n_L$ of $F_L$ and for every $\varphi \in C_{c}(U)$, we have that
                        \begin{equation}
                            \lim_{L \to \infty} \mathcal{E} \bigg[ \bigg |  \frac{1}{L^n} \int_U \varphi(x)  dn_L(x) - \int_U \varphi(x) \bar{\nu}(x) dx \bigg | \bigg] = 0.
                        \end{equation}
        \item[(2)] 	In particular, for a bounded smooth domain $D \Subset U$,
                        \begin{equation}
                            \lim_{L \to \infty} \mathcal{E} \bigg[ \bigg |  \frac{N(D;F_L)}{L^n}  - \int_D  \bar{\nu}(x) dx \bigg | \bigg] = 0,
                        \end{equation}
                    where $N(D;F_L)$ is the number of connected components of $Z(F_L)$ fully contained in $D$.
    \end{itemize}
	}
\end{theorem}
Here, we say that a Borel measure $n_L$ is a connected component counting measure of $F_L$ if $\spt (n_L) \subset Z(F_L)$ and the $n_L$-mass of each component of $Z(F_L)$ is $1$. These $n_L$'s are random.

Now, we state a manifold version of Theorem~\ref{Main Results: Global 1}. 
Suppose that $X$ is an $n$-dimensional $C^3$-manifold without boundary. 

\begin{theorem}\label{Main Results: Global 2}
    \textit{
    Assume that ${\{F_L\}}_{L \in \mathcal{L}}$ is tame on $X$, then we have the following:
		\begin{itemize}
			\item [(1)] There is a locally finite Borel non-negative measure $n_\infty$ on $X$ such that for every choice of connected component counting measures $n_L$ of $F_L$ and every $\varphi \in C_{c}(X)$, we have that
                \begin{equation}
                    \lim_{L \to \infty} \mathcal{E} \bigg[ \bigg |  \frac{1}{L^n} \int_X \varphi(x)  dn_L(x) - \int_X \varphi(x) dn_\infty(x) \bigg | \bigg] = 0.
                \end{equation}
            \item [(2)] In particular, when $X$ is a closed manifold,
                \begin{equation}
                    \lim_{L \to \infty} \mathcal{E} \bigg[ \bigg |  \frac{N(X;F_L)}{L^n}  -  n_{\infty}(X) \bigg | \bigg] = 0,
                \end{equation}
                where $N(X;F_L)$ is the number of connected components of $Z(F_L)$.
		\end{itemize}
	}
\end{theorem}
\begin{remark}
    $n_{\infty}(\cdot)$ in Theorem~\ref{Main Results: Global 2}, $\bar{\nu}(\cdot)$ in Theorem~\ref{Main Results: Global 1}, and $\bar{\nu}_l(\cdot)$'s in Theorem~\ref{Main Results: Global Betti} later, are strictly positive under assumptions similar to those we discussed for Theorem~\ref{Main Results: Local} and Theorem~\ref{Main Results: Local Betti}.
    In particular, for the complex arithmetic random waves and Kostlan's ensemble, the corresponding $n_{\infty}(\cdot)$ and $\bar{\nu}_l(\cdot)$'s are strictly positive.
\end{remark}

\begin{remark}\label{Rmk: Global Limiting Measure}
    For any $c \in \mc{C}(n-m)$, results like Theorem~\ref{Main Results: Global 1} and Theorem~\ref{Main Results: Global 2} also hold true if one replaces `connected component counting' with `connected component with type $c$ counting'.
    In particular, for a closed manifold $X$, one can counstruct an empirical isotopy class counting measure on $\mc{C}(n-m)$ like~\cite{SW19}. More precisely,
    given $N(X;F_L)$ and $N(X;F_L, c)$ (number of connected components of type $c$) as in Theorem~\ref{Main Results: Global 2} and its analogies, we define
		\begin{equation}
			\mu_{L} \equiv \frac{1}{N(X; F_L)} \cdot \sum_{c \in \mc{C}(n-m)} N(X; F_L,c) \cdot \delta_c,
		\end{equation}
	which is a random probability measure on $\mc{C}(n-m)$. Then, we can also obtain a higher codimension generalization of Theorem 1.1 of~\cite{SW19}. 
    That is, when $n_{\infty}(X) >0$, there is a limiting probability measure (in discrepancy sense) on $\mc{C}(n-m)$ defined by
		\begin{equation}
			\mu_{\infty} \equiv \frac{1}{n_{\infty}(X)} \cdot \sum_{c \in \mc{C}(n-m)} n_{\infty,H}(X) \cdot \delta_H,
		\end{equation}
	such that for any $\epsilon >0$,
		\begin{equation}
			\lim_{L \to \infty} \mathcal{P}\big( \sup_{A \subset \mc{C}(n-m)} |\mu_L(A) - \mu_\infty(A)| > \epsilon\big) = 0.
		\end{equation}
    We omit the proof of this last result from this paper. The proof is not a direct corollary of~\cite{SW19}, while it is fairly straightforward once one has the tools developed in this paper. 
\end{remark}

Our final result is on asymptotic laws of Betti numbers.
For each $l = 0, \dots,n-m$, we use $\beta_l(F_L)$ to denote the summation of $l$-th Betti numbers of all connected components of $Z(F_L)$ in the given closed $n$-dimensional manifold $X$.

\begin{theorem}\label{Main Results: Global Betti}
    \textit{
        Assume that ${\{F_L\}}_{L \in \mathcal{L}}$ is tame on $X$ and $X$ is a closed manifold. Then for each $l = 0, \dots, n-m$, there is a nonnegative measurable and bounded function $\bar{\nu}_l(x)$ on $X$, and a positive constant $C_l$ depending on ${\{F_L{\}}}_{L \in \mathcal{L}}$, such that
        \begin{equation}
            \int_{X} \bar{\nu}_l(x) \ d x \leq \liminf_{L \to \infty}\frac{\mathcal{E}(\beta_l(F_L))}{L^n} \leq \limsup_{L \to \infty}\frac{\mathcal{E}(\beta_l(F_L))}{L^n} \leq C_l \cdot |X|,
        \end{equation}
    where $|X|$ is the volume of $X$.
    }
\end{theorem}
The constants $C_l$ in Theorem~\ref{Main Results: Global Betti} depend on the parameters in our Definition~\ref{D: Axiom 2}.
For Theorem~\ref{Main Results: Global Betti}, we do not know whether an actual limit exists for each $l = 1, \dots, n-m-1$. 
This is because of the existence of giant connected components of $Z(F_L)$, which can appear in many Gaussian ensemble examples. 
One may see relevant discussion on page 6 of~\cite{W21} and references therein.

We also believe that some of these results can be extended to general level sets, i.e., $F^{-1}(z)$ with $z \in \mathbb{R}^m$. But for simplicity, we focus on zero sets in this paper.

This introduction has thus far summarized our main results. 
Concerning our methods: very roughly speaking, they combine techniques from the literature on zero sets of random scalar-valued functions with tools from differential geometry. 
In the statistics of random zero sets, a widely-used tool is the Kac-Rice theorem, see our Theorem~\ref{Thm: Kac-Rice} for example.
The Kac-Rice theorem computes the moments of `integral random variables' of the form
    \begin{equation}
        \int_{Z(F) \cap A} h(x, W(x)) \ d \mc{H}^{n-m}(x),
    \end{equation}
where $A \subset \mb{R}^n$ is a Borel set, both $F: \mb{R}^n \to \mb{R}^m$ and $W$ are random fields, $h$ is a function with enough regularity, and $\mc{H}^{n-m}$ is the $(n-m)$-dimensional Hausdorff measure (the volume measure on the embedded submanifolds).
On the other hand, statistics like the number, Betti numbers, or isotopy classes, are related to some geometric integrals on embedded submanifolds (see our Theorem~\ref{Thm: Fenchel} and Theorem~\ref{Thm: Chern}, for example).
Now the idea is to choose suitable $h$'s with enough regularity which are compatible with those geometric integrals. 
However, the $h$'s involved in the geometric integrals always have highly singular parts. 
For example, in the simplest case when $m=1$ so that $F$ is a scalar-valued function, Theorem~\ref{Thm: Fenchel} suggests to choose $h$ as the mean curvature $H$ on $Z(F)$ and $W = (\nabla F , \nabla ^2 F)$, where $H$ has the following formula,
    \begin{equation}
        H = \frac{\Delta F}{||\nabla F||} - \frac{\nabla^2 F(\nabla F, \nabla F)}{{||\nabla F||}^3}.
    \end{equation}
Here, $\nabla F$ is the standard gradient in $\mb{R}^n$, $\nabla^2 F$ is the standard Hessian, and $\Delta_{\mb{R}^n}F$ is the standard Laplacian (see Example 1.1.3 in~\cite{mantegazza2011lecture}). 
When $m >1$, the curvature computations involved in Theorem~\ref{Thm: Fenchel} and Theorem~\ref{Thm: Chern} become more complicated, as shown in our Appendix~\ref{APP: Computation}.
We will explain in detail how we estimate these curvature involved geometric integrals in Section~\ref{Section: L^1}.

Now, we can close this introduction with a summary of the paper's organization.

In Section~\ref{Section: L^1}, we will show $L^1$-estimates for the number and Betti numbers of connected components fully contained in $C_R$, i.e., upper bounds for $\mathcal{E}(N(R;F)) / R^n$ and $\mathcal{E}(\beta_l(R;F)) / R^n$ independent of $R$, which are our Theorem~\ref{Thm: L^1 Bound on Number} and Theorem~\ref{Thm: L^1 Bound on Betti}. 
We also prove a quantitative version of Bulinskaya's Lemma~\ref{L: Quantitative Bulinskaya}.
This part contains core ideas of this paper, so we suggest reading it first before reading other sections.

In Section~\ref{Section: Stationary Random Field}, we will restate Theorem~\ref{Main Results: Local} as Theorem~\ref{Thm: Euclidean Random Field} and prove it. The proof for Theorem~\ref{Main Results: Local Betti} is similar, and we will mention necessary modifications at the end of Section~\ref{Section: Stationary Random Field}.
The ergodicity issue of translation actions and the positivity issue of $\nu_{F,c}$ related to Theorem~\ref{Main Results: Local} and Theorem~\ref{Thm: Euclidean Random Field} will be included in Appendix~\ref{APP: Ergodicity} and Appendix~\ref{APP: Positivity}.

In Section~\ref{Section: Local Double Limit and Global Limit}, we restate Theorem~\ref{Main Results: Global 1} as Theorem~\ref{Thm: Global Limit} and prove it, which also needs two technical lemmas.
One is our quantitative Bulinskaya's Lemma~\ref{L: Quantitative Bulinskaya} for random fields. Another is Lemma~\ref{L: Continuous Stability}, which is a stability lemma for connected components under small perturbations.
The proof of Theorem~\ref{Thm: Global Limit} also needs an upper bound for $\mathcal{E}({(N(R;F))}^q) / R^{nq}$ independent of $R$ with some $q = q(n,m) > 1$, which is our Theorem~\ref{Thm: L^q Bound on Number}.
The proof of Theorem~\ref{Main Results: Global 2} then follows from Theorem 3 in~\cite{NS16} once Theorem~\ref{Main Results: Global 1} is completed.
Finally, we restate Theorem~\ref{Main Results: Global Betti} as Theorem~\ref{Thm: Global Betti} and prove it.

In all of these proofs, we will use $|\cdot|$ or $||\cdot||$ to denote norms of vectors or volumes of a set when there is no ambiguity. We only use $||\cdot||_{C^{1+\beta}}$ or $||\cdot||_{C^{2}}$, i.e., $||\cdot||_{\cdot}$ with subindices, to denote norms related to smoothness.
We will also use the notations $v=(v_i)$ (or $A=(A_{ij})$) to denote vectors (or matrices) with elements $v_i$ (or $A_{ij}$). When we do computations, we will use, for example, $C(n,m)$, to denote constants depending on $n,m$, but they are not necessarily the same from line to line.

%

\section{Random Fields on Cubes}\label{Section: L^1}

We denote $C_R  \equiv {(-R,R)}^n $ as open cubes with side lengths $2R$  for $R>0$ in  $\mb{R}^n$. Let $F : C_{R+1} \to \mb{R}^{m}$, $m\in(0,n)$, be a continuous Gaussian random field, and denote that $F(x) = (f_1(x), \dots, f_{m}(x))$. Then, the $(x,y)$-covariance kernel (two-point covariance kernel),
    \begin{equation}
        K_{i_1 i_2}(x,y) = \mc{E}\big(f_{i_1}(x)f_{i_2}(y)\big), \ i_1,i_2 = 1,2,\dots, m,
    \end{equation}
is an $m \times m$ matrix.

\begin{definition}\label{D: Axiom 1}
    We say that $F:C_{R+1} \to \mb{R}^m$ satisfies \textit{ $(R; M,k_1)$-assumptions} if on $C_{R+1}$,
    \begin{itemize}
	\item[$(A1)$] $F$ is centered, i.e., $\mc{E}(f_i(x)) = 0$ for $i = 1, \dots,  m$. And the joint distribution of $F(x)$ and its Jacobian $\nabla F (x)$, i.e., $(f_1(x), \dots, f_m(x), \pa_1 f_1(x) , \pa_2 f_1(x), \dots, \pa_n f_m(x) )$, is a non-degenerate Gaussian vector with a uniform lower bound for all $x \in C_{R+1}$, i.e.,
		\begin{equation}
			\mc{E}\bigg(\bigg| \sum_{i = 1}^m\eta_i \cdot f_i (x) + \sum_{i=1} ^m \sum_{j=1} ^n \xi_{ij} \cdot \pa_j f_i (x)\bigg|^2 \bigg) \geq k_1 >0,
		\end{equation} 
	for all $x \in C_{R+1}$, and for any $\eta = (\eta_i) \in \mb{R}^m$, $\xi = (\xi_{ij}) \in \mb{R}^{nm}$, with $||\eta||^2 + ||\xi||^2 = 1$.
        \item[$(A2)$] $F$ is almost surely $C^{3-} \equiv \cap_{0 < \beta < 1} C^{2+ \beta}$ in $C_{R+1}$:
                            \begin{equation}
                                \max_{0\leq |\alpha| \leq 3} \sup_{x \in C_{R+1}} | D_x ^{\alpha}D_y ^{\alpha}K_{i_1 i_2}(x,y) |_{y=x}| \leq M, \ i_1,i_2 = 1,2,\dots, m.
                            \end{equation}

        \noindent Furthermore, we say that $F$ satisfies \textit{ $(R; M,k_1, k_2)$-assumptions} if on $C_{R+1}$, $F$ also satisfies the following two-point nondegeneracy assumption $(A3)$ for some $k_2 >0$.
        \item[$(A3)$] For every pair of distinct points $x,y  \in C_{R}$, the Gaussian vector $(F(x),F(y))$ has a nondegenerate distribution with the following restrictions:
                            \begin{equation}\label{E: Non-Degenerate Joint Distribution}
                                p_{(F(x),F(y))}(0,0) \leq 
                                \begin{cases}
                                    \frac{k_2}{||x-y||^m} ,& \ \text{if } ||x-y|| \leq 1, \\
                                    & \\
                                    k_2 ,& \ \text{if } ||x-y|| >1,
                                \end{cases}
                            \end{equation}
                    where $p_{(F(x),F(y))}(0,0)$ is the density of the joint distribution of $(F(x), F(y))$ at $(0,0)$.

    \end{itemize}
\end{definition}

\begin{remark}
    The condition $(A3)$, and the later condition $(B4)$ in Definition~\ref{D: Axiom 3},  will only be used when we prove Theorem~\ref{Thm: L^q Bound on Number}, Theorem~\ref{Thm: Global Limit}, and Theorem~\ref{Thm: Global Betti} in Section~\ref{Section: Local Double Limit and Global Limit}.
    The upper bound $\frac{k_2}{||x-y||^m}$ is natural when $x$ and $y$ are close to each other, which actually follows from the condition $(A1)$. 
    In our Appendix~\ref{APP: 2 Point}, we will explain this and show that $(A1)$ and $(A2)$ would imply (\ref{E: Non-Degenerate Joint Distribution}) locally.
\end{remark}

Our first main result in Section~\ref{Section: L^1} is the following theorem.

\begin{theorem}\label{Thm: L^1 Bound on Number}
	\textit{
    Assume that $F : C_{R+1} \to \mb{R}^{m}$, satisfies \textit{ $(R; M,k_1)$-assumptions} on $C_{R+1}$. Then, there is a constant $D_1 = D_1(n,m, M,k_1) >0$, such that
        \begin{equation}
            \mc{E}(N(R;F)) \leq D_1 \cdot |C_R|.
        \end{equation}
	}
\end{theorem}

Here, $|C_R|$ is the volume of $C_R$. Notice that Theorem~\ref{Thm: L^1 Bound on Number} also holds true for any $r \in (0,R)$ with the the same $D_1$. 

We need two theorems before we prove Theorem~\ref{Thm: L^1 Bound on Number}. The first is the following Kac-Rice theorem, see for example Remark 10 in Section 2 of~\cite{CH20} and Theorem 6.10 in Chapter 6 of~\cite{AW09}. 
\begin{theorem}[Kac-Rice Theorem with Integrand]\label{Thm: Kac-Rice}
	\textit{
    Let $W: C_{R+1} \to \mb{R}^{n_1}$ be a continuous Gaussian random field such that $(F,W)$ is also a Gaussian field on $C_{R+1}$. Assume that $h = h(x,w): \mb{R}^n \times  \mb{R}^{n_1} \to \mb{R}$ is a positive continuous bounded function.
    Under some mild assumptions on $F$, the $s$-th moments of the integral random variable
        \begin{equation}
            \int_{ \{F = 0\} \cap B} h(u,W(u)) d\mc{H}^{n-m}(u) 
        \end{equation}
    has the following formula:
        \begin{equation}\label{E: Kac-Rice}
            \begin{split}
                & \quad \mc{E} \bigg[{\bigg(\int_{ \{F = 0\} \cap B} h(u,W(u)) d\mc{H}^{n-m}(u) \bigg)}^s \bigg] 
                \\  &= \int_{B^s} J_{s,F}(h;u_1 , \dots , u_s)  p_{(F(u_1), \dots, F(u_s))} (0) \ du_1 \dots du_{s},
            \end{split}
        \end{equation}
    for every Borel measurable set $B \subset C_{R+1}$, where
        \begin{equation}
            \begin{split}
                & \quad J_{s,F}(h;u_1, \dots , u_s) 
               \\ &= \mc{E} \bigg[ \prod_{t=1} ^{s} h(u_t, W(u_t)) \sqrt{\det[{(\nabla F (u_t))}^T (\nabla F(u_t) )]} \bigg \vert F(u_t) = 0, t = 1, \dots, s \bigg],
            \end{split}
        \end{equation}
    and $p_{(F(u_1), \dots ,F(u_s))}(0)$ is the density of ${(F(u_t))}_{t=1}^s$ at $(0, \dots,0)$, 
    ${(\nabla F )}^T (\nabla F)$ is the $m \times m$ matrix ${(\nabla f_{i_1} \cdot \nabla f_{i_2})}_{m \times m}$,
    and $\mc{H}^{n-m}$ is the $(n-m)$-dimensional Hausdorff measure, i.e., surface measure for submanifolds of $\mb{R}^n$.
	}
\end{theorem}

In Theorem~\ref{Thm: Kac-Rice}, for $s = 1$, we only need to assume that $F$ satisfies the \textit{ $(R; M,k_1)$-assumptions}, while for $s = 2$, we need to assume that $F$ satisfies the \textit{ $(R; M,k_1,k_2)$-assumptions}. 
Higher moments can also be computed using the same formula as (\ref{E: Kac-Rice}) with some natural added assumptions. See~\cite{AW09,CH20} again.

Our second theorem is to link the number of connected components to a curvature integral over $Z(F)$. 
Let me cite Theorem 3 in~\cite{C71}, and readers can see more references therein. We will give more intuition of geometric inequalities when we state Theorem~\ref{Thm: L^1 Bound on Betti} and Theorem~\ref{Thm: Chern} later.

\begin{theorem}[Fenchel-Borsuk-Willmore-Chern-Lashof Inequality]\label{Thm: Fenchel}
	\textit{
    Let $M$ be a connected compact $(n-m)$-dimensional submanifold of $\mb{R}^{n}$ without boundary. Then,
    \begin{equation}\label{E: Willmore}
        |\mb{S}^{n-m}| \leq \int_{M}\bigg|\bigg|\frac{\bH}{n-m}\bigg|\bigg|^{n-m} \ d\mc{H}^{n-m},
    \end{equation} 
    where $|\mb{S}^{n-m}|$ is the volume of $(n-m)$-dimensional unit sphere and $\bH$ is the mean curvature vector of $M \subset \mb{R}^n$.
    The equality in (\ref{E: Willmore}) holds if and only if when $n-m >1$, $M$ is embedded as a hypersphere in an $(n-m + 1)$-dimensional linear subspace of $\mb{R}^n$; when $n-m=1$, $M$ is embedded as a simple convex closed curve in a $2$-dimensional linear subspace of $\mb{R}^n$.
	}
\end{theorem}

If we already know that all components of $Z(F)$ contained in $C_R$ are regular submanifolds, we can then apply this Theorem~\ref{Thm: Fenchel}, and see that there is a positive constant $C = C(m,n)$ such that
    \begin{equation}\label{E: Bound Components Number}
        \begin{split}
            N(R;F) &\leq C \cdot  \sum_{M \subset F^{-1}(0) \cap C_R, \text{disjoint} } \int_{M} ||\bH||^{n-m} \ d\mc{H}^{n-m} 
            \\     &\leq C \cdot \int_{F^{-1}(0) \cap C_R} ||\bH||^{n-m} \ d\mc{H}^{n-m} .
        \end{split}
    \end{equation}
Then, we can use Theorem~\ref{Thm: Kac-Rice} to compute the expectation of the mean curvature integral. 
Indeed, the regularity follows from Bulinskaya's lemma, see for example, Proposition 6.12 of~\cite{AW09}.
We will also prove a quantitative version of Bulinskaya's lemma in our Lemma~\ref{L: Quantitative Bulinskaya}, which was proved when $m=1$ in Lemma 7 in~\cite{NS16}. Now, we can use (\ref{E: Bound Components Number}) to compute our expectations and prove the Theorem~\ref{Thm: L^1 Bound on Number}.

\begin{proof}[Proof of Theorem~\ref{Thm: L^1 Bound on Number}]
    The general computations for $||\bH||$ are included in Appendix~\ref{APP: Computation} and Theorem~\ref{Thm: APP Curvature}. Let us recall the results and rewrite them to adapt to our current settings.
    We can write $||\bH||$ by terms and derivatives of $F$. If we set $\varphi: M \to \mb{R}^n$ as an embedding of a regular connected component $M$ of $Z(F)$,  then we define $ (g_{ij} ) \equiv (\pa_i \varphi \cdot \pa_j \varphi)$, $(g^{ij}) \equiv {(g_{ij} )}^{-1}$, $A = (A_{\alpha \beta}) \equiv (\nabla f_\alpha \cdot \nabla f_\beta) = \big({(\nabla F)}^T (\nabla F)\big)$, and $A^{\alpha \beta} \equiv {(A^{-1})}_{\alpha \beta}$,
    where $v_1 \cdot v_2$ is the standard inner product for two vectors $v_1, v_2 \in \mb{R}^n$.
    Then, we have the following inequalities ($||\cdot||$ are all vectors norms or matrices norms),
        \begin{equation}\label{E: Bound Mean Curvature}
            \begin{split}
                ||\bH||^2 &= \sum_{1 \leq \alpha ,\beta \leq m} \sum_{1 \leq i,j,k,l \leq n-m} \ (\nabla^2 f_\alpha) (\pa_i \varphi , \pa_j \varphi) g^{ij} A^{\alpha \beta} (\nabla^2 f_\beta) (\pa_k \varphi , \pa_l \varphi) g^{kl}
                \\      &\leq C(n,m)\cdot||A^{-1}|| \cdot ||\nabla^2 F || ^2 
                \\      & = C(n,m)\cdot |\det A |^{-1} \cdot ||\mr{adj} A|| \cdot ||\nabla^2 F ||^2
                \\      &\leq C(n,m)\cdot |\det A |^{-1} \cdot ||\nabla F||^{2(m-1)} \cdot ||\nabla^2 F ||^2 ,
            \end{split}
        \end{equation}
    where $\mr{adj} A$ is the adjugate matrix associated with $A$. Here, $C(n,m)$ are different positive constants but they only depend on $n,m$. 
    
    According to Theorem~\ref{Thm: Kac-Rice}, we let
        \begin{equation}
            W = (\nabla F, \nabla^2 F),
        \end{equation}
    which satisfies that $(F(u),W(u))$ is Gaussian for every $u \in C_{R+1}$. According to (\ref{E: Bound Components Number}) and (\ref{E: Bound Mean Curvature}), we set
        \begin{equation}
            h(x,W) \equiv h(W) \equiv {(C(n,m) )}^{\frac{(n-m)}{2}}\cdot \big|\det A \big|^{-\frac{(n-m)}{2}} \cdot ||\nabla F||^{(n-m)(m-1)} \cdot ||\nabla^2 F||^{(n-m)}.
        \end{equation}
    Here, we regard $W$ as a vector with elements $(\nabla F, \nabla^2 F)$ and we regard $h$ as a function of this vector $W$ and independent of $x$. 
    To be consistent with the boundedness assumption in our Theorem~\ref{Thm: Kac-Rice} (see also Remark 10 in Section 2 of~\cite{CH20} and Theorem 6.10 in~\cite{AW09}), we need to modify $h$ to make it bounded.
    We denote the closed singular set of $h$ as $\mc{S}$, which is of Hausdorff codimension $1$ in the vector space of $W$ because $\det A$ is a polynomial with elements in $\nabla F$.
    So, we can construct a monotone family of positive continuous cut-off functions $\varphi_Q(W)$, such that $\varphi_Q \to 1$ as $Q \to \infty$, and $\varphi_{Q_1}(W) < \varphi_{Q_2}(W)$ when $Q_1 < Q_2$, and $\varphi_Q(W) = 0$ when $\dist(W , \mc{S}) < 1/Q$ or $||W|| > Q$.
    We then set
        \begin{equation}
            h_Q(x, W) \equiv h_Q(W) \equiv \varphi_Q(W) h(W), \quad Q>0,
        \end{equation}
    which is a monotone family of positive continuous bounded functions.
    
    We can then see that, by monotone convergence theorem,
        \begin{equation}\label{E: Kac-Rice in Curvature}
            \begin{split}
                &\mc{E} \bigg(\int_{F^{-1}(0) \cap C_R}  (||\bH (u) ||^{(n-m)})   \ d\mc{H}^{n-m}(u) \bigg ) \leq \mc{E} \bigg( \int_{F^{-1}(0) \cap C_R} h(u,W(u))\ d\mc{H}^{n-m}(u) \bigg)
                \\      &= \mc{E} \bigg( \lim_{Q \to \infty} \int_{F^{-1}(0) \cap C_R} h_Q(u,W(u))\ d\mc{H}^{n-m}(u) \bigg)
                \\      &=  \lim_{Q \to \infty} \mc{E} \bigg( \int_{F^{-1}(0) \cap C_R} h_Q(u,W(u))\ d\mc{H}^{n-m}(u) \bigg)
                \\      &=   \lim_{Q \to \infty} \int_{C_R} J_{1,F}(h_Q;u ) \cdot  p_{F(u)} (0) \ du 
                \\      &=  \lim_{Q \to \infty} \int_{C_R} \mc{E} \big[h_Q(u,W(u)) \cdot | \det A |^{1/2}  \ \big| \ F(u) = 0\big] \cdot p_{F(u)} (0) \ du
                \\      &=   \int_{C_R} \mc{E} \big[h(u,W(u)) \cdot | \det A |^{1/2}  \ \big| \ F(u) = 0\big] \cdot p_{F(u)} (0) \ du .
            \end{split}
        \end{equation}
    Since $F$ satisfies $(A1)$ in Definition~\ref{D: Axiom 1}, $p_{F(u)}(0) \leq C(n,m,k_1)$ for any $u \in C_R$. To estimate the conditional expectation, we first notice that there are terms including $\nabla F$ and $\nabla^2 F$, which we are going to separate
    using the inequality $ab \leq \frac{k}{k+1} a^{\frac{k}{k+1}} + \frac{1}{k+1} b^{k+1} \leq a^{\frac{k}{k+1}} + b^{k+1}$ for some positive integer $k = k(m,n)$ to be determined later. The conditional expectation term in (\ref{E: Kac-Rice in Curvature}) then becomes
        \begin{equation}
            \begin{split}
                & \ \quad \mc{E}\big[| \det A (u)|^{\frac{1-(n-m) }{2}} \cdot ||\nabla F(u) ||^{(n-m)(m-1)} \cdot ||\nabla^2 F(u) ||^{(n-m)}  \ \big| \ F(u) = 0\big]
                \\  &\leq \mc{E}\bigg[ {\bigg(| \det A (u)|^{\frac{1-(n-m) }{2}} \cdot ||\nabla F(u) ||^{(n-m)(m-1)} \bigg)}^{\frac{k+1}{k}} + ||\nabla^2 F(u) ||^{(n-m)(k+1)}  \ \bigg| \ F(u) = 0\bigg ].
            \end{split}
        \end{equation}
    
    We now need to estimate
        \begin{equation}\label{E: Determinant Integral in Curvature}
            \mc{E}\bigg[ {\bigg(| \det A (u)|^{\frac{1-(n-m)}{2}} \cdot ||\nabla F(u) ||^{(n-m)(m-1)} \bigg)}^{\frac{k+1}{k}} \ \bigg| \ F(u) = 0 \bigg ]
        \end{equation}
    and 
        \begin{equation}\label{E: Hessian Term in Curvature}
            \mc{E}\bigg[ ||\nabla^2 F(u) ||^{(n-m)(k+1)}  \ \bigg| \ F(u) = 0\bigg ]
        \end{equation}
    separately.

    We estimate (\ref{E: Hessian Term in Curvature}) first. For each $u \in C_{R}$, and each $\pa^2 _{j_1j_2} f_i (u)$, one needs to find $a^{t}_{j_1 j_2 , i}(u)$ such that for each $s = 1 , \dots m$,
        \begin{equation}\label{E: Independence for Conditional Expectation}
            \mc{E}\big[ \big(\pa^2 _{j_1j_2}f_i(u) + \sum_{t=1}^{m}a^t _{j_1j_2,i} f_t(u) \big) \cdot f_s(u)\big] = 0.
        \end{equation}
    Since $F$ satisfies $(A1)$ in Definition~\ref{D: Axiom 1}, the matrix ${(\mc{E}(f_t(u) \cdot f_s(u)))}_{m \times m}$ is always invertible and has uniformly lowerly bounded determinant value for all $u \in C_R$. Hence, those $a^t _{j_1j_2,i}$ are solvable and there is a constant $C = C(n,m,M,k_1)$ such that 
        \begin{equation}
            |a^t _{j_1j_2,i}| \leq C.
        \end{equation}
    So, there is another constant $C =C(n,m,M,k_1, k) = C(n,m,M,k_1)$ (since $k$ only depend on $n,m$ and will be determined later) such that 
        \begin{equation}
            \begin{split}
                & \quad \mc{E}\big[ |\pa^2 _{j_1j_2} f_i(u) |^{(n-m)(k+1)}  \ \big| \ F(u) = 0\big ]
                \\ & = \mc{E}\big[ |\pa^2 _{j_1j_2} f_i(u) + \sum_{t=1}^{m}a^t _{j_1j_2,i} f_t(u)|^{(n-m)(k+1)}  \big ] \leq C.
            \end{split}
        \end{equation}
    
    The estimate for (\ref{E: Determinant Integral in Curvature}) requires more ingredients. First, notice that since $(F,\nabla F)$ satisfies $(A1)$ in Definition~\ref{D: Axiom 1}. Then, for each $u \in C_{R}$ and each $\pa_j f_i(u)$, we need to find $a^t _{j,i}(u)$ such that for each $s = 1, \dots,m$,
	\begin{equation}
		\mc{E}\big[ \big(\pa _{j}f_i(u) + \sum_{t=1}^{m}a^t _{j,i}(u) f_t(u) \big) \cdot f_s(u)\big] = 0.
	\end{equation}
    For similar reasons in estimating those $a^t _{j_1 j_2,i}$ in (\ref{E: Independence for Conditional Expectation}), $a^t _{j,i}(u)$ are also solvable and there is a constant $C = C(n,m,M,k_1)$ such that
	\begin{equation}
		|a^t _{j,i}| \leq C.
	\end{equation}

    Next, we replace each $\pa_j f_i(u)$ in (\ref{E: Determinant Integral in Curvature}) with ${(\pa_j f_i)}^\# \equiv \pa _{j}f_i(u) + \sum_{t=1}^{m}a^t _{j,i}(u) f_t(u)$, and also, define ${(\nabla F (u))}^\#$ as replacing elements $\pa_j f_i$ in $\nabla F$ with ${(\pa_j f_i)}^\#$. Since the joint distribution of $(F,\nabla F)$ is non-degenerate, the Gaussian vector ${(\nabla F (u))}^\#$ is also non-degenerate.

    We can divide the non-degenerate $nm$-dimensional Gaussian vector ${(\nabla F (u))}^\#$ into $m$ $n$-dimensional vectors:
        \begin{equation}
            {(\nabla F (u))}^\# = (v_1(u) , v_2(u) ,\dots, v_m(u)),
        \end{equation}
    where for each $t = 1, \dots ,m$,
        \begin{equation}
            v_t(u) ={( {(\nabla f_t (u))}^\# )}^T= ({(\pa_1 f_t(u))}^\#, {(\pa_2 f_t(u))}^\# , \dots , {(\pa_n f_t(u))}^\# ).
        \end{equation}
    With these notations, we also set ${(A(u))}^\# = {(v_{t_1} \cdot v_{t_2})}_{m \times m}$, and we see that (\ref{E: Determinant Integral in Curvature}) equals to 
	\begin{equation}\label{E: Determinant Integral in Curvature, No Condition}
	    \mc{E}\bigg[ {\bigg({| \det {(A (u))}^\#|}^{\frac{1-(n-m)}{2}} \cdot {||{(\nabla F(u))}^\# ||}^{(n-m)(m-1)} \bigg)}^{\frac{k+1}{k}}  \bigg ].
	\end{equation}
    Notice that the determinant of the covariance kernel of ${(\nabla F(u))}^\#$ is lowerly bounded by a constant $C = C(n,m,M,k_1)$ as in $(A1)$ and $(A2)$ in Definition~\ref{D: Axiom 1}, the (\ref{E: Determinant Integral in Curvature, No Condition}) then reduces to estimating the following integral: 
        \begin{equation}\label{E: Determinant Integral Origin}
            \begin{split}
                C(n,m,M, k_1) \bigg(\int_{{(\mb{R}^n)}^{\otimes m}} & {|\det(v_{t_1} \cdot v_{t_2})|}^{\alpha} \cdot {|{|v_1|} ^2 + \cdots + {|v_m|}^2|}^\beta 
                \\ &\cdot e^{-c(n,m,M,k_1)(|v_1|^2 + \cdots |v_m|^2)}\ dv_1 \cdots d v_m \bigg).
            \end{split}
        \end{equation}
    Here, $\alpha = \frac{k+1}{k} \cdot \frac{1- (n-m) }{2}$, which is very close to but less than $- \frac{n-m-1}{2}$ since $k$ will be chosen to be large enough, and $\beta = \frac{k+1}{2k} \cdot (n-m)(m-1) $, and we use $|\cdot|$ to denote vector norms temporarily. 
    Also, $c(n,m,M,k_1)$ is a positive constant related to the covariance kernel for ${(\nabla F(u))}^\#$.
    We can simplify the integral in (\ref{E: Determinant Integral Origin}) further and then get
        \begin{equation}\label{E: Determinant Integral Spherical Decomposition}
            \begin{split}
                & \quad \bigg(\int_{{(\mb{S}^{n-1})}^{\otimes m}} |\det(w_{t_1} \cdot w_{t_2})| ^\alpha \ d\sigma(w_1) \cdots d\sigma(w_m) \bigg) 
                \\  &\cdot \bigg(\int_{{(\mb{R}_+)}^{\otimes m}} |r_1|^{2\alpha+n-1}  \cdots |r_m|^{2\alpha+n-1}||r_1| ^2 + \cdots + |r_m|^2|^\beta 
                \\  & \quad \cdot e^{-c(n,m,M,k_1)(|r_1|^2 + \cdots |r_m|^2)}\ dr_1 \cdots d r_m \bigg) ,
            \end{split}
        \end{equation}
    where we use spherical coordinates, and $\sigma(\cdot)$ is the surface measure on the unit sphere $\mb{S}^{n-1}$. Notice that $0 < 2\alpha + n-1 < n-1$ since $m \geq 1$ and $\alpha = - \frac{n-m-1}{2} - \epsilon$ for some small $\epsilon = \epsilon(m,n)>0$ to be determined later. Hence, the second multiplier is bounded by a constant $C = C(n,m,M,k_1)$.
    The rest is to consider the singular integral 
        \begin{equation}\label{E: Determinant Integral 1}
            \int_{{(\mb{S}^{n-1})}^{\otimes m}} |\det(w_{t_1} \cdot w_{t_2})| ^\alpha \ d\sigma(w_1) \cdots d\sigma(w_m)
        \end{equation}
    for some $\alpha <0$.

    Currently, our $\alpha$ is close to $-\frac{n-m-1}{2}$ but not less than $-\frac{n-m}{2}$. Indeed, we can prove a stronger statement and the sharp $\alpha $ is actually $-\frac{n-m+1}{2}$! Hence, one can just choose any $\epsilon = \epsilon(m,n) \in (0,1/2)$, so that $k$ is well-chosen in the previous settings and $k$ only depends on $m,n$. 
    This sharp value of $\alpha$ will be used when we prove the quantitative Bulinskaya's lemma, our Lemma~\ref{L: Quantitative Bulinskaya}.

    \begin{lemma}\label{L: Singular Integral}
	\textit{
        Let $n > m >0$. For any $s_0 \in (0,1)$ and $\alpha \in ( - \frac{n-m+1}{2} s_0 , 0 )$, there is a positive constant $C = C(s_0,n,m) $ such that 
        \begin{equation}\label{E: Determinant Integral in Lemma}
            \int_{{(\mb{S}^{n-1})}^{\otimes m}} |\det(w_{t_1} \cdot w_{t_2})| ^\alpha \ d\sigma(w_1) \cdots d\sigma(w_m) \leq C.
        \end{equation}
	}
    \end{lemma}

    \begin{proof}[Proof of Lemma~\ref{L: Singular Integral}]
        If $m = 1$, then there is nothing to prove. For $m \geq 2$, we fix $\{e_1 , \dots, e_n\}$ as the standard orthonormal frame of $\mb{R}^n$. Notice that $| \det ( w_{t_1} \cdot w_{t_2}) | $ is invariant under actions of $O(n)$. Hence, the left hand side of (\ref{E: Determinant Integral in Lemma}) equals to 
        \begin{equation}\label{E: Determinant Integral 2}
            |\mb{S}^{n-1}| \cdot \int_{{(\mb{S}^{n-1})}^{\otimes (m-1)}} |\det(A(e_1))| ^\alpha \ d\sigma(w_2) \cdots d\sigma(w_m),
        \end{equation}
    where $A(e_1)$ has the form
        \begin{equation}
                A(e_1)=
                  \begin{bmatrix}
                    1 & e_1 \cdot w_2 & \cdots & e_1 \cdot w_m \\
                    w_2 \cdot e_1 & 1 &\cdots & w_2 \cdot w_m  \\
                    \vdots & \vdots & \cdots & \vdots \\
                    w_m \cdot e_1 & w_m \cdot w_2 & \cdots & 1
                  \end{bmatrix}     .
        \end{equation}
    By elementary row operations, we see that $A(e_1)$ has the same determinant value as the matrix 
        \begin{equation}
            A'(e_1)=
            \begin{bmatrix}
                1 & e_1 \cdot w_2 & \cdots & e_1 \cdot w_m \\
                0 &   & &   \\
                \vdots &  &{\bigg[w_{t_1} \cdot w_{t_2} - (e_1 \cdot w_{t_1})(e_1 \cdot w_{t_2})\bigg]} _{ ( m-1) \times (m-1)}&  \\
                0 &  &  & 
            \end{bmatrix}     .
        \end{equation}
    Then, we can parametrize $\mb{S}^{n-1}$ by $B^{n-1}$, the unit disc of dimension $n-1$. That is, for each $t = 2, \dots,m$, almost all $w_t \in \mb{S}^{n-1}$ can be written in coordinates of the form $(\sqrt{1- |w_t '|^2}, \ w_t ')$ or $(- \sqrt{1- |w_t'|^2}, \ w_t')$ for 
    $w_t ' = (w_t(2), \dots, w_t(n) ) \in B^{n-1}$. So, $w_{t_1} \cdot w_{t_2} - (e_1 \cdot w_{t_1})(e_1 \cdot w_{t_2})  = \langle w_{t_1} ' , w_{t_2} ' \rangle_{\mb{R}^{n-1} }$, where $\langle \cdot , \cdot \rangle_{\mb{R}^{n-1} }$ is the standard inner product in $\mb{R}^{n-1}$.
    By the area formula, we can rewrite (\ref{E: Determinant Integral 2}) as 
        \begin{equation}\label{E: Induction on Singular Integral}
            \begin{split}
                & \quad |\mb{S}^{n-1}| \cdot 2^{m-1 } \cdot \int_{{(B^{n-1})}^{\otimes (m-1)}} |\det(\langle w_{t_1} ' , w_{t_2} ' \rangle_{\mb{R}^{n-1}})| ^\alpha \cdot \prod_{t=2} ^m  \frac{1}{\sqrt{1- |w_t '|^2}}   \ dw_2 ' \cdots dw_m '
                \\  &= |\mb{S}^{n-1}| \cdot 2^{m-1 } \cdot \int_{{(\mb{S}^{n-2})}^{\otimes (m-1)}} |\det(\langle \tilde{w}_{t_1} , \tilde{w}_{t_2}  \rangle_{\mb{R}^{n-1}})| ^\alpha \ d\tilde{\sigma}{(\tilde{w}_2)}  \cdots d\tilde{\sigma}{(\tilde{w}_m)} 
                \\  &\qquad \cdot \int_{{(0,1)}^{\otimes (m-1)}} \prod_{t=2} ^m  \frac{r_t ^{2 \alpha + n -2}}{\sqrt{1- |r_t |^2}}   \ dr_2  \cdots dr_m .
            \end{split}
        \end{equation}
    Here $\tilde{\sigma}(\cdot)$ is the surface measure of $\mb{S}^{n-2}$. Notice that since $n > m \geq 2$, we have that $-(n-m+1) s_0 + n-2 \geq -(n-1)s_0 + n-2 = (n-1)(1-s_0) - 1 \geq 2(1-s_0) -1 = 1- 2s_0$. Hence,
        \begin{equation}
            \int_0 ^1 r^{2\alpha + n-2} \ dr \leq \int_0 ^1 r^{ -(n-m+1) s_0 + n-2} \ dr \leq \frac{1}{2(1-s_0)} \ .
        \end{equation}
    If $m=2$, we finish our proof because in (\ref{E: Induction on Singular Integral}), the determinant value in the integral is just $1$. If $m >2$, then we run the same process for the pair $n-2, m-1$ with the condition that $ n-2 \geq m -1 \geq 2$. 
    Continue doing this process, we can finish the proof for Lemma~\ref{L: Singular Integral}.
    
    \end{proof}

    With this Lemma~\ref{L: Singular Integral}, we can bound (\ref{E: Determinant Integral in Curvature}) by a constant $C = C(n,m,M,k_1)>0 $. Combine this constant with the constant $C = C(n,m,M,k_1)>0$ bounding (\ref{E: Hessian Term in Curvature}), we can use (\ref{E: Bound Components Number}) and finish our proof for Theorem~\ref{Thm: L^1 Bound on Number}.

\end{proof}

\begin{remark}\label{Rmk: Fary-Milnor}
    The sharp constant in Theorem~\ref{Thm: Fenchel} does not influence our proofs for Theorem~\ref{Thm: L^1 Bound on Number} as we only need an upper bound. On the other hand, one should not expect that we can calculate an accurate expectation for the number of connected components in such a way.
    For example, when $n=3$ and $m = 2$, $Z(F)$ consists of many closed curves, which we call knots. The F\'{a}ry-Milnor's theorem says that if a closed curve is not an unknot, then the integral on the right hand side of (\ref{E: Willmore}) is strictly greater than $4\pi$, while in this case, $|\mb{S}^1| = 2\pi$. 
    By our later Theorem~\ref{Thm: Euclidean Random Field} (also stated as Theorem~\ref{Main Results: Local} previously), all types of knots will appear with positive probabilities in limiting cases under some assumptions, which include the complex arithmetic random waves and Kostlan's ensemble.
    So, even an accurate calculation of curvature integrals will not tell us the accurate value of the limiting expectation, $\nu_F$ in Theorem~\ref{Thm: Euclidean Random Field}, of the number of connected components. 
\end{remark}

Next, for Betti numbers, we can build up similar upper bounds as in Theorem~\ref{Thm: L^1 Bound on Number}.
\begin{theorem}\label{Thm: L^1 Bound on Betti}
	\textit{
    Assume that $F : C_{R+1} \to \mb{R}^{m}$, satisfies \textit{ $(R; M,k_1)$-assumptions} on $C_{R+1}$. Then, there is a constant $\widetilde{D}_1 = \widetilde{D}_1(n,m, M,k_1) >0$ such that
        \begin{equation}
            \sum_{l=0}^{n-m} \mc{E}(\beta_l(R;F)) \leq \widetilde{D}_1 \cdot |C_R|.
        \end{equation}
	}
\end{theorem}
The proof is the same as Theorem~\ref{Thm: L^1 Bound on Number} once we replace Theorem~\ref{Thm: Fenchel} with the Chern-Lashof inequalities in~\cite{CL57,CL58}, which connects the topology of a closed manifold to its extrinsic curvatures.
On the other hand, there are some other works and techniques in differential geometry that lead to similar topological controls from the study of well-known vanishing theorems. See, for example,~\cite{B88} and references therein, where a De Giorgi-Nash-Moser iteration scheme also leads to a control on Betti numbers.
But the forms shown in~\cite{CL57, CL58} are cleaner and well adapted to our applications here. The author also proved them independently in earlier researches in differential geometry, but then learned that these things were studied well by very great geometers many years ago.
The following inequality is Theorem 1 in~\cite{CL58} but we rewrite it a little to be adapted to our notations.
\begin{theorem}[Chern-Lashof Inequality]\label{Thm: Chern}
	\textit{
    If $M$ is a connected closed $(n-m)$-dimensional submanifold of $\mb{R}^{n}$ without boundary, then,
    \begin{equation}\label{E: Chern-Lashof}
        \sum_{l=0}^{n-m} \beta_l(M) \leq \frac{1}{|\mb{S}^{n-1}|} \cdot \int_{M} \int_{T_x ^\perp M \cap \{|y| = 1\}} \big| \det \big(\langle \II(x) , y  \rangle_{\mb{R}^n}\big) \big| d\sigma^{m-1}(y)\ d\mc{H}^{n-m}(x).
    \end{equation} 
	}
\end{theorem}
In (\ref{E: Chern-Lashof}), $|\mb{S}^{n-1}|$ is the volume of $(n-1)$-dimensional unit sphere, $\beta_l(M) $ is the $l$-th Betti number of $M$ over $\mb{R}$, $\II(x)$ is the second fundamental form of $M \subset \mb{R}^n$ at $x$,
and $y$ ranges over unit vectors in $\mb{R}^n$ that are on the fiber $T_x ^\perp M $ of normal bundle $T ^\perp M $, $\langle \cdot , \cdot \rangle_{\mb{R}^n}$ is the Euclidean inner product in $\mb{R}^n$, and $\sigma^{m-1}$ is the standard surface measure on $\mb{S}^{m-1} \subset T_x ^\perp M \subset \mb{R}^n$.

We can see that $\big| \det \big(\langle \II(x) , y  \rangle\big) \big| \leq C(n,m) \cdot ||\II(x)||^{n-m}$. To estimate the norm of $\II(x)$, we can also get from Theorem~\ref{Thm: APP Curvature} that, 
	\begin{equation}\label{E: Bound Second Fundamental Form}
		\begin{split}
			||\II||^2   &\leq C(n,m)\cdot||A^{-1}|| \cdot ||\nabla^2 F || ^2 
                \\      & = C(n,m)\cdot |\det A |^{-1} \cdot ||\mr{adj} A|| \cdot ||\nabla^2 F ||^2
                \\      &\leq C(n,m)\cdot |\det A |^{-1} \cdot ||\nabla F||^{2(m-1)} \cdot ||\nabla^2 F ||^2 ,
		\end{split}
	\end{equation}
for some positive constants $C(n,m)$. We can use the same $h$ and $h_Q$ as shown in the proof of Theorem~\ref{Thm: L^1 Bound on Number} and then obtain the proof of Theorem~\ref{Thm: L^1 Bound on Betti}.

Before finishing this section, we prove a quantitative Bulinskaya's lemma for random fields, where the proof will use the optimal degree $\alpha$ appeared in Lemma~\ref{L: Singular Integral}. 
We will prove it for Gaussian random fields for simplicity, although the result and the proof actually do not depend on Gaussian random variables essentially. 
This lemma is one of the key factors that will be used when we prove theorems in Section~\ref{Section: Local Double Limit and Global Limit}.

Assume that our Gaussian random field $F : C_{R+1} \to \mb{R}^{m}$ satisfies \textit{ $(R; M,k_1)$-assumptions} on $C_{R+1}$. But $C^{2-}$-smoothness for $F$ is already enough for our proof here. 
First, let $||\cdot||$ denote the norm of a vector and we define 
    \begin{equation}\label{E: Definition of Smallest Eigenvalue}
        \lambda(F(x)) \equiv \min_{w \in \mb{S}^{m-1}} ||\nabla (F \cdot w) ||(x).
    \end{equation}
So, ${\lambda(F(x))} ^2$ is the smallest eigenvalue of the matrix $A = {(\nabla f_{i_1} \cdot \nabla f_{i_2})}_{m \times m}$ at $x$. This matrix $A$ appeared in the proof of Theorem~\ref{Thm: L^1 Bound on Number}. 
Then, we can obtain the following lemma.
\begin{lemma}\label{L: Quantitative Bulinskaya}
	\textit{
    Given $\delta >0$, there exists a $\tau = \tau(\delta, R,n,m,M,k_1) \in (0,1/2)$, such that 
        \begin{equation}
            \mc{P} \big( \min_{x \in \bar{C}_R} \max\{||F(x)||,\lambda(F(x)) \} < \tau \big)  < \delta
        \end{equation}
	}
\end{lemma}
\begin{proof}
    Denote $\Omega_\tau$ as the event that 
        \begin{equation}
            \{ \exists\ y \in \bar{C}_R, \text{ s.t. } ||F(y)|| < \tau, \ \lambda(F(y)) < \tau  \}.
        \end{equation}
    We put $W = 1 + ||F||_{C^{1+ \beta}(C_{R+1/2})}$ with $\beta \in (0,1)$ to be specified later. Here $||F||_{C^{1+ \beta}(C_{R+1/2})}$ is the $C^{1 + \beta}$-norm for $F$ on the cube $C_{R+1/2}$. Then, if $\Omega_\tau$ occurs for a $y \in \bar{C}_R$, we assume that $||\nabla( F \cdot w_y) || (y) < \tau$ for a corresponding fixed $w_y \in \mb{S}^{m-1}$.  Then, for $x$ in the ball $B(y,\tau)$,
        \begin{equation}
            ||F(x)|| \leq \tau + \tau \cdot ||F||_{C^{1+ \beta}(C_{R+1/2})} =  W \tau,
        \end{equation}
    and 
        \begin{equation}
            \lambda(F(x)) \leq ||\nabla (F \cdot w_y) ||(x) \leq \tau + \tau^{\beta} \cdot ||F||_{C^{1+ \beta}(C_{R+1/2})} \leq W \tau^{\beta}.
        \end{equation}
    Since each element of the matrix ${(\nabla F (x))}^T (\nabla F(x)) = {(\nabla f_{i_1} \cdot \nabla f_{i_2})}_{m \times m}$ has an upper bound $W^2$, the largest eigenvalue of it is also bounded by $C(n,m) \cdot W^2$. We can further get that 
        \begin{equation}
            \big|\det \big({(\nabla F (x))}^T (\nabla F(x))\big) \big| \leq C(n,m) \cdot {(\lambda(F(x)))}^2 \cdot W^{2(m-1)} \leq C(n,m) \cdot W^{2m} \tau^{2\beta} .
        \end{equation}
    Let $\alpha_1 = m + (n-m+1) m$, $\alpha_2 = m + (n-m+1)\beta$. Then, for some $\eta \in (0,1)$, we define a function $\Phi_\eta(x)$ for $x \in C_{R+1}$ and get that
        \begin{equation}
            \Phi_{\eta}(x) \equiv ||F(x)||^{-m \eta } \cdot \big|\det \big({(\nabla F (x))}^T (\nabla F(x))\big) \big| ^{-\frac{n-m+1}{2} \eta} \geq C(n,m) \cdot W^{-\alpha_1 \eta } \cdot \tau^{-\alpha_2 \eta},
        \end{equation}
    where the inequality is for $x \in B(y,\tau)$.
    So, whenever $\Omega_\tau$ happens, we have that
        \begin{equation}
            \int_{C_{R+1}} \Phi_{\eta}(x) \geq \int_{B(y,\tau)} \Phi_{\eta}(x) \geq C(n,m) \cdot W^{-\alpha_1 \eta } \cdot \tau^{n-\alpha_2 \eta} . 
        \end{equation}
    Notice that $n- \alpha_2 \eta $ = $n(1- \eta \beta) - m \eta (1-\beta) - \beta \eta$. Hence, we can choose $\beta = \beta(n,m) \in (0,1)$, $\eta = \eta(n,m) \in (0,1)$, which are close to $1$, such that $n - \alpha_2 \eta < - \epsilon$ for some small positive $\epsilon = \epsilon(n,m) \in (0,1/2)$.
    Hence,
        \begin{equation}
            \begin{split}
                \mc{P}(\Omega_\tau) &\leq   C(n,m) \cdot \tau^{\alpha_2 \eta - n} \cdot \mc{E} \bigg[W^{\alpha_1 \eta} \cdot \int_{C_{R+1}} \Phi_{\eta}(x) \bigg]
                \\  &\leq C(n,m) \cdot \tau^{\epsilon} \cdot {\bigg(\mc{E} \big(W^{p' \alpha_1 \eta} \big) \bigg)}^{\frac{1}{p'}} \cdot {\bigg( \mc{E}  {\bigg(\int_{C_{R+1}} \Phi_{\eta}(x)  \bigg)}^p\bigg)}^{\frac{1}{p}}
                \\  &\leq C(n,m) \cdot \tau^{\epsilon} \cdot |C_{R+1}|^{1-\frac{1}{p}} \cdot {\bigg(\mc{E} \big(W^{p' \alpha_1 \eta} \big) \bigg)}^{\frac{1}{p'}} \cdot {\bigg( \int_{C_{R+1}} \mc{E}(\Phi_{\eta}^p(x))  \bigg)}^{\frac{1}{p}}
            \end{split}
        \end{equation}
    for some $p' = p'(n,m)> 1$, $p = p(n,m) >1$, with $\frac{1}{p'} + \frac{1}{p} = 1$ and $p \cdot \eta <1$. 
    
    For the term $\mc{E} \big(W^{p' \alpha_1 \eta} \big)$, an upper bound depending on $n,m,M,k_1$ follows from the Kolmogorov's theorem, see for example Appendix A.9 and Appendix A.11 in~\cite{NS16}.
    Next, since for each $x \in C_{R+1}$, the joint distribution of $(F(x),\nabla F(x))$ satisfies $(A1)$ in Definition~\ref{D: Axiom 1}, then similar to (\ref{E: Determinant Integral Origin}), we can write
        \begin{equation}
            \begin{split}
		\mc{E}(\Phi_{\eta}^p(x)) &= \mc{E} \bigg(||F(x)||^{-m \eta p} \cdot {\big|\det {\big({(\nabla F (x))}^T (\nabla F(x))\big)} \big|} ^{-\frac{n-m+1}{2} \eta p}  \bigg),
		\\	&\leq C(n,m,M, k_1) \cdot \int_{\mb{R}^m} ||v_0||^{-m\eta p} \cdot e^{-c(n,m,M,k_1)(||v_0||^2)} \ dv_0
		\\	& \cdot	\int_{{(\mb{R}^n)}^{\otimes m}} {|\det(v_{t_1} \cdot v_{t_2})|}^{-\frac{n-m+1}{2}\eta p} \cdot e^{-c(n,m,M,k_1)({||v_1||}^2 + \cdots {||v_m||}^2)}\ dv_1 \cdots d v_m,
	    \end{split}
        \end{equation}
    where we use $v_0 $ to denote the components of $F(x) = (f_1(x) ,\dots, f_m(x))$, and $(v_1, \dots, v_m)$ to denote components of $\nabla F(x) = ({(\nabla f_1 (x))}^T, \dots , {(\nabla f_m (x))}^T)$, and $(v_{t_1} \cdot v_{t_2})$ to denote the matrix with elements $v_{t_1} \cdot v_{t_2}$.
    Since $\eta p < 1$ by our choice, we can bound the above two terms by some constants $C = C(n,m,M,k_1) > 0$ following the proof of Theorem~\ref{Thm: L^1 Bound on Number}. See the process proving (\ref{E: Determinant Integral in Curvature}), (\ref{E: Determinant Integral Spherical Decomposition}) and Lemma~\ref{L: Singular Integral}.

    Since $\epsilon = \epsilon(n,m) >0$, we can finish the proof for Lemma~\ref{L: Quantitative Bulinskaya}.
    
\end{proof}



\section{Proof of Theorem~\ref{Main Results: Local} and Theorem~\ref{Main Results: Local Betti}}\label{Section: Stationary Random Field}

With theorems exhibited in Section~\ref{Section: L^1}, we are now ready to prove Theorem~\ref{Main Results: Local} and Theorem~\ref{Main Results: Local Betti}.
In this section, we assume that $F:\mb{R}^n \to \mb{R}^{m}$, $m \in(0,n)$, is a centered stationary Gaussian random field satisfying $(A1)$ and $(A2)$ in Definition~\ref{D: Axiom 1} at the origin $x=0$, which is equivalent to saying that $F$ satisfies the \textit{ $(R; M,k_1)$-assumptions} on at the origin $x = 0$.
Also, since $F$ is stationary, these assumptions hold true at any point $x \in \mb{R}^n$.

We will first focus on Theorem~\ref{Main Results: Local} and explain the necessary modifications for proving Theorem~\ref{Main Results: Local Betti} after the proof of Theorem~\ref{Main Results: Local}. We first explain some definitions in Theorem~\ref{Main Results: Local} and give some basic settings similar to the $m=1$ cases in~\cite{NS16,SW19,W21}.
We denote $\mfk{B}(C^1 (\mb{R}^n ,\mb{R}^{m}))$ as the Borel $\sigma$-algebra generated by open sets in $C^1 (\mb{R}^n ,\mb{R}^{m})$, and denote $\gamma_F = F_* (\mc{P})$ as the pushforward probability measure of $\mc{P}$, where $\mc{P}$ is the probability measure on the background probability space $(\Omega, \mfk{S}, \mc{P})$.
By Bulinskaya's lemma (see either Proposition 6.12 of~\cite{AW09} or our Lemma~\ref{L: Quantitative Bulinskaya}), we see that we only need to consider a Borel subset of $C^1(\mb{R}^n,\mb{R}^m)$:
    \begin{equation}
        C^1_* (\mb{R}^n ,\mb{R}^{m}) \equiv \{G \in C^1(\mb{R}^n,\mb{R}^m) \ | \ \nabla G \text{ is of full rank on } Z(G)\},
    \end{equation}
since $\gamma_F(C^1 (\mb{R}^n ,\mb{R}^{m})  \backslash C^1_* (\mb{R}^n ,\mb{R}^{m}) ) = 0$.

The action of $\mb{R}^n$ on $(C^1_* (\mb{R}^n ,\mb{R}^{m}), \mfk{B}(C^1_* (\mb{R}^n ,\mb{R}^{m})), \gamma_F )$ is by shifts $\tau_v: G(x) \mapsto G(x + v)$, which is measure preserving since $F$ is stationary.
We say that the action of $\mb{R}^n$ is \textit{ ergodic} if, for every set $A \subset  \mfk{B}(C^1_* (\mb{R}^n ,\mb{R}^{m}))$ that satisfies $\gamma_F((\tau_v)A \Delta A) =0$ for all $v \in \mb{R}^n$, then either $\gamma_F(A) = 0$ or $\gamma_F(A ) = 1$. Here, $\Delta$ is the symmetric difference.

We restate Theorem~\ref{Main Results: Local} in the following way.

\begin{theorem}\label{Thm: Euclidean Random Field}
	\textit{
    Assume that $F:\mb{R}^n \to \mb{R}^{m}$, $m \in(0,n)$, is a centered stationary Gaussian random field satisfying the \textit{ $(R; M,k_1)$-assumptions} at the origin $x=0$.
    \begin{itemize}
        \item[(1)] There exists a number $\nu = \nu_F \geq 0$ so that
                    \begin{equation}\label{E: Growth Degree Constant}
                        \mc{E}(N(R;F)) = \nu \cdot |C_R| + o_{R \to \infty} (R^n).
                    \end{equation} 
        \item[(2)] If the translation action of $\mb{R}^n$ on $(C^1_* (\mb{R}^n ,\mb{R}^{m}), \mfk{B}(C^1_* (\mb{R}^n ,\mb{R}^{m})), \gamma_F )$ is ergodic, then
                    \begin{equation}\label{E: Ergodicity Convergence}
                        \lim_{R \to \infty}\frac{N(R;F)}{|C_R|} = \nu \ \text{almost surely} \quad \text{and} \quad \lim_{R  \to \infty}\mc{E} \bigg| \frac{N(R;F)}{|C_R|} - \nu \bigg| = 0 .
                    \end{equation}
    \end{itemize}
	Moreover, for any $c \in \mc{C}(n-m)$, (1) and (2) hold true if one replaces $N(R;F)$ with $N(R;F,c)$ and replaces the number $\nu_F$ with a $\nu_{F,c} \geq 0$.
    \begin{itemize}
	\item[(3)]
	    \begin{equation}
		\nu_F = \sum_{c \in \mc{C}(n-m)} \nu_{F,c} .
	    \end{equation} 
    \end{itemize}
	}
\end{theorem}
\begin{remark}
    When $m=1$, there are two assumptions in~\cite{NS16,SW19, W21} about when the shifting action of $\mb{R}^n$  is ergodic and when $\nu_F>0$. 
    For general $m>1$, one may need more ingredients. In particular, when the random field $F$ consists of independent random functions, i.e., $F = (f_1, \dots,f_m)$ and $f_{i_1}$ is independent of $f_{i_2}$ when $i_1 \neq i_2$, we can denote the spectral measure of $f_i$ as $\rho_i$, and assume that $\rho_i$ has no atoms and its compact support has nonempty interior. 
    Under these two assumptions, we can obtain ergodicity and positivity. 
    One can also obtain the ergodicity and positivity of $\nu_F$ for Berry's monochromatic random waves model as shown in Appendix~\ref{APP: Ergodicity} and Appendix~\ref{APP: Positivity}, where we will discuss more on the ergodicity and the positivity without the independence assumption.
\end{remark}

To prove Theorem~\ref{Thm: Euclidean Random Field}, we define $N(x,R; F)$ as the number of all connected components of $Z(F)$ fully contained in the open cube $C_R(x) = C_R +x$, shifted from $C_R$ by $x$, of $\mb{R}^n$, and define $N^*(x,R;F)$ as the number of all connected components of $Z(F)$ that intersect 
the closed cube $\overline{C_R(x)}$. For any $c \in \mc{C}(n-m)$, we can also define $N(x,R; F , c)$ and $N^* (x,R; F , c)$  in a similar way.

So, for any $G \in C^1_* (\mb{R}^n ,\mb{R}^{m})$, $N^*(R;G) - N(R;G) \leq \mfk{N}_\#(R;G)$, where
    \begin{equation}\label{E: Definition of Residue Parts}
        \mfk{N}_\# (R;G) =  
            \begin{cases}
                \sum_{k=i} ^{n-1} \mfk{N}_k(R;G) ,& \ \text{if } Z(G) \text{ transversally intersect and only intersect with } \\
                                   \ & \  \text{$i$-th skeletons to $(n-1)$-th skeletons of $\pa C_R$ } \\
                                   \ & \  \text{for some $i \in \{ m, \dots, n-1\}$}, \\
                                   \ & \ \\
                + \infty,& \ \text{otherwise. And we denote this set as Degen$(R)$}.
            \end{cases}
    \end{equation}
Here, $\mfk{N}_k(R;G)$ is the number of connected components of $Z(G)$ which intersect with skeletons in dimension $k$ but do not intersect with skeletons in dimensions $\leq k$.

Notice that the measurability of $G \mapsto N(R;G)$ for fixed $R$ follows from its lower semicontinuity on $C^1_* (\mb{R}^n ,\mb{R}^{m})$. See for example Lemma~\ref{L: Continuous Stability}.
For the measurability of $\mfk{N}_\# (R;G)$, we notice that transversal intersection is stable under small $C^1$-perturbations, and hence $C^1_* (\mb{R}^n ,\mb{R}^{m}) \backslash \text{Degen}(R)$ is an open subset of $C^1_* (\mb{R}^n ,\mb{R}^{m})$.
By applying Bulinskaya's lemma to each skeleton of $\pa C_R$, we see that $\gamma_F(\text{Degen}(R)) = 0$. The measurability of $\mfk{N}_\#(R;G)$ then follows again from the lower semicontinuity on $C^1_* (\mb{R}^n ,\mb{R}^{m}) \backslash \text{Degen}(R)$.
    
For the expectation of $\mc{E}(\mfk{N}_k(R;F))$, we can restrict the Gaussian random field $F$ on each $k$-th skeleton of $\pa C_R$, which is a $k$-dimensional open cube in the corresponding $k$-dimensional linear subspace of $\mb{R}^n$. Such a restriction still keeps the conditions $(A1)$ and $(A2)$ in Definition~\ref{D: Axiom 1}.
Hence, by Theorem~\ref{Thm: L^1 Bound on Number}, there is a constant $C = C(k,m,M,k_1)>0$ such that $\mc{E}(\mfk{N}_k(R;F)) \leq C \cdot R^k$. So, there is a constant $C = C(n,m,M,k_1)>0$ such that $\mc{E}(\mfk{N}_\# (R;F)) \leq C \cdot R^{n-1}$. This observation is crucial in proving Theorem~\ref{Thm: Euclidean Random Field}.

We also need the following lemma.
\begin{lemma}\label{L: Integral Sandwich}
	\textit{
	For every $r \in (0,R)$, we have that
        \begin{equation}
            \frac{1}{|C_r|} \int_{C_{R-r}} N(x,r;F) \ dx \leq N(R;F) \leq \frac{1}{|C_r|} \int_{C_{R+r}} N^* (x,r;F) \ dx.
        \end{equation}
	}
\end{lemma}
The proof of this lemma follows from Lemma 1 in~\cite{NS16}. For any $c \in \mc{C}(n-m)$, similar inequalities hold true if one replaces $N(x,r;F)$ (resp., $N(R;F)$, $N^*(x,r;F)$) with $N(x,r;F,c)$ (resp., $N(R;F,c)$, $N^*(x,r;F,c)$). 
See, for example, Lemma 3.7 in~\cite{SW19}. In~\cite{W21}, there is also a similar lemma for Betti numbers.
We omit the proof for Lemma~\ref{L: Integral Sandwich}.

\begin{proof}[Proof of (1) in Theorem~\ref{Thm: Euclidean Random Field}]
    We only prove (1) for $N(R;F)$ and the proof for $N(R;F,c)$ is the same. Let
        \begin{equation}
            \nu = \limsup_{R \to \infty} \frac{\mc{E}(N(R;F))}{R^n},
        \end{equation}
    which is bounded by a constant $D_1 = D_1(n,m,M,k_1)>0$ by Theorem~\ref{Thm: L^1 Bound on Number}. So, for each $\epsilon >0$, we can choose an $r > 0$ such that 
        \begin{equation}
            \frac{\mc{E}(N(r;F))}{r^n} \geq \nu -\epsilon.
        \end{equation}
    By Lemma~\ref{L: Integral Sandwich}, we see that for all $R > r$,
        \begin{equation}
            \mc{E}(N(R;F)) \geq \frac{1}{|C_r|} \int_{C_{R-r}} \mc{E}( N(x,r;F) ) \ dx  = \frac{|C_{R-r}|}{|C_r|} \cdot \mc{E}(N(r;F)),
        \end{equation}
    since $F$ is stationary. Hence,
        \begin{equation}
            \liminf_{R \to \infty} \frac{\mc{E}(N(R;F))}{R^n} \geq \frac{\mc{E}(N(r;F))}{r^n} \geq \nu - \epsilon.
        \end{equation}
    Let $\epsilon \to 0$, we then get (\ref{E: Growth Degree Constant}).
\end{proof}
\begin{proof}[Proof of (2) in Theorem~\ref{Thm: Euclidean Random Field}]
    We only prove (2) for $N(R;F)$ and the proof for $N(R;F,c)$ is the same. We will use Wiener's ergodic theorem in~\cite{NS16}, which has also been used in~\cite{SW19,W21}.
    We define two random variables:
        \begin{equation}
            \Phi(R,r)  \equiv\frac{1}{|C_r|} \int_{C_{R-r}} N(x,r;F) \ dx = \frac{1}{|C_r|} \int_{C_{R-r}} N(r;\tau_x (F)) \ dx,
        \end{equation}
    and 
        \begin{equation}
            \Psi(R,r)  \equiv\frac{1}{|C_r|} \int_{C_{R+r}} \mfk{N}_\# (x,r;F) \ dx = \frac{1}{|C_r|} \int_{C_{R+r}} \mfk{N}_\# (r;\tau_x (F)) \ dx.
        \end{equation}
    
    By Lemma~\ref{L: Integral Sandwich}, we see that $\Phi(R,r) \leq N(R;F) \leq \Phi(R+2r,r) + \Psi(R,r) $. Also, Wiener's ergodic theorem gives that, for fixed $r$, 
        \begin{equation}
            \lim_{R \to \infty} \frac{\Phi(R,r)}{|C_{R-r}|} = \frac{\mc{E}(N(r;F))}{|C_r|} , \ \text{and } \lim_{R \to \infty} \frac{\Psi(R,r)}{|C_{R+r}|} = \frac{\mc{E}(\mfk{N}_\# (r;F))}{|C_r|}, 
        \end{equation}
    a.s.~and in $L^1(\mc{P})$. Equivalently,
        \begin{equation}
            \lim_{R \to \infty} \frac{\Phi(R,r)}{|C_{R}|} = \frac{\mc{E}(N(r;F))}{|C_r|} , \ \text{and } \lim_{R \to \infty} \frac{\Psi(R,r)}{|C_{R}|} = \frac{\mc{E}(\mfk{N}_\# (r;F))}{|C_r|}, 
        \end{equation}
    a.s.~and in $L^1(\mc{P})$. Notice that
        \begin{equation}\label{E: Split Sandwich into Parts}
            \begin{split}
                \bigg| \frac{N(R;F)}{|C_R|} - \nu  \bigg| &\leq \bigg| \frac{N(R;F)}{|C_R|} - \frac{\Phi(R,r)}{|C_R|}  \bigg| +  \bigg| \frac{\Phi(R,r)}{|C_R|} - \frac{\mc{E}(N(r;F))}{|C_r|}  \bigg| + \bigg| \frac{\mc{E}(N(r;F))}{|C_r|} - \nu  \bigg|
            \\  &= \frac{N(R;F)}{|C_R|} - \frac{\Phi(R,r)}{|C_R|}   +  \bigg| \frac{\Phi(R,r)}{|C_R|} - \frac{\mc{E}(N(r;F))}{|C_r|}  \bigg| + \bigg| \frac{\mc{E}(N(r;F))}{|C_r|} - \nu  \bigg|.
            \end{split}
        \end{equation}
    Hence, for the convergence in mean part, after taking expectations on both sides and letting $R \to \infty$, by (\ref{E: Growth Degree Constant}),
        \begin{equation}
            \begin{split}
                0 &\leq \limsup_{R\to \infty} \mc{E} \bigg[\bigg| \frac{N(R;F)}{|C_R|} - \nu  \bigg|\bigg]
                \\  &\leq \nu- \frac{\mc{E}(N(r;F))}{|C_r|}  +  \bigg| \frac{\mc{E}(N(r;F))}{|C_r|} - \nu  \bigg|.                
            \end{split}
        \end{equation}
    Then, let $r \to \infty$, by (\ref{E: Growth Degree Constant}) again, we get the convergence in mean part of (\ref{E: Ergodicity Convergence}).

    For the almost sure convergence part, we see that we can further reduce (\ref{E: Split Sandwich into Parts}) and get 
        \begin{equation}
            \begin{split}
                \bigg| \frac{N(R;F)}{|C_R|} - \nu  \bigg| &\leq \frac{\Phi(R+2r,r) + \Psi(R,r)}{|C_R|} - \frac{\Phi(R,r)}{|C_R|}  +  \bigg| \frac{\Phi(R,r)}{|C_R|} - \frac{\mc{E}(N(r;F))}{|C_r|}  \bigg|  
                    \\ & \ + \bigg| \frac{\mc{E}(N(r;F))}{|C_r|} - \nu  \bigg|.
            \end{split}
        \end{equation}
    So, a.s., 
        \begin{equation}
            0 \leq \limsup_{R \to \infty} \bigg| \frac{N(R;F)}{|C_R|} - \nu  \bigg| \leq \frac{\mc{E}(\mfk{N}_\# (r;F))}{|C_r|} + \bigg| \frac{\mc{E}(N(r;F))}{|C_r|} - \nu  \bigg|.
        \end{equation}
    By the discussions before Lemma~\ref{L: Integral Sandwich}, we see that $\mc{E}(\mfk{N}_\# (r;F)) \leq C \cdot r^{n-1}$ for a $C = C(n,m,M,k_1) >0$. Then, let $r \to \infty$ and use (\ref{E: Growth Degree Constant}) again,
    we get the almost sure convergence part of (\ref{E: Ergodicity Convergence}).
\end{proof}
\begin{proof}[Proof of (3) in Theorem~\ref{Thm: Euclidean Random Field}]
	By definition, for any $F \in C^1_* (\mb{R}^n ,\mb{R}^{m})$ and any $R > 1$,
		\begin{equation}
			N(R;F) = \sum_{c \in \mc{C}(n-m)} N(R;F,c).
		\end{equation}
	For a subset $A \subset \mc{C}(n-m)$, we define $N(R; F,A)$ as the number of connected components of $Z(F)$ lying entirely in $C_R$ with their isotopy classes lying in $A$. 
    Hence, for any finite subset $A \subset \mc{C}(n-m)$,
		\begin{equation}
			N(R;F) \geq \sum_{c \in A} N(R;F,c).
		\end{equation}	
	Take expectations on both sides, we divide both sides by $|C_R|$ and let $R \to \infty$. 
    By (1) of Theorem~\ref{Thm: Euclidean Random Field}, we see that $\nu_F \geq \sum_{c \in A} \nu_{F,c}$.
	Since $A$ is arbitrary, we see that 
		\begin{equation}
			\nu_F \geq \sum_{c \in \mc{C}(n-m)} \nu_{F,c}.
		\end{equation}
	On the other hand, for any $c \in \mc{C}(n-m)$, by Lemma~\ref{L: Integral Sandwich} applied to $N(R;F,c)$, since $F$ is stationary, we see that
		\begin{equation}
			\frac{|C_{R-r}|}{|C_r|}  \cdot \mc{E}(N(r;F,c))=\frac{1}{|C_r|} \int_{C_{R-r}} \mc{E}(N(x,r;F,c)) \ dx \leq \mc{E}(N(R;F,c)),
		\end{equation}
	when $0<r < R$. Divide both sides by $|C_R|$ and let $R \to \infty$, we see that for any $r >0$, 
		\begin{equation}
			\frac{\mc{E}(N(r;F,c))}{|C_r|} \leq \nu_{F,c}.
		\end{equation}
	Hence, by monotone convergence theorem,
		\begin{equation}
			\frac{\mc{E}(N(r;F))}{|C_r|} =  \mc{E} \bigg(\sum_{c \in \mc{C}(n-m)} \frac{N(r;F,c)}{|C_r|} \bigg) = \sum_{c \in \mc{C}(n-m)} \frac{\mc{E}(N(r;F,c))}{|C_r|} \leq \sum_{c \in \mc{C}(n-m)}\nu_{F,c}.
		\end{equation}
	Let $r \to \infty$, the left hand side becomes $\nu_F$. We then finish the proof.
		
\end{proof}

\begin{remark}
    One can also prove the part (3) of Theorem~\ref{Thm: Euclidean Random Field} by analyzing Cheeger's finiteness theorem quantitatively, like the proof of Theorem 4.2 in~\cite{SW19}, which will give readers a better understanding of the geometry of random zero sets.
\end{remark}

For Theorem~\ref{Main Results: Local Betti}, the proof needs the following inequality like Lemma~\ref{L: Integral Sandwich}, which was proved as Lemma 3.1 in~\cite{W21}.

\begin{lemma}\label{L: Integral Sandwich Betti}
	\textit{
	For each $l = 0, 1, \dots, n-m$ and every $r \in (0,R)$, we have that
        \begin{equation}
            \frac{1}{|C_r|} \int_{C_{R-r}} \beta_l (x,r;F) \ dx \leq \beta_l (R;F),
        \end{equation}
	where $\beta_l (x,r;F)$ is the sum of $l$-th Betti numbers over $\mb{R}$ of all connected components of $Z(F)$ fully contained in the open cube $C_r(x)$.
	}
\end{lemma}
Combine Lemma~\ref{L: Integral Sandwich Betti} and Theorem~\ref{Thm: L^1 Bound on Betti}, we can conclude Theorem~\ref{Main Results: Local Betti} by a simila way as we proved Theorem~\ref{Thm: Euclidean Random Field}.
We do not have an almost sure convergence for Betti numbers because for Betti numbers, we cannot similarly define an $\mfk{N}_\# (r;F)$ and get an estimate like $\mc{E}(\mfk{N}_\# (r;F)) \leq C \cdot r^{n-1}$.



\section{Proof of Theorem~\ref{Main Results: Global 1} and Theorem~\ref{Main Results: Global Betti} }\label{Section: Local Double Limit and Global Limit}

In this section, we are going to prove Theorem~\ref{Main Results: Global 1} and Theorem~\ref{Main Results: Global Betti} for parametric ensembles of vector-valued Gaussian fields. 
When $m=1$, those ensembles are constructed in a similar similar way as models in~\cite{CH20, GW15, NS16, SW19,W21}. Let us first give some assumptions similar to Definition~\ref{D: Axiom 1}.

Let ${\{F_L \}}_{L \in \mc{L}}$ be a family of Gaussian fields defined on an open subset $U \subset \mb{R}^n$, where the index $L$ attains a discrete set $\mc{L} \subset \mb{R}$. 
We usually call this family a Gaussian ensemble. We let the covariance matrix associated with $F_L = (f_{L,1}, \dots, f_{L,m})$ be 
    \begin{equation}
        {(K_L)}_{i_1 i_2}(x,y) = \mc{E}\big(f_{L, i_1}(x) f_{L,i_2}(y)\big), \ i_1,i_2 = 1,2,\dots, m,
    \end{equation}
and let the scaled version for $F_{x,L}(u) \equiv F_L(x + L^{-1} u )$ be
\begin{equation}
    {(K_{x,L})}_{i_1 i_2}(u,v) = \mc{E}({(F_{x,L})}_{i_1}(u){(F_{x,L})}_{i_2}(v)) = {(K_L)}_{i_1 i_2}(x + L^{-1}u, x + L^{-1}v).
\end{equation}

\begin{definition}\label{D: Axiom 2}
    We say that a family of Gaussian random fields ${\{F_L\}}_{L \in \mc{L}}$ is locally uniformly controllable, if for every compact subset $Q \subset U$, there are positive constants $M = M(Q),k_1 = k_1(Q)$ such that 
    \begin{itemize}
        \item [$(B1)$]   For each $x \in Q$,  $F_L(x)$ is centered and the joint distribution of $(F_L(x),\nabla F_L (x) )$ is a non-degenerate Gaussian vector with a rescaled uniformly lower bound on the covariance kernel, i.e.,
				\begin{equation}
					\liminf_{L \to \infty} \inf_{x \in Q} \inf_{(\eta,\xi) \in \mb{S}^{m+mn-1}}\mc{E}\bigg(\bigg| \sum_{i = 1}^m\eta_i \cdot f_{L,i} (x) + L^{-1} \cdot \sum_{i=1} ^m \sum_{j=1} ^n \xi_{ij} \cdot \pa_j f_{L,i} (x)\bigg|^2 \bigg) \geq k_1 >0,
				\end{equation} 
                where $\eta = (\eta_i) \in \mb{R}^m$, $\xi = (\xi_{ij}) \in \mb{R}^{nm}$ with $||\eta||^2 + ||\xi||^2 = 1$.
        \item [$(B2)$]  We have the local uniform $C^{3-}$-smoothness:
				\begin{equation}
					\limsup_{L \to \infty}  \max_{0\leq { |\alpha| }\leq 3} \sup_{x \in Q} L^{-2|\alpha|} {| {{D_x}^{\alpha} {D_y }^{\alpha}{(K_L)}_{i_1 i_2}(x,y) |}_{y=x}|} \leq M, \ i_1,i_2 = 1,2,\dots, m.
				\end{equation}  
    \end{itemize}
    
\end{definition}

For an $x \in U$, if there exists a continuous stationary Gaussian field $F_x : \mb{R}^n \to \mb{R}^m$, such that for every finite point set $\mc{U} \subset \mb{R}^n$, the Gaussian vectors $(F_{x,L}) |_{\mc{U}}$ converge to $(F_x)|_{\mc{U}}$ in distribution, then
we call $F_x$ the translation invariant local limit of ${\{F_L\}}_{L \in \mc{L}}$ at $x$. And hence,
    \begin{equation}
        \lim_{L\to \infty} K_{x,L}(u,v) = K_x(u-v) \quad \text{for any } (u,v) \in \mb{R}^n \times \mb{R}^n,
    \end{equation}
where $K_x$ is the covariance kernel of $F_x$. On the other hand, since we have the uniform smoothness assumption $(B2)$, the existence of such an $F_x$ is equivalent to the existence of the limiting covariance kernel $K_x$. See more discussions in Appendix A.12 of~\cite{NS16}.

\begin{definition}\label{D: Axiom 3}
	For a family of Gaussian fields ${\{F_L\}}_{L \in \mc{L}}$ defined on $U \subset \mb{R}^n$ that are also locally uniformly controllable, we say that ${\{F_L\}}_{L \in \mc{L}}$ is tame if there exists a Borel subset $U' \subset U$ of full Lebesgue measure such that, for all $x \in U'$,
	\begin{itemize}
		\item [$(B3)$] ${\{F_L\}}_{L \in \mc{L}}$ has a translation invariant local limit $F_x$ at $x$, which satisfies that the action of $\mb{R}^n$ on $(C^1_* (\mb{R}^n ,\mb{R}^{m}), \mfk{B}(C^1_* (\mb{R}^n ,\mb{R}^{m})), \gamma_{F_x} )$ is ergodic.
		\item [$(B4)$] There is a $k_2 = k_2(x)$, and for every $R > 0$, there is an $L_0 = L_0 (R,x)>0$, such that when $L > L_0$, $F_{x,L}(u)$ satisfies $(A3)$ in Definition~\ref{D: Axiom 1}, i.e., if $x \in Q \cap U'$ for some compact set $Q \subset U$, then when $L > L_0$, $F_{x,L}(u)$ satisfies \textit{ $(R;M(Q),k_1(Q),k_2(x))$-assumptions} on $C_{R+1}$.
	\end{itemize}
\end{definition}

\begin{remark}
    For an $n$-dimensional $C^3$-manifold $X$ without boundary, we say that a parametric Gaussian ensemble ${\{F_L\}}_{L \in \mc{L}}$ is tame if for every $C^3$-chart $(U,\pi)$, $\pi: U \subset \mb{R}^n \to X$, ${\{F_L \circ \pi \}}_{L \in \mc{L}}$ is tame on $U$.
    This definition does not depend on the choice of charts. See Section 9 of~\cite{NS16} when $m=1$.
\end{remark}

We then give an example, the Kostlan's ensemble, as a tame ensemble defined on $\mb{S}^n \subset \mb{R}^{n+1}$ when $m = 1$. When $m>1$, we can let $F_L = (f_{L,1}, \dots , f_{L,m})$ consist of independent ${\{f_{L,i}\}}_{L}$'s so that each ${\{f_{L,i}\}}_{L}$ is a Kostlan's ensemble (or other general tame Gaussian functions ensembles) defined on the same manifold. 
See more examples in~\cite{NS16, S16}.
Apart from these concrete examples, in~\cite{CH20,SW19} and references therein, one will see a large family of Gaussian functions ensembles satisfying our definitions of tame ensembles and these definitions are actually very generic.

\begin{example}[Kostlan's ensemble]\label{Example}
    Consider the linear space of real homogeneous polynomials of degree $d$ with coordinates $(x_0, \dots, x_n) \in \mb{R}^{n+1}$, which is endowed with the inner product
        \begin{equation}
            \langle P_d, Q_d \rangle = \sum_{|J|  = d } \binom{d}{J} ^{-1} p_J q_J,
        \end{equation}
    where 
        \begin{equation}
            J = (j_0 , \dots, j_n) , \ |J| = \sum_{s= 0 } ^n j_s  , \  \binom{d}{J} = \frac{d !}{j_0 ! \cdot \cdots \cdot j_n !} , 
        \end{equation}
    and
        \begin{equation}
            P_d(x) = \sum_{|J| = d} p_J x^J, \  Q_d(x) = \sum_{|J| = d} q_J x^J , \ x^J = x_0 ^{j_0} \cdot \cdots \cdot x_n ^{j_n}.
        \end{equation}
    This inner product, up to a positive constant $C = C(n,d)$, equals to the inner product on the Bargmann-Fock space~\cite{B61}, i.e., the subspace of analytic functions on $\mb{C}^{n+1}$ where the inner product
        \begin{equation}
            \langle f , g \rangle_{BK} \equiv C(n,d)\int_{\mb{C}^{n+1}} f(z) \overline{g(z)} e^{-||z||^2} \ d \vol(z)
        \end{equation}
    is well-defined. One has that $\langle P_d , Q_d \rangle = \langle P_d , Q_d \rangle_{BK}$ after the extension of $P_d$ and $Q_d$ to $\mb{C}^{n+1}$.
    Hence, we get an orthonormal basis
        \begin{equation}
            {\bigg \{ \sqrt{ \binom{d}{J} } x^J \bigg \}}_{|J| = d}
        \end{equation}
    on the linear space of real homogeneous polynomials of degree $d$, with respect to the inner product we defined above.
    We then obtain a random homogeneous polynomial $R_d(x)$ of degree $d$, i.e.,
        \begin{equation}
            R_d(x) = \sum_{|J| = d} \sqrt{ \binom{d}{J} } a_{J} x^J,
        \end{equation}
    where ${\{a_J\}}_{|J| = d}$ are i.i.d.~standard Gaussian random variables. The zero sets of $R_d(x)$ make sense as hypersurfaces on either $\mb{S}^{n}$ or $\mb{R}P^n$, the real projective space of dimension $n$.
    Hence, we can view ${\{R_d(x)\}}_{d \in \mb{N}}$ as a family of Gaussian random functions on $\mb{S}^n$. The two-point covariance kernel is 
        \begin{equation}
            \mc{E}(R_d(x)R_d(y)) = {(\langle x , y \rangle_{\mb{R}^{n+1}})}^d = {(\cos(\theta(x,y)))}^d,
        \end{equation}
    where $\langle x , y \rangle_{\mb{R}^{n+1}}$ is the standard inner product and $\theta(x,y)$ is the angle between $x,y$ as vectors in $\mb{R}^{n+1}$. Fix an $x \in \mb{S}^n$, let $\exp_x: \mb{R}^n = T_x \mb{S}^n \to \mb{S}^n$ be the exponential map at $x$. We consider
        \begin{equation}
            R_{x,d}(u) \equiv R_d \big( \exp_x(d^{-1/2} \cdot u ) \big)
        \end{equation}
    and then
        \begin{equation}
            K_{x,d}(u,v) \equiv \mc{E}(R_{x,d}(u) R_{x,d}(v) ) = {(\langle \exp_x(d^{-1/2} \cdot u )  , \exp_x(d^{-1/2} \cdot v )  \rangle_{\mb{R}^{n+1}})}^d.
        \end{equation}
    As $d \to \infty$, $K_{x,d}(u,v)$, together with its partial derivatives of any finite order, locally uniformly converges to a covariance kernel $K_{R_x}(u,v)$ of a stationary Gaussian function $R_x(u)$ defined on $\mb{R}^n$. This $K_{R_x}$ is actually independent of $x$, and one can see that
        \begin{equation}
            K_{R_x}(u,v) = K_{R}(u-v) \equiv e^{-||u-v||^2/2}.
        \end{equation}
    We have already seen this kernel in Example~\ref{Example: Bargmann-Fock}.
    Define the parameters set $\mc{L} \equiv {\{\sqrt{d} \}}_{d\in \mb{N}}$, one can then verify that this Kostlan's ensemble satisfies all of our definitions of tame ensembles. To see this, we first notice that $K_R$ will satisfy all definitions for some universal constants only depending on $n$, and one uses the uniform convergence from $K_{x,d}$ to $K_R$ to show that these $K_{x,d}$ satisfy our definitions uniformly.
\end{example}

Now, we focus on ${\{F_L\}}_{L \in \mc{L}}$ defined on an open subset $U \subset \mb{R}^n$ again. Assume that ${\{F_L\}}_{{L \in \mc{L}}}$ is tame and $U'$ is the full measure set in Definition~\ref{D: Axiom 3}. 
For each $x \in U'$, we define $\bar{\nu}(x ) \equiv \nu_{F_x}$, where we obtained each $\nu_{F_x}$ from Theorem~\ref{Thm: Euclidean Random Field}.
We say that a Borel measure $n_L$ is a connected component counting measure of $F_L$ if $\spt (n_L) \subset Z(F_L)$ and the $n_L$-mass of each component is $1$.
Now, we can restate Theorem~\ref{Main Results: Global 1} as the following.

\begin{theorem}\label{Thm: Global Limit}
	\textit{
	Assume that ${\{F_L\}}_{L \in \mc{L}}$ is tame on an open set $U \subset \mb{R}^n$. Then,
		\begin{itemize}
			\item [(1)] The function $x \mapsto \bar{\nu}(x )$ is measurable and locally bounded in $U$.
			\item [(2)] For every sequence of connected component counting measures $n_L$ of $F_L$ and for every $\varphi \in C_{c}(U)$, we have that
				\begin{equation}
					\lim_{L \to \infty} \mc{E} \bigg[ \bigg |  \frac{1}{L^n} \int_U \varphi(x)  dn_L(x) - \int_U \varphi(x) \bar{\nu}(x) dx \bigg | \bigg] = 0.
				\end{equation}
		\end{itemize}
	}
\end{theorem}


Consider an $x \in U' $ with the translation invariant limit
$F_x : \mb{R}^n \to \mb{R}^m$, which is centered and stationary. We see, for example, by Appendix A.12 of~\cite{NS16}, that $F_x(u)$ satisfies $(A1)$, $(A2)$, and $(A3)$ in Definition~\ref{D: Axiom 1} at the origin $u=0$ with $M(Q)$ and $k_1(Q)$ if $x \in Q$ for a compact set $Q \subset U$, and with $k_2(x)$ in $(B4)$ of Definition~\ref{D: Axiom 3}. We then have a local double limit lemma since $F_x$ also satisfies the ergodicity assumption $(B3)$ in Definition~\ref{D: Axiom 3}.
\begin{lemma}\label{L: Double Limit in Probability}
	\textit{
    Assume that ${\{F_L\}}_{L \in \mc{L}}$ is locally uniformly controllable and $F_x$ is its translation invariant local limit at a fixed $x \in U$. 
    We also assume that for this $F_x$, the action of $\mb{R}^n$ on $(C^1_* (\mb{R}^n ,\mb{R}^{m}), \mfk{B}(C^1_* (\mb{R}^n ,\mb{R}^{m})), \gamma_{F_x} )$ is ergodic.
    Then, for any $\epsilon >0$,
        \begin{equation}
            \lim_{R \to \infty} \limsup_{L \to \infty} \mc{P} \bigg(\bigg| \frac{N(R; F_{x,L})}{|C_R|} - \nu_{F_x}\bigg| > \epsilon \bigg) = 0 . 
        \end{equation}
	}
\end{lemma}
\begin{remark}\label{Rmk: Smooth Double Limit}
	For each $l = 0, 1, \dots, n-m$, Lemma~\ref{L: Double Limit in Probability} holds true if we replace $N(R; F_{x,L})$ with $\beta_l(R; F_{x,L})$ and replace $\nu_{F_x}$ with $\nu_{l;F_x}$. 
    For any $c \in \mc{C}(n-m)$, Lemma~\ref{L: Double Limit in Probability} also holds true if one replaces $N(R; F_{x,L})$ with $N(R; F_{x,L},c)$ and replaces $\nu_{F_x}$ with $\nu_{F_x,c}$. 
    
\end{remark}

When $m=1$, the proof of this lemma was shown in Theorem 5 in~\cite{S16}, Proposition 6.2 of~\cite{SW19}, or Theorem 1.5 in~\cite{W21}. 
In order to prove this Lemma~\ref{L: Double Limit in Probability}, one needs to replace some lemmas for random functions with those for random fields, which are our Lemma~\ref{L: Quantitative Bulinskaya} and Lemma~\ref{L: Continuous Stability}. 
So, we only sketch the proof idea.
On the other hand, this result can be improved to $L^1$-convergence, i.e., 
    \begin{equation}
        \lim_{R \to \infty} \limsup_{L \to \infty} \mc{E} \bigg(\bigg| \frac{N(R; F_{x,L})}{|C_R|} - \nu_{F_x}\bigg|  \bigg) = 0 ,
    \end{equation}
after building up our Theorem~\ref{Thm: L^q Bound on Number}. See (\ref{E: Double Limit in L^1}) in the proof for part (1) of Theorem~\ref{Thm: Global Limit}.

\begin{proof}[Sketch of the proof idea for Lemma~\ref{L: Double Limit in Probability}]
    Recall that
        \begin{equation}
            \lim_{R \to \infty} \mc{P} \bigg(\bigg| \frac{N(R; F_{x})}{|C_R|} - \nu_{F_x}\bigg| > \frac{\epsilon}{2} \bigg) = 0,
        \end{equation}
    by Theorem~\ref{Thm: Euclidean Random Field}. Fix a large $R > 0$. One can also show that 
        \begin{equation}
            \lim_{A \to \infty} \mc{P} \big( ||F_x||_{C^{2}(C_{2R})} > A \big) = 0,
        \end{equation}
    and by our quantitative Bulinskaya's lemma, Lemma~\ref{L: Quantitative Bulinskaya},
        \begin{equation}
            \lim_{\tau \to 0} \mc{P} \big( \min_{y \in \bar{C}_{2R}} \max\{||F_x(y)||,\lambda(F_x(y)) \} < \tau \big) = 0 .
        \end{equation}
    Also, one can show that for any $\delta >0$,
        \begin{equation}
            \limsup_{L \to \infty} \mc{P} \big( ||F_{x,L} - F_x||_{C^1(C_{2R})} > \delta \big) = 0.
        \end{equation}
    The above limits show that, outside an event $\Sigma$ with small probability, $Z(F_x) \cap C_R$ are all regular submanifolds, and $Z(F_{x,L}) \cap C_R$ is actually a tiny perturbation of $Z(F_{x}) \cap C_R$. Hence, for those components of $Z(F_{x,L})$ fully contained in $C_R$, they are $C^1$-isotopic to those components of $Z(F_{x})$ fully contained in $C_{R+1}$.
    In particular, $N(R-1;F_{x}) \leq  N(R;F_{x,L}) \leq N(R+1;F_{x})$. Hence, 
        \begin{equation}
            \begin{split}
                \mc{P} \bigg(\bigg| \frac{N(R; F_{x,L})}{|C_R|} - \nu_{F_x}\bigg| > \epsilon \bigg) &\leq \mc{P} \bigg( \frac{N(R+1; F_{x})}{|C_R|} - \nu_{F_x} > \epsilon \bigg)
                \\  & \  + \mc{P} \bigg( \frac{N(R-1; F_{x})}{|C_R|} - \nu_{F_x} < - \epsilon \bigg) + \mc{P}(\Sigma).
            \end{split}
        \end{equation}

Let us formulate the $C^1$-isotopic result we used above.

\begin{lemma}\label{L: Continuous Stability}
	\textit{
    Fix a small $\tau >0$, a large $A>0$, and a $\beta \in (0,1)$. Assume that $G(x) \in C^{1+\beta}(C_{R+2}, \mb{R}^m)$ for some $m \in (0,n)$ and $||G||_{C^{1+\beta}(C_{R+2})} < A$. If for each $x \in C_{R+1}$, either $||G(x)|| > \tau$ or $\lambda(G(x)) > \tau$ (see the definition for $\lambda(G(x))$ in Lemma~\ref{L: Quantitative Bulinskaya}), then, there is a positive constant $\delta = \delta(n , m, \tau, A, \beta) > 0$,
    such that for any $C^1$-vector field $\tilde{G}(x) \in C^1(C_{R+2}, \mb{R}^m)$ with $|| \tilde{G}(x) - G(x) ||_{C^1(C_{R+1})} < \delta$, one can show that each connected component of $Z(G)$ fully contained in $C_{R-1}$ is $C^1$-isotopic to a connected component of $Z(\tilde{G})$ fully contained in $C_{R}$, and this map is injective.
    In particular, $N(R; \tilde{G}) \geq N(R-1; G)$.
	}
\end{lemma}
For this stability lemma, one can see, for example, Lemma 4.3 of~\cite{W21}, Proposition 6.8 of~\cite{SW19}, and Thom's isotopy Theorem.
The general idea is considering $G_t(x) = G(x) + t \cdot (\tilde{G}(x)-G(x))$, $t \in [0,1]$. Since now $|| \tilde{G}(x) - G(x) ||_{C^1(C_{R+1})} $ is small, for each $t$, we know that $Z(G_t)$ is regular. 
For each connected component of $Z(G)$ fully contained in $C_{R-1}$, say $\gamma$, the normal vector bundle ${(T\gamma)}^\perp$ is trivial and one can use $G(x)$ as a local trivialization of this normal vector bundle.
At each $x \in \gamma$, one can consider an $m$-dimensional disc $D^m(x,r) \subset {(T_x \gamma)}^{\perp} \subset \mb{R}^n$ for some $r = r(n,m,\tau,A,\beta)>0$ small, and then apply contraction mapping theorem to get a unique $y_t$ in each $D(x,r)$ such that $G_t(y_t) = 0$. This process actually builds up the $C^1$-isotopy.
If one only needs the inequality $N(R; \tilde{G}) \geq N(R-1; G)$, one can just use $C^0$-perturbations, i.e., just assume that $|| \tilde{G} - G ||_{C^0(C_{R+1})}$ is small. 
The proof follows by replacing the contraction mapping theorem with Brouwer's fixed-point theorem. Similar inequalities of $C^0$-perturbations also hold true for Betti numbers, see for example Theorem 2 of~\cite{LS19}.

So, we can finish the proof idea for Lemma~\ref{L: Double Limit in Probability}.

\end{proof}

In order to prove Theorem~\ref{Thm: Global Limit}, we need an estimate stronger than Theorem~\ref{Thm: L^1 Bound on Number} with the assumption $(A3)$ in Definition~\ref{D: Axiom 1}.

\begin{theorem}\label{Thm: L^q Bound on Number}
	\textit{
    Assume that $F : C_{R+1} \to \mb{R}^{m}$, satisfies \textit{ $(R; M,k_1,k_2)$-assumptions} on $C_{R+1}$. Then, there is a constant $D_2 = D_2(n,m, M,k_1,k_2) >0$ and a constant $q = q(n,m) > 1$, such that for $R >1$,
        \begin{equation}
            \mc{E}({(N(R;F))}^q) \leq D_2 \cdot |C_R|^q.
        \end{equation}
	}
\end{theorem}
Before we give the proof, let us remark that a similar result also holds true for Betti numbers.
\begin{theorem}\label{Thm: L^q Bound on Betti}
	\textit{
    Assume that $F : C_{R+1} \to \mb{R}^{m}$, satisfies \textit{ $(R; M,k_1,k_2)$-assumptions} on $C_{R+1}$. Then there is a constant $\widetilde{D}_2 = \widetilde{D}_2(n,m, M,k_1,k_2) >0$ and a constant $q = q(n,m) > 1$, such that for $R >1$,
        \begin{equation}
            \sum_{l=0}^{n-m} \mc{E}({(\beta_l(R;F))}^q) \leq \widetilde{D}_2 \cdot |C_R|^q.
        \end{equation}
	}
\end{theorem}
The necessary modification is to again use the Chern-Lashof inequality in Theorem~\ref{Thm: Chern}. Hence, we only give the proof for Theorem~\ref{Thm: L^q Bound on Number}.

\begin{proof}[Proof of Theorem~\ref{Thm: L^q Bound on Number}]
    For a $q = q(n,m)>1$ but close to $1$ and a positive integer $k$ to be determined later, by (\ref{E: Bound Components Number}), we have that for some positive constants $C = C(n,m)$,
    \begin{equation}\label{E: Bound Components Number 2}
        \begin{split}
            &{(N(R;F))}^q \leq C(n,m) \cdot {\bigg( \int_{F^{-1}(0) \cap C_R} ||\bH||^{n-m} \ d\mc{H}^{n-m} \bigg)}^q
            \\         &\leq C \cdot \bigg( \int_{F^{-1}(0) \cap C_R} ||\bH||^{(n-m)q} \ d\mc{H}^{n-m} \bigg) \cdot {\big(\mc{H}^{n-m}(F^{-1}(0) \cap C_R)\big)}^{q-1}
            \\         &= C \cdot  \bigg( \int_{F^{-1}(0) \cap C_R} ||\bH||^{(n-m)q} \ d\mc{H}^{n-m} \bigg) \cdot {\bigg(\frac{\mc{H}^{n-m}(F^{-1}(0) \cap C_R)}{|C_R|}\bigg)}^{q-1} \cdot |C_R|^{q-1}
            \\         &= C \cdot |C_R|^{q-1}  \bigg[ \int_{F^{-1}(0) \cap C_R} {(||\bH||^{\frac{(n-m)q(k+1)}{k}})}^{\frac{k}{k+1}}    
            \\         & \qquad \qquad \qquad \quad  \cdot {\bigg( {\bigg(\frac{\mc{H}^{n-m}(F^{-1}(0) \cap C_R)}{|C_R|}\bigg)}^{(k+1)(q-1)} \bigg)}^{\frac{1}{k+1}}  \ d\mc{H}^{n-m} \bigg]
            \\         &\leq \frac{C}{k+1}  |C_R|^{q-1} \bigg[ \int_{F^{-1}(0) \cap C_R} k  (||\bH||^{\frac{(n-m)q(k+1)}{k}}) 
            \\         & \qquad \qquad \qquad \quad + \bigg( {\bigg(\frac{\mc{H}^{n-m}(F^{-1}(0) \cap C_R)}{|C_R|}\bigg)}^{(k+1)(q-1)} \bigg)  \ d\mc{H}^{n-m} \bigg]
            \\         &\leq C \cdot   |C_R|^{q-1} \int_{F^{-1}(0) \cap C_R}  (||\bH||^{\frac{(n-m)q(k+1)}{k}})   \ d\mc{H}^{n-m}
            \\         & \quad +C \cdot |C_R|^{-k(q-1)}{\big(\mc{H}^{n-m}(F^{-1}(0) \cap C_R)\big)}^{1+(k+1)(q-1)}.
        \end{split}
    \end{equation}
We set $k = \frac{2-q}{q-1}$ and we will determine $q = q(m,n)$ later. Then, we can see that 
    \begin{equation}\label{E: Bound Components Number 3}
       \begin{split}
        {(N(R;F))}^q &\leq  C \cdot |C_R|^{q-1} \int_{F^{-1}(0) \cap C_R}  (|\bH|^{\frac{(n-m)q}{2-q}})   \ d\mc{H}^{n-m} 
        \\ & \quad + C\cdot |C_R|^{q-2}{\big(\mc{H}^{n-m}(F^{-1}(0) \cap C_R)\big)}^{2}.
       \end{split}
    \end{equation}
Notice that when $q$ is greater than $1$ and close to $1$, $q/(2-q)$ is also greater than $1$ and close to $1$. To simplify the notation, we let $q/(2-q) = \theta >1$.

In order to estimate $\mc{E}({(N(R;F))}^q)$, by (\ref{E: Bound Components Number 3}), we need to estimate 
    \begin{equation}\label{E: Bound Term 1}
        \mc{E} \bigg(\int_{F^{-1}(0) \cap C_R}  (|\bH|^{(n-m)\theta})   \ d\mc{H}^{n-m} \bigg )
    \end{equation}
and 
    \begin{equation}\label{E: Bound Term 2}
        \mc{E}\bigg({\big(\mc{H}^{n-m}(F^{-1}(0) \cap C_R)\big)}^{2} \bigg).
    \end{equation}

Since $\theta$ is close to $1$, the estimate for (\ref{E: Bound Term 1}) is the same as Theorem~\ref{Thm: L^1 Bound on Number} and, in particular, the same as estimating (\ref{E: Determinant Integral in Curvature}). The $\theta$ is then allowed to be greater than $1$ but close to $1$ because of the sharp constant $\alpha$ shown in Lemma~\ref{L: Singular Integral}.
So, we can choose a $\theta = \theta(n,m)>1$ and choose a $q = q(n,m)>1$ such that $k+1 = \frac{1}{q-1}$ is a positive integer. 
We also get that (\ref{E: Bound Term 1}) is bounded by $|C_R|$ multiplying a constant $C = C(n,m,M,k_1) >0$.

For (\ref{E: Bound Term 2}), by the Kac-Rice Theorem~\ref{Thm: Kac-Rice}, we have that
        \begin{equation}
            \mc{E}\bigg({\big(\mc{H}^{n-m}(F^{-1}(0) \cap C_R)\big)}^{2} \bigg) = \int_{C_R \times C_R} J_{2,F}(1;u_1,u_2) p_{(F(u_1), F(u_2))}(0) \ du_1 du_2,
        \end{equation}
where  $J_{2,F}(h;u_1, u_2) $ has the expression that $J_{2,F}(h;u_1, u_2) = $
    \begin{equation*}
       \mc{E}\bigg[ \sqrt{\det[{(\nabla F (u_1))}^T (\nabla F(u_1) )]} \sqrt{\det[{(\nabla F (u_2))}^T (\nabla F(u_2) )]} \ \bigg| \  F(u_1) = F(u_2) =  0 \bigg].  
    \end{equation*}
We have an estimate for $p_{(F(u_1), F(u_2))}(0)$ from the $(A3)$ assumption in Definition~\ref{D: Axiom 1}. 
For this $J_{2,F}$ term, we will prove that it has a positive upper bound $C = C(n,m,M,k_1,k_2)< \infty$.
By the inequality $2ab \leq a^2 + b^2$ and symmetry, we only need to estimate
    \begin{equation}\label{E: Jacobian in L^2 Coarea}
        \mc{E}\big[ \det[{(\nabla F (u_1))}^T (\nabla F(u_1) )] \ \big| \  F(u_1) = F(u_2) =  0 \big],
    \end{equation}
for $u_1,u_2 \in C_R$ and $u_1 \neq u_2$.
By similar tricks in (\ref{E: Independence for Conditional Expectation}), fix $i \in \{1, \dots , m\}$ and $j \in \{1 ,\dots , n \}$, we first need to find $a_{j,i} ^t$ and $b_{j,i} ^t$ such that for each $s = 1, \dots, m$,
    \begin{equation}\label{E: Orthogonal Condition 1}
        \mc{E}\big[ (\pa_j f_i (u_1) + \sum_{t=1} ^m a_{j,i} ^t f_t(u_1)+ b_{j,i} ^t f_t(u_2)) \cdot f_s(u_1)\big] = 0,
    \end{equation}
and
    \begin{equation}\label{E: Orthogonal Condition 2}
        \mc{E}\big[ (\pa_j f_i (u_1) + \sum_{t=1} ^m a_{j,i} ^t f_t(u_1)+ b_{j,i} ^t f_t(u_2)) \cdot f_s(u_2)\big] = 0.
    \end{equation}
To simplify the notations, we let $v = \pa_j f_i (u_1)$, $A = (a_{j,i}^1, \dots, a_{j,i}^m)$, $B = (b_{j,i}^1, \dots, b_{j,i}^m)$, $\mc{E}(F(u_1)v) = (\mc{E}(f_1(u_1)v) ,\dots, \mc{E}(f_m(u_1)v))$, $\mc{E}(F(u_2)v) = (\mc{E}(f_1(u_2)v) ,\dots, \mc{E}(f_m(u_2)v))$. Set $\mr{Cov}(x,y)$ as the $m\times m $ matrix with elements ${(\mr{Cov}(x,y))}_{i_1 i_2} = \mc{E}(f_{i_1}(x) f_{i_2}(y))$, and set $\Lambda(x,y)$ as the covariance kernel of the joint distribution of $(F(x),F(y))$, i.e.,
    \begin{equation}
	\Lambda(x,y)=
	\begin{bmatrix}
	  \mr{Cov}(x,x) & \mr{Cov}(x,y)\\
	  \mr{Cov}(y,x) & \mr{Cov}(y,y)
	\end{bmatrix}    .
    \end{equation}
$\Lambda(x,y)$ is invertible when $x \neq y$ by the $(A3)$ assumption.
Hence, the equations (\ref{E: Orthogonal Condition 1}) and (\ref{E: Orthogonal Condition 2}) are solvable, and we get 
    \begin{equation}\label{E: Solved Orthogonal Condition}
	    (A,B)  = - (\mc{E}(F(u_1)v), \mc{E}(F(u_2)v)) \cdot {(\Lambda (u_1,u_2))}^{-1}.
    \end{equation}
Notice that 
    \begin{equation}
        \det[{(\nabla F (u_1))}^T (\nabla F(u_1) )] \leq C(n,m) \cdot {||\nabla F ||}^{2m} ,
    \end{equation}
we see that the upper bound for (\ref{E: Jacobian in L^2 Coarea}) reduces to the upper bound for
    \begin{equation}
	    \mc{E}{(v + A\cdot F(u_1) + B \cdot F(u_2))}^{2m}.
    \end{equation}
Since it is the $2m$-th moment of a centered Gaussian variable, we only need to estimate its variance. By (\ref{E: Solved Orthogonal Condition}), we have that
    \begin{equation}
	    \begin{split}
            & \quad \mc{E}{( A\cdot F(u_1) + B \cdot F(u_2))}^{2} 
            \\ &= (\mc{E}(F(u_1)v), \mc{E}(F(u_2)v)) \cdot {(\Lambda (u_1,u_2))}^{-1} \cdot {(\mc{E}(F(u_1)v), \mc{E}(F(u_2)v))}^T.
        \end{split}
    \end{equation}
We will use Theorem~\ref{Thm: APP C 1} and consider the case $|u_1 - u_2| \leq \delta$ and the case $|u_1 - u_2 | > \delta$ separately, where $\delta$ is chosen in Theorem~\ref{Thm: APP C 1}. When $|u_1 - u_2| \leq \delta$,
by elementary row operations, 
    \begin{equation}
	\begin{split}
		\Lambda (u_1,u_2) & = \begin{bmatrix}
			\mr{Id}_m & 0\\
			\mr{Cov}(u_2,u_1) {\mr{Cov}(u_1,u_1)}^{-1} & {\mr{Id}}_m
		      \end{bmatrix}    
		\\	\quad &\cdot 
		      \begin{bmatrix}
			\mr{Cov}(u_1,u_1) & 0\\
			0 & \mr{Cov}(u_2,u_2) - \mr{Cov}(u_2,u_1) {\mr{Cov}(u_1,u_1)}^{-1}\mr{Cov}(u_1,u_2)
		      \end{bmatrix}  
		\\	\quad &\cdot 
		      \begin{bmatrix}
			\mr{Id}_m &  {\mr{Cov}(u_1,u_1)}^{-1}\mr{Cov}(u_1,u_2)\\
			0 & \mr{Id}_m
		      \end{bmatrix}  .
	\end{split}
    \end{equation}
Hence, 
    \begin{equation}\label{E: Bound Variance of Two Points}
	\begin{split}
		&\quad \mc{E}{( A F(u_1) + B  F(u_2))}^{2}  
        \\ &= (\mc{E}(F(u_1)v), \mc{E}(F(u_2)v) - \mc{E}(F(u_1)v) {\mr{Cov}(u_1,u_1)}^{-1} \mr{Cov}(u_1,u_2))
		\\	&\quad \cdot 
		      \begin{bmatrix}
			\mr{Cov}(u_1,u_1) & 0\\
			0 & \mr{Cov}(u_2,u_2) - \mr{Cov}(u_2,u_1) {\mr{Cov}(u_1,u_1)}^{-1}\mr{Cov}(u_1,u_2)
		      \end{bmatrix} ^{-1} 
		\\   &\quad \cdot 
		{(\mc{E}(F(u_1)v), \mc{E}(F(u_2)v) - \mc{E}(F(u_1)v) {\mr{Cov}(u_1,u_1)}^{-1} \mr{Cov}(u_1,u_2))}^T . 
	\end{split}
    \end{equation}
For the middle matrix in (\ref{E: Bound Variance of Two Points}), the $\mr{Cov}(u_1,u_1)$ block is good, because by $(A1)$ in Definition~\ref{D: Axiom 1}, $\mr{Cov}(u_1,u_1)$ is quantitatively non-degenerate. For the second $m \times m$ block, by Theorem~\ref{Thm: APP C 1},
the smallest eigenvalue of $\mr{Cov}(u_2,u_2) - \mr{Cov}(u_2,u_1) {\mr{Cov}(u_1,u_1)}^{-1}\mr{Cov}(u_1,u_2)$ is lowerly bounded by $(k_1/2) \cdot |u_1-u_2 |^2$. On the other hand, by the mean value theorem,
    \begin{equation}
	\big| \mc{E}(F(u_2)v) - \mc{E}(F(u_1)v) {\mr{Cov}(u_1,u_1)}^{-1} \mr{Cov}(u_1,u_2) \big|^2 \leq C(n,m,M) \cdot |u_1 - u_2|^2 .
    \end{equation}
Hence,  when $|u_1 - u_2| \leq \delta$,
    \begin{equation}
	\mc{E}{( A F(u_1) + B  F(u_2))}^{2} \leq C(n,m,M,k_1).
    \end{equation}
When $|u_1 - u_2| > \delta$, by $(A3)$ in Definition~\ref{D: Axiom 1}, we see that there is a positive constant $c = c(n,m,M,k_1,k_2)$ such that
    \begin{equation}
	    \det(\Lambda(u_1,u_2)) \geq c.
    \end{equation}
Since the norm of $(\mc{E}(F(u_1)v), \mc{E}(F(u_2)v))$ is bounded by a constant $C = C(n,m,M)>0$, we get that
    \begin{equation}
	\mc{E}{( A F(u_1) + B  F(u_2))}^{2} \leq C(n,m,M,k_1,k_2).
    \end{equation}
Combine with the fact that $\mc{E}(v^2) \leq C(n,m,M)$, we can then prove that 
    \begin{equation}
        \mc{E}{(v + A\cdot F(u_1) + B \cdot F(u_2))}^{2}\leq C(n,m,M,k_1,k_2).
    \end{equation}
Hence, by the previous arguments,
    \begin{equation}
        J_{2,F} \leq C(n,m,M,k_1,k_2).
    \end{equation}
Combine this with $(A4)$ in Definition~\ref{D: Axiom 1}, i.e.,
    \begin{equation}
        p_{(F(u_1), F(u_2))} (0) \leq 
        \begin{cases}
            \frac{k_2}{|u_1-u_2|^m} ,& \ \text{if } |u_1-u_2| \leq 1, \\
            & \\
            k_2 ,& \ \text{if } |u_1-u_2| >1,
        \end{cases}
    \end{equation}
we get that, for $R > 1$, 
    \begin{equation}
        \begin{split}
            & \quad \mc{E}\bigg({\big(\mc{H}^{n-m}(F^{-1}(0) \cap C_R)\big)}^{2} \bigg) 
        \\    &\leq C(n,m, M,k_1,k_2) \cdot R^{n} \cdot (\int_0 ^1 r^{n-1-m} \ dr+ \int_1 ^{2R} r^{n-1 } \ dr )
        \\  &\leq C(n,m, M,k_1,k_2) \cdot R^{2n }.
        \end{split}
    \end{equation}
 
Now, combine estimates for (\ref{E: Bound Term 1}) and (\ref{E: Bound Term 2}), we can then estimate $\mc{E}({(N(R;F))}^q)$ with some $q  = q(n,m)>1$ but close to 1 by (\ref{E: Bound Components Number 3}), and finally obtain that for all $R >1$,
    \begin{equation}
        \mc{E}({(N(R;F))}^q) \leq C(n,m,  M, k_1,k_2) \cdot R^q.
    \end{equation}

\end{proof}

In Theorem~\ref{Thm: L^q Bound on Number}, we need to add an additional condition $(A3)$ to insure the nondegeneracy of $F$ at two different points, so that we can get $L^q$-bounds for the number and Betti numbers of connected components. In order to get a uniform control on the family ${\{F_L\}}_{L \in \mc{L}}$, we need to add the tame conditions in Definition~\ref{D: Axiom 3}, with which we can obtain the proof of our Theorem~\ref{Thm: Global Limit}.
Notice that, if ${\{F_L\}}_{L \in \mc{L}}$ satisfies $(B4)$ for some $k_2(x)$ at $x \in U$ in Definition~\ref{D: Axiom 3}, then its translation invariant limit $F_x$ also satisfies $(A3)$ in Definition~\ref{D: Axiom 1} for this $k_2(x)$ and for any $R > 1$.

\begin{proof}[Proof of (1) in Theorem~\ref{Thm: Global Limit}]
	Assume that $x \in Q \cap U'$ for some compact subset $Q \subset U$, then $F_x$ satisfies \textit{ $(R;M(Q),k_1(Q),k_2(x))$-assumptions} on $C_{R+1}$ for each $R >0$. By Theorem~\ref{Thm: L^1 Bound on Number},
		\begin{equation}
			\bar{\nu}(x)  =\lim_{R \to \infty} \frac{\mc{E}(N(R;F_x))}{|C_R|} \leq D_1(n,m,M(Q),k_1(Q)).
		\end{equation}
	The measurability is directly if we consider the function
		\begin{equation}
			\nu_{R,L}(x,w) \equiv \frac{N(R;F_{x,L})}{|C_R|},
		\end{equation}
	defined on $U_{-(R+1)/L} \times \Omega'$ for $U_{-r} \equiv \{z \in U \ | \ \dist(z,\pa U) > r\}$ and $\Omega' = \{w \in \Omega \ | \ F_L \in C_* ^1 (\mb{R}^n,\mb{R}^m) \text{ for all } L \}$. Notice that $\mc{P}(\Omega \backslash \Omega') = 0$ by Bulinskaya's lemma. The measurability of $\nu_{R,L}(x,w)$ follows from the measurability 
	of compositions of lower semicontinuous and measurable maps. 
	
	Then, we claim that for any fixed $x \in Q \cap U'$,
		\begin{equation}\label{E: Double Limit in L^1}
			\lim_{R \to \infty} \limsup_{L \to \infty} \mc{E}(|\nu_{R,L}(x, \cdot) - \bar{\nu}(x)|) = 0,
		\end{equation}
	which is a stronger version than Lemma~\ref{L: Double Limit in Probability}. 
    The measurability of $\bar{\nu}(x)$ follows from the Fubini-Tonelli theorem. First, we notice that for each $x \in Q \cap U'$ and each $R >0$, when $L > L_0(R,x)$, Theorem~\ref{Thm: L^q Bound on Number} gives that
		\begin{equation}
			\mc{E}({(\nu_{R,L}(x, \cdot))}^q) \leq D_2
		\end{equation}
	for some $q  = q(n,m)>1$ and $D_2 = D_2(n,m,M(Q),k_1(Q), k_2(x))< \infty$. On the other hand, given any $\epsilon > 0$, Lemma~\ref{L: Double Limit in Probability} tells us that
		\begin{equation}
			\lim_{R \to \infty} \limsup_{L \to \infty} \mc{P}(\Omega_\epsilon) = 0,
		\end{equation}
	where $\Omega_\epsilon \equiv \{ w \in \Omega' \ | \ |\nu_{R,L}(x,w) - \bar{\nu}(x) | > \epsilon\}$. Hence, when $L > L_0(R,x)$, by the H\"{o}lder inequality,
		\begin{equation}
			0 \leq \mc{E}(|\nu_{R,L}(x, \cdot) - \bar{\nu}(x)|) \leq \epsilon + {(\mc{P}(\Omega_\epsilon))}^{1- \frac{1}{q}} \cdot C(n,m,D_1,D_2).
		\end{equation}
	Since $D_1,D_2$ are independent of $R,L$, we can let $L \to \infty$ first, and then let $R \to \infty$, and finally let $\epsilon \to 0$. Hence, we finish the proof of the claim.
\end{proof}
\begin{proof}[Proof of (2) in Theorem~\ref{Thm: Global Limit}]
	The proof follows similar strategies in~\cite{NS16,SW19} when $m=1$, but we still need to modify it by our earlier theorems and lemmas for $m>1$. We include the proof adapting to our settings here. 
    We can assume that $\varphi \geq 0$ and set $Q = \spt(\varphi)$. Fix an arbitrary $\delta \in (0,1)$ such that $Q_{+4\delta} \subset U$ and we denote $Q_1 = Q_{+ \delta}$, $Q_2 = Q_{ + 2\delta}$, where $Q_{+\delta} \equiv \{z \in \mb{R}^n \ | \ \dist(z,Q) \leq \delta \}$.
	For any $x \in Q_1$, let $\varphi_{-}(x) \equiv \inf_{C_\delta(x)} \varphi$ and $\varphi_{+}(x) \equiv \sup_{C_\delta(x)} \varphi$, where we recall that $C_\delta(x)$ is the shift from the open cube $C_\delta(0)$ to the center $x$. Then, for parameters $T, R ,L$, with $1 < T < R < \delta L$, we have that $\varphi_{-}(x) \leq \varphi(y) \leq \varphi_{+}(x)$ for any $y \in C_{R/L}(x)$. 
    Then, we obtain that
		\begin{equation}\label{E: Global Limit Sandwich 1}
			\begin{split}
				\int_{Q_1} \varphi_{-}(x) \nu_{R,L}(x,w) \ dx &\leq \int_{Q_1} \varphi_{-}(x) \frac{n_L(C_{R/L}(x))}{|C_R|} \ dx
				\\	&\leq \int_{Q_1} \int_{C_{R/L}(x)} \frac{\varphi(y)}{|C_R|} \ dn_L(y) \ dx=\frac{1}{L^n} \int_U \varphi(y)  \ dn_L(y)
				\\	&\leq \int_{Q_1} \varphi_{+}(x) \frac{n_L(C_{R/L}(x))}{|C_R|} \ dx.
			\end{split}
		\end{equation}
	Then, we cover $Q_2$ with $C(n) \cdot |Q_2| \cdot {(L/T)}^n$ open cubes with side lengths $ (T/L)$, and we deonte these cubes as $\{C_j\}$. For example, one can consider an approximation of $Q_2$ by dyadic cubes.
    For each component of $Z(F_L) \cap C_{R/L}(x)$ that does not intersect with the boundaries of any $C_j$, it is fully contained in one $C_j$. Hence, it is fully contained in $C_{(R+T)/L}(x)$.
	For those components of $Z(F_L)$ that intersect with at least one $C_j$, the number of them are bounded by $\sum_j \mfk{N}_\# ( C_j; F_L)$. Recall that, by Bulinskaya's lemma and the discussions after the definition of $\mfk{N}_\# (R;G)$ in (\ref{E: Definition of Residue Parts}), for each $j$, it is in probability $1$ that $\mfk{N}_\# ( C_j; F_L)$ is finite and $\mc{E}(\mfk{N}_\# ( C_j; F_L)) \leq C(n,m,M(Q),k_1(Q)) \cdot {(T)}^{n-1}$ when $L$ is large.
	Hence, 
		\begin{equation}\label{E: Global Limit Sandwich 2}
			\begin{split}
				\int_{Q_1} \varphi_{+}(x) \frac{n_L(C_{R/L}(x))}{|C_R|} \ dx &\leq {\bigg(\frac{R+T}{R}\bigg)}^n \cdot \int_{Q_1} \varphi_{+}(x) \nu_{R+T,L}(x,w) \ dx 
				\\	& \quad + (\sup_U \varphi) \cdot L^{-n} \cdot (\sum_j \mfk{N}_\# ( C_j; F_L)).
			\end{split}
		\end{equation}
	For the $\int_U \varphi(x) \bar{\nu}(x) dx$ term, let $\omega_\varphi(\cdot)$ be the modulus of continuity of $\varphi$,
		\begin{equation}
			\begin{split}
				\int_U \varphi(x) \bar{\nu}(x) dx &\geq {\bigg(\frac{R + T}{R} \bigg)}^n \cdot \int_{Q_1} \varphi_{+}(x) \bar{\nu}(x) \ dx 
            \\	& \quad    - C(n) \cdot (T/R) \cdot (\sup_U \varphi) \cdot ||\bar{\nu}||_{L^\infty(Q_1)} \cdot |Q_1|
			     - \omega_\varphi(\delta) \cdot ||\bar{\nu}||_{L^\infty(Q_1)} \cdot |Q_1|.
			\end{split}
		\end{equation}
	Combine with (\ref{E: Global Limit Sandwich 1}) and (\ref{E: Global Limit Sandwich 2}), we see that
		\begin{equation}
			\begin{split}
				& \quad \frac{1}{L^n} \int_U \varphi(y)  \ dn_L(y) - \int_U \varphi(x) \bar{\nu}(x) dx 
            \\  &\leq {\bigg(\frac{R+T}{R}\bigg)}^n \cdot (\sup_U \varphi) \cdot \int_{Q_1} \big|\nu_{R+T,L}(x,w) - \bar{\nu}(x) \big|\ dx 
			\\	& \quad + C(n) \cdot (T/R) \cdot (\sup_U \varphi) \cdot ||\bar{\nu}||_{L^\infty(Q_1)} \cdot |Q_1| + \omega_\varphi(\delta) \cdot ||\bar{\nu}||_{L^\infty(Q_1)} \cdot |Q_1| 
			\\	& \quad + (\sup_U \varphi) \cdot L^{-n} \cdot (\sum_j \mfk{N}_\# ( C_j; F_L)).
			\end{split}
		\end{equation}
	and 
		\begin{equation}
			\begin{split}
				& \quad  \frac{1}{L^n} \int_U \varphi(y)  \ dn_L(y) - \int_U \varphi(x) \bar{\nu}(x) dx 
            \\  &\geq -(\sup_U \varphi) \cdot \int_{Q_1} \big|\nu_{R,L}(x,w) - \bar{\nu}(x) \big|\ dx 
			     - \omega_\varphi (\delta) \cdot ||\bar{\nu}||_{L^\infty(Q_1)} \cdot |Q_1| .
			\end{split}
		\end{equation}
	Hence,
		\begin{equation}\label{E: Global Limit Sandwich 3}
			\begin{split}
				& \quad \mc{E}^* \bigg[ \bigg |  \frac{1}{L^n} \int_U \varphi(x)  dn_L(x) - \int_U \varphi(x) \bar{\nu}(x) dx \bigg | \bigg] 
            \\  &\leq 2^n \cdot (\sup_U \varphi) \cdot \int_{Q_1} \mc{E}\big(\big|\nu_{R+T,L}(x,\cdot) - \bar{\nu}(x) \big| \big)\ dx 
			\\	& \quad + (\sup_U \varphi) \cdot \int_{Q_1} \mc{E} \big(\big|\nu_{R,L}(x,w) - \bar{\nu}(x) \big| \big)\ dx + (\sup_U \varphi) \cdot L^{-n} \cdot (\sum_j \mc{E}\big( \mfk{N}_\# ( C_j; F_L) \big)  )
			\\	& \quad + C(n) \cdot (T/R) \cdot (\sup_U \varphi) \cdot ||\bar{\nu}||_{L^\infty(Q_1)} \cdot |Q_1| + \omega_\varphi(\delta) \cdot ||\bar{\nu}||_{L^\infty(Q_1)} \cdot |Q_1| .
			\end{split}
		\end{equation}
	Notice that, by $(B1)$ and $(B2)$ in Definition~\ref{D: Axiom 2}, for fixed $R>T>1$, there is a uniform $L_1 = L_1(R)>0$, such that when $L> L_1$,
        \begin{equation}
            \mc{E} \big(\big|\nu_{R+T,L}(x,w) - \bar{\nu}(x) \big| \big) \leq  D_1 (n,m,M(Q),k_1(Q)),
        \end{equation}
    for every $x \in Q_1$ by Theorem~\ref{Thm: L^1 Bound on Number}. 
    Hence, the function
		\begin{equation}
			\eta_R(x) \equiv \limsup_{L \to \infty} \mc{E} \big(\big|\nu_{R,L}(x,w) - \bar{\nu}(x) \big| \big)
		\end{equation}
	has the same upper bound $D_1$. Recall that the claim (\ref{E: Double Limit in L^1}) in proving part (1) of Theorem~\ref{Thm: Global Limit} shows that $\lim_{R \to \infty} \eta_R(x) = 0$. Hence, apply $ \limsup_{L \to \infty}$ first and then apply $\lim_{R \to \infty}$, 
	the first two terms in (\ref{E: Global Limit Sandwich 3}) go to $0$ by the dominated convergence theorem. The fourth term will also go to $0$ since $T < R$ is still fixed at this step.
	As previously mentioned, since ${\{F_L\}}_{L \in \mc{L}}$ satisfies $(B1)$ and $(B2)$ in Definition~\ref{D: Axiom 2}, when $L >L_1(R)$, for each $j$, $ \mc{E}(\mfk{N}_\# ( C_j; F_L)) \leq C(n,m,M(Q),k_1(Q)) \cdot {(T)}^{n-1}$. 
	So, the third term is bounded by $(\sup_U \varphi ) \cdot C(n,m,M(Q) , k_1(Q) ) \cdot  T^{-1}$ for any fixed $T  < R < \delta L$. We now let $T \to \infty$ and finally let $\delta \to \infty$, we can see that the right hand side of (\ref{E: Global Limit Sandwich 3}) goes to $0$. 
    Hence, we can finish the proof.
\end{proof}

The proof of Theorem~\ref{Main Results: Global 2} is the same as the proof of Theorem 3 in~\cite{NS16} after we finished Theorem~\ref{Thm: Global Limit} as above. Therefore, we omit the details.

Instead, let us explore the general Betti numbers of tame parametric ensembles ${\{F_L\}}_{L \in \mc{L}}$ defined on an $n$-dimensional closed $C^3$-manifold $X$. 
Recall that we say that a parametric Gaussian ensemble ${\{F_L\}}_{L \in \mc{L}}$ is tame on $X$ if for every $C^3$-chart $\pi: U \subset \mb{R}^n \to X$, ${\{F_L \circ \pi \}}_{L \in \mc{L}}$ is tame on $U$ as in Definition~\ref{D: Axiom 3}.

For each $l = 0, \dots,n-m$, we will use $\beta_l(F_L)$ to denote the summation of $l$-th Betti numbers over $\mb{R}$ of all connected components of $Z(F_L)$ in $X$.
For each $x \in X$, let $F_x$ be the translation invariant local limit of $\{F_L\}$ at $x$. By Theorem~\ref{Main Results: Local Betti}, we get a limiting constant $\nu_{l;F_x}$. We then define a function $\bar{\nu}_l(x) = \nu_{l;F_x}$ on $X$. 
The following theorem also holds true if $X$ is noncompact and without boundary, but one needs to modify the definition of $\beta_l(F_L)$ into the summation of Betti numbers of those connected components that are fully contained in a geodesic ball with a finite radius in $X$.

\begin{theorem}\label{Thm: Global Betti}
    \textit{
        Assume that ${\{F_L\}}_{L \in \mc{L}}$ is tame on $X$ and $X$ is closed. Then, for each $l = 0, \dots, n-m$, the function $x \mapsto \bar{\nu}_l(x)$ is measurable and bounded on $X$, and there is a positive constant $C_l = C_l(n,m,M,k_1)$, such that
        \begin{equation}
            \int_{X} \bar{\nu}_l(x) \ d \vol( x) \leq \liminf_{L \to \infty}\frac{\mc{E}(\beta_l(F_L))}{L^n} \leq \limsup_{L \to \infty}\frac{\mc{E}(\beta_l(F_L))}{L^n} \leq C_l \cdot |X|,
        \end{equation}
    where $|X|$ is the volume of $X$, and the two positive constants $M = M(X)$ and $k_1 = k_1(X)$ are chosen in Definition~\ref{D: Axiom 2}.
    }
\end{theorem}

\begin{proof}[Proof of the lower bound in Theorem~\ref{Thm: Global Betti}]
    The proof for measurability and local boundedness of $\bar{\nu}_l(x)$ is the same as part (1) of Theorem~\ref{Thm: Global Limit} because one can verify these properties on local charts. Since $X$ is closed, $\bar{\nu}_l$ is actually bounded on $X$.
    One thing we can similarly get is the following $L^1$-convergence. Choose a local chart $\pi : U \to X$ with a bounded open set $U \subset \mb{R}^n$. For $x \in \pi(U)$ and $y = \pi^{-1}(x)$, we define a random function
        \begin{equation}
            \nu_{l;R,L}(y,w) \equiv \frac{\beta_l(R;{(F\circ \pi)}_{y,L})}{|C_R|} ,
        \end{equation}
    for $w \in \Omega$. Here ${(F\circ \pi)}_{y,L}$ is the pull-back of $F_L$ from $\pi(U)$ to $U$ and is rescaled, i.e.,
        \begin{equation}
            {(F\circ \pi)}_{y,L}(u) \equiv F_L(\pi(y + L^{-1}u)),
        \end{equation}
    for $u \in \mb{R}^n$ with $y + L^{-1}u \in U$.
    Hence, $\beta_l(R;{(F\circ \pi)}_{y,L}) = \beta_l(C_{R/L}(y); F_L \circ \pi)$.

    Then, using Theorem~\ref{Thm: L^q Bound on Betti}, we can similarly get, as in the claim (\ref{E: Double Limit in L^1}) in proving part (1) of Theorem~\ref{Thm: Global Limit}, that
        \begin{equation}\label{E: Double Limit in L^1 for Betti}
            \lim_{R \to \infty} \limsup_{L \to \infty} \mc{E}(|\nu_{l;R,L}(y, \cdot) - \bar{\nu}_l (x) \sqrt{\det{g_y}}|) = 0.
        \end{equation}
    Here $g_y$ is the metric tensor of $X$ at $ x = \pi(y)$, which is compatible with change of coordinates to add this determinant term here, because we  use $|C_R|$ in standard $\mb{R}^n$ volume to define $\nu_{l;R,L}(y,w) $. For more discussions on change of coordinates, see Section 9 of~\cite{NS16}.

    Recall that $\beta_l(R;{(F\circ \pi)}_{y,L})$ only counts those connected components that are fully contained in the cube $C_R$.
    We set parameters $R,L$ with $1 <R < \delta L$ for some $\delta \in (0,1)$, and let $Q \subset U$ be a compact subset such that $Q_{+ \delta} \subset U$ but $Q_{+2 \delta}  \nsubseteq U$.
    We have that
        \begin{equation}\label{E: Betti Global Lower Bound}
            \begin{split}
                &\quad \int_Q \nu_{l;R,L}(y,w) \ dy =\int_Q \frac{\beta_l(C_{R/L}(y);F_L \circ \pi)}{|C_R|} \ dy
                \\  & = \frac{1}{|C_R|} \int_Q \sum_{\gamma \subset Z(F_L \circ \pi)} \chi_{\gamma \subset C_{R/L}(y)}  \cdot \beta_l(\gamma) \ dy
                \\  & \leq \frac{1}{|C_R|} \sum_{\gamma \subset Z(F_L \circ \pi) \cap U}\ \beta_l(\gamma) \int_Q  \chi_{\gamma \subset C_{R/L}(y)} \ dy
                \\  & \leq \frac{|C_{R/L}|}{|C_R|} \sum_{\gamma \subset Z(F_L \circ \pi) \cap U}\ \beta_l(\gamma)
                 = \frac{\beta_l(U;F_L \circ \pi)}{L^n} = \frac{\beta_l(\pi(U);F_L) }{L^n}.
            \end{split}
        \end{equation}
    Here, $\chi$ is the indicator function, and we use $\beta_l(\pi(U);F_L)$ to denote the summation of $l$-th Betti numbers of all connected components of $Z(F_L)$ fully contained in the open set $\pi(U)$. 
    The last inequality is because
        \begin{equation}
            \int_Q  \chi_{\gamma \subset C_{R/L}(y)} \ dy=  |\{y \in Q \ \vert \ y \in \cap_{z \in \gamma} C_{R/L}(z) \}| \leq |C_{R/L}|.
        \end{equation}
    Then,
        \begin{equation}
            \begin{split}
                \int_{Q} \bar{\nu}_l(\pi(y)) \sqrt{\det{g_y}} \ dy &\leq \int_Q \mc{E}( \nu_{l;R,L}(y,\cdot) )\ dy + \int_Q \mc{E}(|\nu_{l;R,L}(y, \cdot) - \bar{\nu}_l (x) \sqrt{\det{g_y}}|) \ dy.
                \\  &\leq \frac{\mc{E}(\beta_l(\pi(U);F_L))}{L^n} + \int_Q \mc{E}(|\nu_{l;R,L}(y, \cdot) - \bar{\nu}_l (x) \sqrt{\det{g_y}}|) \ dy.
            \end{split}
        \end{equation}
    Use (\ref{E: Double Limit in L^1 for Betti}), and take limits $\lim_{R \to \infty} \liminf_{L \to \infty}$ at both sides. Notice that it is $\limsup_{L \to \infty}$ in (\ref{E: Double Limit in L^1 for Betti}). Hence, we can apply the dominated convergence theorem to the second term of the right hand side, which is dominated by the constant in Theorem~\ref{Thm: L^1 Bound on Betti}.
    We then get that
        \begin{equation}
            \int_{\pi(Q)} \bar{\nu}_l(x) \ d \vol (x) = \int_{Q} \bar{\nu}_l(\pi(y)) \sqrt{\det{g_y}} \ dy \leq \liminf_{L \to \infty} \frac{\mc{E}(\beta_l(\pi(U);F_L))}{L^n}.
        \end{equation}
    Since $\delta$ and $Q$ were chosen arbitrarily, we see that
        \begin{equation}
            \int_{\pi(U)} \bar{\nu}_l(x) \ dx \leq \liminf_{L \to \infty} \frac{\mc{E}(\beta_l(\pi(U);F_L))}{L^n}.
        \end{equation}
    Finally, consider a triangulation of $X$ with finite partitions ${\{\pi_j(U_j)\}}_{j=1} ^{J}$ such that each $(U_j , \pi_j)$ is a local chart on $X$.
    Then, 
        \begin{equation}
            \int_{X} \bar{\nu}_l(x) \ dx = \sum_{j=1}^J \int_{\pi(U_j)} \bar{\nu}_l(x) \ dx  \leq \liminf_{L \to \infty}  \sum_{j=1}^J  \frac{\mc{E}(\beta_l(\pi(U_j);F_L))}{L^n} \leq \liminf_{L \to \infty}   \frac{\mc{E}(\beta_l(X;F_L))}{L^n}.
        \end{equation}
\end{proof}

For the upper bound in Theorem~\ref{Thm: Global Betti}, the methods we used for Theorem~\ref{Thm: Global Limit} do not work, because one cannot omit those giant components that intersect with boundaries of small cubes $\{C_j\}$ now. So, we do not know whether limits exist for Betti numbers cases. In fact, those giant components exist in many known concrete examples.
Readers can see more discussions on page 6 of~\cite{W21} and references therein.

Here, we use the Chern-Lashof inequality in Theorem~\ref{Thm: Chern} to calculate a global upper bound directly.

\begin{proof}[Proof of the upper bound in Theorem~\ref{Thm: Global Betti}]
    We remark that this upper bound result actually holds true for uniformly controllable Gaussian ensembles ${\{F_L\}}_{L \in \mc{L}}$, i.e., for ${\{F_L\}}_{L \in \mc{L} }$ only satisfying Definition~\ref{D: Axiom 2}.
    
    Since $X$ is a $C^3$-manifold, let $\Phi:X \to \mb{R}^{n+q}$ for some $q > 0$ be an isometric embedding, which follows from Nash's embedding theorem. 
    Then, we can regard $X$ as a submanifold of $\mb{R}^{n+q}$ and also regard $Z(F_L) = F_L ^{-1}(0)$ as finitely many closed submanifolds of $\mb{R}^{n+q}$.
    According to Theorem~\ref{Thm: Chern}, similar to the proofs of Theorem~\ref{Thm: L^1 Bound on Number} and Theorem~\ref{Thm: L^1 Bound on Betti} (see (\ref{E: Bound Components Number})), one can get that
        \begin{equation}
            \beta_l(X;F_L) \leq C(n,m) \cdot \int_{Z(F_L)} || \II^{Z(F_L)} (x)||^{n-m} \ d \mc{H}^{n-m}(x).
        \end{equation}
    Here, the second fundamental form is the one for $Z(F_L)$ as submanifolds of $\mb{R}^{n+q}$. Then, by Theorem~\ref{Thm: APP Curvature}, 
        \begin{equation}
            \beta_l(X;F_L) \leq C(n,m) \int_{Z(F_L)} C(n,m,X) + ||A_L ^{-1}||^{\frac{(n-m)}{2}} \cdot || {(\nabla^X)}^2 F_L ||^{(n-m)} \ d \mc{H}^{n-m}(x),
        \end{equation}
    where $C(n,m,X)$ is a constant depending on $X$ and the embedding map $\Phi$, and $A_L$ is an $m \times m$ matrix, whose elements are ${(A_L)}_{\alpha \beta} = \langle (\nabla^X f_{L,\alpha}) , (\nabla^X f_{L,\beta}) \rangle_X$. We also use $\nabla^X$ to denote gradients on $X$, use ${(\nabla^X)}^2$ to denote Hessians on $X$, and use $\langle \cdot , \cdot \rangle_X$ to denote the Riemannian metric tensor on $X$.
    One can use a similar trick in (\ref{E: Bound Mean Curvature}) and then can obtain an upper bound to $||A_L ^{-1}||$ by $|\det A_L | ^{-1} \cdot ||\nabla^X F||^{2(m-1)}$. So, again, we can let the Gaussian field be
        \begin{equation}
            W_L = (\nabla^X F_L , {(\nabla^X)}^2 F_L )
        \end{equation}
    in the Kac-Rice formula, Theorem~\ref{Thm: Kac-Rice}, and also set 
        \begin{equation}
           \begin{split}
            & \quad h(x,W_L(x)) 
            \\ &=  C(n,m,X) + \big|\det A_L (x) \big|^{-\frac{(n-m)}{2}} ||\nabla^X F_L(x)||^{(n-m)(m-1)} {||{(\nabla^X)}^2 F_L(x) ||}^{(n-m)}.
           \end{split}
        \end{equation}
    Hence, 
        \begin{equation}
            \begin{split}
                \mc{E}(\beta_l(X;F_L)) &\leq \mc{E} \bigg(  \int_{Z(F_L)} h(x,W_L(x))\ d \mc{H}^{n-m}(x) \bigg) 
            \\  &= \int_{X} \mc{E} \big[ h(x,W_L(x)) \cdot |\det A_L (x) |^{1/2} \ \big| \ F_L(x) = 0\big] \cdot p_{F_L(x)}(0)\ d \vol (x).
            \end{split}
        \end{equation}
    The term $p_{F_L(x)}(0)$ is bounded by a constant $C = C(n,m,k_1,M)>0$ according to $(B1)$ and $(B2)$ in Definition~\ref{D: Axiom 2} with positive $M = M(X)$ and $k_1 = k_1(X)$.
    So, let us estimate the conditional expectation term at each $x \in X$.
    For a local chart $(U,\pi)$ around $x \in \pi(U)$, after choosing local normal coordinates $(u_1, \dots, u_n)$ around $x = 0$, the $\nabla^X $ and ${(\nabla^X)}^2$ are just the usual gradient and the usual Hessian at $x=0$. Then, we rewrite them with the rescaled Gaussian fields we defined at the beginning of Section~\ref{Section: Local Double Limit and Global Limit}, i.e., $F_{x,L}(u) = F_L \circ \pi (x + L^{-1} u)$, $\nabla^X F_L(x) = L \cdot \nabla F_{x,L}(0)$, and ${(\nabla^X)}^2F_L(x) = L^2 \cdot (\nabla ^2 F_{x,L}(0))$.
    Also, the elements of $A_L$ then become ${(A_L)}_{\alpha \beta} = L^2 \cdot \langle (\nabla {(F_{x,L})}_\alpha (0) ) , (\nabla {(F_{x,L})}_\beta (0) ) \rangle_{\mb{R}^n} \equiv L^2 \cdot {(A_{x,L})}_{\alpha \beta}$, where $\langle \cdot , \cdot \rangle_{\mb{R}^n}$ is the standard inner product in $\mb{R}^n$ formed by coordinates $(u_1, \dots,u_n)$. 
    Then,
        \begin{equation}
            \begin{split}
                & \quad \big|\det A_L (x) \big|^{-\frac{(n-m)}{2}} ||\nabla^X F_L(x)||^{(n-m)(m-1)} ||{(\nabla^X)}^2 F_L(x) ||^{(n-m)}
                \\  &= L^{-m(n-m)}\big|\det A_{x,L} (0) \big|^{-\frac{(n-m)}{2}} \cdot L^{(n-m)(m-1)}||\nabla F_{x,L}(0)||^{(n-m)(m-1)}  
                \\  & \quad \cdot L^{2(n-m)} {||{(\nabla)}^2 F_{x,L}(0) ||}^{(n-m)}
                \\  &= L^{(n-m)}\big|\det A_{x,L} (0) \big|^{-\frac{(n-m)}{2}} ||\nabla F_{x,L}(0)||^{(n-m)(m-1)}  ||{(\nabla)}^2 F_{x,L}(0) ||^{(n-m)}
                \\  &\equiv L^{(n-m)} \cdot P(F_{x,L})(0),
            \end{split}
        \end{equation}
    where $P(F_{x,L})(0)$ denote the long terms in order to simplify the notation in this proof.
    Hence, the conditional expectation term at $x$ becomes 
        \begin{equation}
            \begin{split}
                & \  \quad \mc{E} \big[ h(x,W_L(x)) \cdot |\det A_L (x) |^{1/2} \ \big| \ F_L(x) = 0\big] 
            \\  &= L^m \cdot C(n,m,X)  \cdot \mc{E} \big[ |\det A_{x,L}(0) |^{1/2} \ \big| \ F_{x,L}(0) = 0\big]
            \\  & + L^n \cdot \mc{E}\big[ P(F_{x,L})(0) \cdot |\det A_{x,L}(0) |^{1/2} \ \big| \ F_{x,L}(0) = 0\big] .
            \end{split}
        \end{equation}
    Recall that $F_L \circ \pi$ satisfies $(B1)$ and $(B2)$ in Definition~\ref{D: Axiom 2} with $M = M(X), k_1 = k_1(X)$, when $L$ is large. We then see, from the same proof in Theorem~\ref{Thm: L^1 Bound on Number}, that
        \begin{equation}
            \mc{E} \big[ |\det A_{x,L}(0) |^{1/2} \ \big| \ F_{x,L}(0) = 0\big] \leq C(n,m,M,k_1),
        \end{equation}
    and 
        \begin{equation}
            \mc{E}\big[ P(F_{x,L})(0) \cdot |\det A_{x,L}(0) |^{1/2} \ \big| \ F_{x,L}(0) = 0\big] \leq  C(n,m,M,k_1),
        \end{equation}
    for some positive constants $C(n,m,M,k_1)$. Since $n>m$, the second term with $L^n$ is much larger. 
    Hence,
        \begin{equation}
           \begin{split}
            & \quad \limsup_{L \to \infty}\frac{\mc{E}(\beta_l(F_L))}{L^n}
            \\  & \leq \limsup_{L \to \infty} \frac{C(n,m,M,k_1)}{L^n}  \int_{X} (L^m \cdot C(n,m,X) + L^n) \ d \mc{H}^n(x) = C_l \cdot |X|,
           \end{split}
        \end{equation}
    where $C_l = C_l(n,m,M,k_1)$ is the constant $C(n,m,M,k_1)$, which is independent of the constant $C(n,m,X)$, and is also independent of the arbitrarily chosen embedding $\Phi$.
\end{proof}



\begin{appendix}


    \section{Curvatures Computations}\label{APP: Computation}
    
    In this appendix, we give some preliminary computations for curvatures involved in the two geometric inequalities, Theorem~\ref{Thm: Fenchel} and Theorem~\ref{Thm: Chern}. More precisely, we will compute second fundamental forms and mean curvatures for submanifolds obtained from zero sets.
    
    Let $X \subset \mb{R}^{n+q}$ be an $n$-dimensional $C^3$-manifold embedded in $\mb{R}^{n+q}$ for some $q \geq 0$. A special case is when $q=0$ and $X = \mb{R}^n$. We let $M$ be an $(n-m)$-dimensional manifold embedded in $X$. Hence, $M$ is also a submanifold of $\mb{R}^{n+q}$.
    We assume that $M$ is realized as a regular part of the zero set of a non-degenerate $m$-dimensional $C^2$-vector field $F$ on $X$, i.e., for $F = (f_1, \dots, f_m)$, $M$ is a connected component of $F^{-1}(0) = Z(F)$, and on $M$, $\nabla^X F$, the gradient of $F$ with respect to the Riemannian connection on $X$, is of full rank.
    
    Choose a point $p_0 \in M$, assume that $p = (t_1, \dots, t_{n-m})$ is a local coordinate on $M$ around $p_0 = (0, \dots,0)$. Denote the embedding from $M$ into $X$ as $\varphi$, and choose $x = (x_1, \dots, x_n)$ as a local coordinate on $X$ around $x_0 = \varphi(p_0)$. So, one may write $\varphi = (\varphi^1, \dots, \varphi^n)$.
    We denote the embedding from $X$ into $\mb{R}^{n+q}$ by $\Phi$.
    If we view all these objects as subsets in $\mb{R}^{n+q}$, then $p_0$, $x_0 = \varphi(p_0)$, and $y_0 = \Phi(x_0) = \Phi(\varphi(p_0))$ are actually the same point.
    
    Our aim in this section is to compute the second fundamental form and mean curvature of $M$ as a submanifold of $\mb{R}^{n+q}$ at the point $y_0 = \Phi(\varphi(p_0))$. 
    Since we only consider tensors at a fixed point, for simplicity, let us assume that $(t_1 , \dots, t_{n-m})$ and $(x_1, \dots,x_n)$ are local normal coordinates on $M$ and $X$ at the points $p_0, x_0 = \varphi(p_0)$ respectively.
    That is, at $x_0$,
        \begin{equation}\label{E: APP 0 orthonormal 1}
            \langle \partial_{x_l} \Phi , \partial_{x_k} \Phi \rangle_{\mb{R}^{n+q}} = \delta_{lk},
        \end{equation}
    for $\langle \cdot , \cdot \rangle_{\mb{R}^{n+q}}$ the standard inner product in $\mb{R}^{n+q}$, together with that 
        \begin{equation}\label{E: APP 0 normal coordinates 1}
            \partial^2 _{x_l x_k} \Phi \in T_{y_0} ^{\perp} X, \quad \forall \ l,k = 1, \dots n,
        \end{equation}
    where $T_{y_0} ^{\perp} X$ is the normal vector space of $X$ at $y_0 = \Phi(x_0)$ in $\mb{R}^{n+q}$.
    At $p_0$,
        \begin{equation}\label{E: APP 0 orthonormal 2}
            \sum_{l} (\partial_{t_i} \varphi^l)(\partial_{t_j} \varphi^l) = \delta_{ij},
        \end{equation}
    together with that
        \begin{equation}
            \sum_l (\partial^2 _{t_i t_j}\varphi^l) (\partial_{t_s}\varphi^l) = 0, \quad \forall \ i,j,s = 1, \dots , n-m.
        \end{equation}

    We first consider the second fundamental form $\II = \II^M$ of $M$ as a submanifold of $\mb{R}^{n+q}$ at $y_0 = \Phi(\varphi(p_0))$, which is defined by
        \begin{equation}
            \langle \II_{ij} , \mf{n} \rangle \equiv \bigg\langle \frac{\partial^2(\Phi \circ \varphi)}{\partial t_i \partial t_j }(p_0) , \mf{n} \bigg\rangle_{\mb{R}^{n+q}},
        \end{equation}
    where $\mf{n} \in T_{y_0} ^{\perp} M$ is a unit normal vector of $M$ at $y_0 $ in $\mb{R}^{n+q}$. Let us compute these derivatives.
        \begin{equation}
            \partial_{t_j}(\Phi \circ \varphi) = \sum_{l} (\partial_{x_l} \Phi) (\partial_{t_j} \varphi^{l}),
        \end{equation}
    and
        \begin{equation}
            \partial^2_{t_i t_j} (\Phi \circ \varphi) = \sum_{l,k} (\partial^2 _{x_l x_k}\Phi) (\partial_{t_j}\varphi^l) (\partial_{t_i}\varphi^k) + \sum_l (\partial_{x_l}\Phi) (\partial^2 _{t_i t_j} \varphi^l).
        \end{equation}
    
    If $\mf{n} \in T_{y_0} ^{\perp} X$, then $\mf{n} \perp (\partial_{x_l}\Phi)$. Hence,
        \begin{equation}
            |\langle \II_{ij} , \mf{n} \rangle|  = \bigg| \sum_{l,k} \big\langle (\partial^2 _{x_l x_k}\Phi)  , \mf{n} \big\rangle (\partial_{t_j}\varphi^l) (\partial_{t_i}\varphi^k)  \bigg | \leq ||\II^X(y_0) ||.
        \end{equation}
    The last inequality is because of (\ref{E: APP 0 orthonormal 2}), which says that $\partial_{t_i} \varphi$ and $\partial_{t_j} \varphi$ are two unit vectors. So, the value is bounded by the norm of the matrix, which is $||\II^X(y_0)||$, the norm of the second fundamental form of $X$ as a submanifold of $\mb{R}^{n+q}$ at $y_0$.
    This is a constant only depending on $X$ and the embedding $\Phi$, so we denote an upper bound of this constant by $C(X)>0$.
    
    If $\mf{n} \in T_{y_0}  X \cap T_{y_0} ^{\perp} M$, i.e., $\mf{n}$ is tangent to $X$ but normal to $M$ at $y_0$ in $\mb{R}^{n+q}$, we have that
        \begin{equation}\label{E: APP 0 Estimates 1}
            |\langle \II_{ij} , \mf{n} \rangle|  = \bigg| \sum_{l} \big\langle (\partial _{x_l}\Phi)  , \mf{n} \big\rangle (\partial^2 _{t_i t_j}\varphi^l)   \bigg |,
        \end{equation}
    which is because of (\ref{E: APP 0 normal coordinates 1}), $\partial^2 _{x_l x_k} \Phi \perp \mf{n}$. Now, let us use the fact that $M$ is realized as a part of $Z(F)$ to make more calculations and to estimate (\ref{E: APP 0 Estimates 1}).
    Since $f_\alpha(\varphi) = 0$ for $\alpha = 1, \dots , m$, we differentiate these equations and get that
        \begin{equation}\label{E: APP 0 Function Derivatives 1}
            \sum_l ( \partial_{x_l} f_\alpha) (\partial_{t_j} \varphi^l) = 0,
        \end{equation}
    and 
        \begin{equation}\label{E: APP 0 Function Derivatives 2}
            \sum_{l,k} (\partial^2 _{x_l x_k} f_\alpha)(\partial_{t_j} \varphi^l) (\partial_{t_i} \varphi^k) + \sum_{l} (\partial_{x_l} f_\alpha) (\partial^2 _{t_i t_j} \varphi^l) = 0.
        \end{equation}
    Since $\nabla^X F$ is of full rank on $M$, we see that these $\nabla^X f_\alpha \equiv \sum_l (\partial_{x_l} f_\alpha) (\partial_{x_l} \Phi)$, actually form a basis of $T_{y_0}  X \cap T_{y_0} ^{\perp} M$. Hence, we take an $m \times m$ symmetric matrix $T = (T^{\alpha \beta})$ such that these
        \begin{equation}
            \mf{n}_\alpha \equiv  \sum_{\beta} T^{\alpha \beta} (\nabla^X f_\beta) =\sum_{\beta} T^{\alpha \beta} \sum_l (\partial_{x_l} f_\beta) (\partial_{x_l} \Phi)
        \end{equation}
    are orthonormal, i.e., $\langle \mf{n}_{\alpha_1} , \mf{n}_{\alpha_2}\rangle = \delta_{\alpha_1 \alpha_2}$.
    On the other hand,
        \begin{equation}
            \langle \mf{n}_{\alpha_1} , \mf{n}_{\alpha_2}\rangle = \sum_{\beta_1, \beta_2} T^{\alpha_1 \beta_1} T^{\alpha_2 \beta_2} \sum_l (\partial_{x_l} f_{\beta_1}) (\partial_{x_l} f_{\beta_2}),
        \end{equation}
    where if one uses matrices notations, the right hand side is $T A T^T$, where $T^T$ is the transpose of $T$, and $A$ is the matrix consists of inner products of $\nabla^X f_\beta$, i.e., $A = (A_{\beta_1 \beta_2}) = \big(\sum_l (\partial_{x_l} f_{\beta_1}) (\partial_{x_l} f_{\beta_2})\big) $.
    Hence, 
        \begin{equation}
            T^T T = A^{-1}.
        \end{equation}
    Then, according to (\ref{E: APP 0 orthonormal 1}), (\ref{E: APP 0 Estimates 1}), and (\ref{E: APP 0 Function Derivatives 2}), we can get that
        \begin{equation}
            |\langle \II_{ij} , \mf{n}_\alpha \rangle| = \bigg|\sum_{\beta} T^{\alpha \beta}  \sum_{l}  (\partial_{x_l} f_\beta) (\partial^2 _{t_i t_j}\varphi^l)   \bigg | = \bigg|\sum_{\beta} T^{\alpha \beta}  \sum_{l,k}  (\partial^2 _{x_l x_k} f_\beta) (\partial_{ t_j}\varphi^l) (\partial_{ t_i}\varphi^k)  \bigg |.
        \end{equation}
    So,
        \begin{equation}
            \begin{split}
                    &\sum_\alpha |\langle \II_{ij} , \mf{n}_\alpha \rangle|^2 = \sum_{\alpha} \bigg|\sum_{\beta} T^{\alpha \beta}  \sum_{l,k}  (\partial^2 _{x_l x_k} f_\beta) (\partial_{ t_j}\varphi^l) (\partial_{ t_i}\varphi^k)  \bigg |^2
                \\  &= \sum_{\alpha} \sum_{\beta_1 , \beta_2} T^{\alpha \beta_1} T^{\alpha \beta_2} \bigg[\sum_{l_1,k_1}  (\partial^2 _{x_{l_1} x_{k_1}} f_{\beta_1}) (\partial_{ t_j}\varphi^{l_1}) (\partial_{ t_i}\varphi^{k_1})  \bigg] \bigg[\sum_{l_2,k_2}  (\partial^2 _{x_{l_2} x_{k_2}} f_{\beta_2}) (\partial_{ t_j}\varphi^{l_2}) (\partial_{ t_i}\varphi^{k_2})  \bigg]
                \\  &= \sum_{\beta_1 , \beta_2} {(A^{-1})}_{\beta_1 \beta_2} \bigg[\sum_{l_1,k_1}  (\partial^2 _{x_{l_1} x_{k_1}} f_{\beta_1}) (\partial_{ t_j}\varphi^{l_1}) (\partial_{ t_i}\varphi^{k_1})  \bigg] \bigg[\sum_{l_2,k_2}  (\partial^2 _{x_{l_2} x_{k_2}} f_{\beta_2}) (\partial_{ t_j}\varphi^{l_2}) (\partial_{ t_i}\varphi^{k_2})  \bigg].
            \end{split}
        \end{equation}
    For an upper bound, by (\ref{E: APP 0 orthonormal 2}) again, we notice that those $\partial_{t_j} \varphi$ are unit vectors. Hence,
        \begin{equation}
            \sum_\alpha {|\langle \II_{ij} , \mf{n}_\alpha \rangle|}^2 \leq C(n,m) \cdot ||A^{-1}||\cdot ||{(\nabla^X )}^2 F ||^2,
        \end{equation}
    for some positive constant $C(n,m)$, where ${(\nabla^X )}^2 F$ is the Hessian of $F$ and those $||\cdot||$ are matrices norms. 
    
    We denote $\{\mf{n}_\mu\}$ as an orthonormal basis of $T_{y_0} ^{\perp} X$, then the norm square of the second fundamental form $\II^M$ at $y_0$ is defined by
        \begin{equation}
           ||\II^M(y_0)||^2 \equiv \sum_{i,j} \bigg(\sum_\alpha |\langle \II_{ij} , \mf{n}_\alpha \rangle|^2 + \sum_{\mu} |\langle \II_{ij} , \mf{n}_{\mu} \rangle|^2 \bigg).
        \end{equation}
    We can now summarize our results by the following theorem.
    
    \begin{theorem}\label{Thm: APP Curvature}
        \textit{
            The second fundamental form $\II^M$ of $M$ as a submanifold of $\mb{R}^{n+q}$ satisfies that
            \begin{equation}
                ||\II^{M}||^2\leq C(n,m,X) + C(n,m) \cdot ||A^{-1}||\cdot ||{(\nabla^X )}^2 F ||^2,
            \end{equation}
            where the matrix $A$ consists of inner products of $\nabla^X f_\alpha$ for $\alpha = 1, \dots , m$, and  ${(\nabla^X )}^2 F $ is the Hessian. The positive constant $C(n,m,X)$ depends on $X$ and the embedding $\Phi : X \to \mb{R}^{n+q}$.
            In particular, when $X = \mb{R}^n$, one can let $C(n,m,X)  = 0$.
        }
    \end{theorem}
    
    We remark that the mean curvature $\bH$ at $y_0$ is defined by 
        \begin{equation}
            \langle \bH(y_0) , \mf{n} \rangle \equiv \sum_{i} \langle \II_{ii} , \mf{n} \rangle,
        \end{equation}
    for $\mf{n} \in T_{y_0} ^{\perp} M$.
    Hence,
        \begin{equation}
            ||\bH(y_0)||^2 \equiv \sum_\alpha \bigg|\sum_i  \langle \II_{ii} , \mf{n}_\alpha \rangle \bigg|^2 + \sum_{\mu} \bigg|\sum_i\langle \II_{ii} , \mf{n}_{\mu} \rangle \bigg|^2 \leq C(n,m) \cdot || \II^M(y_0) ||^2 , 
        \end{equation}
    for some positive constant $C(n,m)$.



    \section{Ergodicity in Theorem~\ref{Main Results: Local} and Theorem~\ref{Thm: Euclidean Random Field}}\label{APP: Ergodicity}
    
    In this appendix, we give a sufficient condition to guarantee ergodicity of translations on $(C^1_* (\mb{R}^n ,\mb{R}^{m}), \mfk{B}(C^1_* (\mb{R}^n ,\mb{R}^{m})), \gamma_F )$. Here $\mfk{B}(C^1 _* (\mb{R}^n ,\mb{R}^{m}))$ is  the Borel $\sigma$-algebra generated by all open sets in $C^1 _* (\mb{R}^n ,\mb{R}^{m})$. Recall that $F: \mb{R}^n \to \mb{R}^m$ is a centered stationary Gaussian random field satisfying $(A1)$ and $(A2)$ in Definition~\ref{D: Axiom 1} at the origin $x=0$. We write $F(x) = (f_1(x), \dots,f_m(x))$, and let
        \begin{equation}
            k_{i_1 i_2}(x-y) \equiv K_{i_1 i_2}(x,y) = \mc{E}\big(f_{i_1}(x)f_{i_2}(y)\big), \ i_1,i_2 = 1,2,\dots, m,
            \end{equation}
    be an $m \times m$ matrix. Notice that, since $F$ is stationary, $\mc{E}\big(f_{i_1}(x)f_{i_2}(y)\big) = \mc{E}\big(f_{i_1}(x-y)f_{i_2}(0)\big)$ and $k_{i_2 i_1}(x-y) = \mc{E}\big(f_{i_2}(x)f_{i_1}(y)\big) = \mc{E}\big(f_{i_2}(0)f_{i_1}(y-x)\big) = k_{i_1 i_2}(y-x)$. Hence, $k_{i_2 i_1}(x) = k_{i_1 i_2}(- x)$ for all $x \in \mb{R}^n$.
    
    \begin{theorem}
        \textit{
        If
            \begin{equation}\label{E: APP A Assumption}
                \lim_{R \to \infty} \sum_{1 \leq i_1,i_2 \leq m} \frac{1}{|B_R|} \int_{B_R} {(k_{i_1 i_2}(x))}^2 dx = 0,
            \end{equation}
        then the translations actions of $\mb{R}^n$ on $(C^1_* (\mb{R}^n ,\mb{R}^{m}), \mfk{B}(C^1_* (\mb{R}^n ,\mb{R}^{m})), \gamma_F )$ is ergodic.
        }
    \end{theorem}
    This is the Fomin-Grenander-Maruyama theorem. The proof is inspired by Appendix B of~\cite{NS16} for random functions. Here, we prove it for random fields. 
    
    \begin{proof}
        For an $A \in \mfk{B}(C^1_* (\mb{R}^n ,\mb{R}^{m}))$ satisfying $\gamma_F ( \tau_v (A) \Delta A) = 0$ for every $v \in \mb{R}^n$, we need to show that $\gamma_F(A)$ is either $0$ or $1$. Here $(\tau_v G)(x) \equiv G(x+v)$ for any $v \in \mb{R}^n$ and $G \in C^1_* (\mb{R}^n ,\mb{R}^{m})$.
    
        Notice that $\mfk{B}(C^1_* (\mb{R}^n ,\mb{R}^{m}))$ is generated by evaluations sets $I(x; a ,b)$ for $x \in \mb{R}^n$ and $a \prec b \in \mb{R}^m$ ($a \prec b$ means that $a_i < b_i$ for all $i = 1, \dots m$):
            \begin{equation}
                I(x; a ,b) = \{G \in C^1_* (\mb{R}^n ,\mb{R}^{m}) \ | \ a_i \leq G_i(x) < b_i \text{ for all } i = 1, \dots m \},
            \end{equation}
        where $G_i(x)$ is the $i$-th component of $G(x)$.
        Hence, for any $\epsilon > 0$, we can take finitely many points $x_1 , \dots ,x_k \in \mb{R}^n$ and a Borel set $ B \subset \mb{R}^{mk}$ so that $\gamma_F(A \Delta P) < \epsilon$ for 
            \begin{equation}
                P  = P(x_1 ,\dots ,x_k ; B ) \equiv \{ G \in C^1_* (\mb{R}^n ,\mb{R}^{m}) \ | \ (G(x_1) , \dots, G(x_k)) \in B\}.
            \end{equation}
        We can also assume that the distribution of the Gaussian vector $(F(x_1), \dots, F(x_k))$ is non-degenerate (Otherwise, one can use the maximal linear basis of this Gaussian vector.). So, we can get that
            \begin{equation}
                \gamma_F(P) = {(2\pi)}^{-\frac{mk}{2}} {(\det \Lambda)}^{-\frac{1}{2}} \int_B e^{-\frac{1}{2} \langle \Lambda^{-1}t,  t \rangle} \ dt,
            \end{equation}
        where $\Lambda_{j_1 i_1, j_2 i_2} = \mc{E}(f_{i_1}(x_{j_1}) f_{i_2}(x_{j_2})) = k_{i_1 i _2}(x_{j_1} - x_{j_2})$ is the covariance matrix for the Gaussian vector $(F(x_1), \dots, F(x_k))$ and it is an $mk \times mk$ matrix, and $\langle \cdot , \cdot \rangle$ is the inner product for vectors in $\mb{R}^{mk}$.
    
        Since $\tau_v P = P(x_1 - v , \dots , x_k - v; B)$, we have that
            \begin{equation}
                P \cap \tau_v P = P(x_1, \dots , x_k , x_1 - v , \dots , x_k - v; B \times B).
            \end{equation}
        If we know that $(F(x_1), \dots, F(x_k) , F(x_1 - v) , \dots, F(x_k - v))$ is also non-degenerate, we can write
            \begin{equation}\label{E: APP A 1}
                \gamma_F(P\cap \tau_v P) = {(2\pi)}^{-mk} {(\det \tilde{\Lambda})}^{-\frac{1}{2}} \int_{B \times B} e^{-\frac{1}{2} \tilde{\Lambda}^{-1}t \cdot t } \ dt,
            \end{equation}
        where 
            \begin{equation}\label{E: APP A 2}
                \tilde{\Lambda}=
                \begin{bmatrix}
                \Lambda & \Theta(v )\\
                
                \Theta^T (v) & \Lambda
                \end{bmatrix}
            \end{equation}
        with $ \Theta_{j_1 i_1, j_2 i_2} = \mc{E}(f_{i_1}(x_{j_1}) f_{i_2}(x_{j_2} - v)) = k_{i_1 i _2}(x_{j_1} - x_{j_2} +v)$. We will prove that we can choose a sequence $\{v_l\} \subset \mb{R}^n$ so that $||\Theta(v_l)|| \to 0$ as $l \to \infty$. Hence, we can not only get the nondegeneracy of (\ref{E: APP A 2}) and then get (\ref{E: APP A 1}), 
        but also conclude that 
            \begin{equation}
                \lim_{l \to \infty} \gamma_F(P\cap \tau_{v_l} P) = {(\gamma_F(P))}^2.
            \end{equation}
        And then
            \begin{equation}
                \gamma_F(A) = \limsup_{l \to \infty }\gamma_F(A \cap \tau_{v_l} A) \leq \limsup_{l \to \infty } \gamma_F(P\cap \tau_{v_l} P) + 2 \epsilon \leq {(\gamma_F(A) + \epsilon)}^2 + 2\epsilon.
            \end{equation}
        Let $\epsilon \to 0$, we see that $\gamma_F(A)$ is either $0$ or $1$ since $\gamma_F$ is a probability measure.
    
        Now, the reason that we can choose such a sequence $\{v_l\}$ is because of our assumption. Notice that if we denote $T = \max_{j_1,j_2} | x_{j_1} - x_{j_2}|$, then for any $R >0$, we have that
            \begin{equation}
                \begin{split}
                    \frac{1}{|B_R|} \int_{B_R} ||\Theta(v)||^2  \ dv &\leq \frac{C(m,k)}{|B_R|} \sum_{i_1,i_2,j_1,j_2} \int_{B_R} {(k_{i_1 i _2}(x_{j_1} - x_{j_2} +v))}^2 \ dv
                \\	&\leq \frac{C(m,k)}{|B_R|} \sum_{1 \leq i_1,i_2 \leq m}\int_{B_{R+T}} {(k_{i_1 i_2})}^2(v) \ dv,
                \end{split}
            \end{equation}
        which goes to $0$ as $R \to \infty$ by our assumption. We then choose a sequence $\{v_l\}$ using the mean value theorem.
    \end{proof}
    \begin{remark}
        When $m=1$, by Wiener's lemma, our assumption is equivalent to that the spectral measure, i.e., the Fourier transform, of the covariance function $k(x-y) = \mc{E}(f(x)f(y))$ has no atoms. This is mentioned in Appendix B of~\cite{NS16} in proving this ergodicity theorem when $m=1$. 
        When $m>1$, if we assume that $F$ consists of independent centered stationary Gaussian random functions, i.e., each $f_{i_1}(x)$ is independent of $f_{i_2}(y)$ for $i_1\neq i_2$ and any $x,y \in \mb{R}^n$, and we make the assumption that for each $f_{i_1}$, the corresponding spectral measure $\rho_{i_1}$
        does not have atoms, then we can also get (\ref{E: APP A Assumption}) by Wiener's lemma and hence get our ergodicity theorem. Also, when $m=1$, it is mentioned in~\cite{NS16} that ergodicity is actually equivalent to spectral measures having no atoms, which is actually the full version of the Fomin-Grenander-Maruyama theorem.
    \end{remark}
    
    
    \section{Positive Limiting Constants in Theorem~\ref{Main Results: Local} and Theorem~\ref{Thm: Euclidean Random Field}}\label{APP: Positivity}
    In this appendix, similar to Appendix~\ref{APP: Ergodicity}, we let $F: \mb{R}^n \to \mb{R}^m$ be a centered stationary Gaussian random field satisfying $(A1)$ and $(A2)$ in Definition~\ref{D: Axiom 1} at the origin $x=0$. 
    Our main goal in this appendix is to build up an approximation result under some further assumptions on $F$. For any bounded open set $B \subset \mb{R}^n$, any $G \in C^1(\mb{R}^n, \mb{R}^m)$, and any $\epsilon >0$, we want to show that
        \begin{equation}
            \mc{P}(||F - G||_{C^1(B)} < \epsilon ) >0.
        \end{equation}
    For this purpose, we need to recall some basic facts about Gaussian fields first. One can see Chapter 4 of~\cite{GS04} together with Appendix A of~\cite{NS16} as references to the following discussions. 
    The general strategies were also mentioned in~\cite{CS19,SW19}.
    
    We write $F(x) = (f_1(x), \dots,f_m(x))$, then
    \begin{equation}
        k_{i_1 i_2}(x-y) \equiv K_{i_1 i_2}(x,y) = \mc{E}\big(f_{i_1}(x)f_{i_2}(y)\big), \ i_1,i_2 = 1,2,\dots, m,
        \end{equation}
    is an $m \times m$ matrix. Since $F$ is stationary, $k_{i_2 i_1}(x) = k_{i_1 i_2}(- x)$ for all $x \in \mb{R}^n$. Also, $k$ has a positive property in the following sense. Let $x_1, \dots, x_r \in \mb{R}^n$ and $a_1, \dots, a_r \in \mb{C}$, then 
        \begin{equation}
            \sum_{1 \leq t,s \leq r} k(x_t - x_s)a_t \overline{a_s} \quad \text{is positive semi-definite in } \mr{End}_{\mb{C}}(\mb{R}^m).
        \end{equation}
    This is because for any $\eta = {(\eta_1, \dots, \eta_m)}^T \in \mb{C}^m $,
        \begin{equation}
            \begin{split}
                \eta^T \cdot \bigg(\sum_{1 \leq t,s \leq r} k(x_t - x_s)a_t \overline{a_s} \bigg)\cdot \overline{\eta} &= \sum_{1\leq i_1,i_2 \leq m}\sum_{1 \leq t,s \leq r}\mc{E}\bigg( a_t\eta_{i_1}f_{i_1}(x_t) \cdot \overline{a_s\eta_{i_2}f_{i_2}(x_s)} \bigg) 
                \\         &= \mc{E} \bigg| \sum_{i_1=1}^m \sum_{t = 1}^r a_t\eta_{i_1}f_{i_1}(x_t) \bigg|^2 \geq 0.
            \end{split}
        \end{equation}
    Then, by the Bochner-Herglotz theorem, the Fourier transform of $k$, which we denote as $\rho = (\rho_{i_1 i_2})$, is a finite positive semi-definite $\mr{End}_{\mb{C}}(\mb{R}^m)$-valued measure on $\mb{R}^n$, and satisfies that
        \begin{equation}
            \mc{E}\big(f_{i_1}(x)f_{i_2}(y)\big) = k_{i_1 i_2}(x-y) = \int_{\mb{R}^n} e^{i \langle x-y, \xi \rangle} \ d\rho_{i_1 i_2}(\xi).
        \end{equation}
    Since $k(x) = {(k(-x))}^T$, we see that $\rho = \rho^* \equiv {(\overline{\rho})}^T$ and $\rho(\xi) = {(\rho(-\xi))}^T$. Indeed, $F$ has the following expression.
        \begin{equation}\label{E: APP D Spectral Decomposition}
            F(x) = \int_{\mb{R}^n} e^{i \langle x, \xi \rangle} \ dZ(\xi),
        \end{equation}
    where $Z(\cdot)$ is a complex vector-valued orthogonal random measure, which satisfies that for any Borel sets $U_1, U_2 \subset \mb{R}^n$ with $\rho(U_1)$ and $\rho(U_2)$ finite, $\mc{E}(Z(U_1)) =\mc{E}(Z(U_2)) = 0$ and $\mc{E}(Z_{i_1}(U_1) \overline{Z_{i_2}(U_2)}) = \rho_{i_1 i_2}(U_1 \cap U_2)$ for $i_1,i_2 = 1, \dots, m$.
    Notice that this also implies that $|\rho_{i_1 i_2}(U_1 \cap U_2)|^2 \leq |\rho_{i_1 i_1}(U_1)| \cdot |\rho_{i_2 i_2}(U_2)|$ by the Cauchy-Schwarz inequality. Hence, for any Borel set $U \subset \mb{R}^n$, if $\rho_{i_1 i_2}(U) \neq 0$, then $\rho_{i_1 i_1}(U) \neq 0$ and $\rho_{i_2 i_2}(U) \neq 0$.

    We define the space of square-summable Hermitian fields as 
        \begin{equation}
           L^2 _H (\rho) \equiv \{G:\mb{R}^n \to \mb{C}^m \ \vert \  G \in L^2(\rho), \text{ and } G(x) = \overline{G(-x)} \text{ for all }x \in \mb{R}^n \},
        \end{equation}
    where for $G = {(g_1, \dots,g_m)}^T\in L^2(\rho)$,
        \begin{equation}
            ||G||^2 _{L^2(\rho)} \equiv  \sum_{1 \leq i_1, i_2 \leq m} \int_{\mb{R}^n} g_{i_1}(\xi) \overline{g_{i_2}(\xi)} \ d\rho_{i_1 i_2}(\xi) < \infty.
        \end{equation}
    We also define the reproducing kernel Hilbert spaces associated with $\rho$ (or equivalently, associated with $k$), which are
        \begin{equation}
            \mc{H}(\rho) \equiv \mc{F}_\rho (L^2  (\rho)), \quad \mc{H}_0(\rho) \equiv \mc{F}_\rho (L^2 _H (\rho)),
        \end{equation}
    where for $G= {(g_1, \dots,g_m)}^T \in L^2 (\rho)$ and for each $i_2 = 1, \dots,m$,
        \begin{equation}\label{E: APP D Fourier}
            \hat{g}_{i_2}(x) \equiv {(\mc{F}_\rho(G)(x))}_{i_2} = \sum_{i_1 = 1}^m \int_{\mb{R}^n} e^{-i \langle x, \xi \rangle} g_{i_1}(\xi) \ d\rho_{i_1 i_2}(\xi).
        \end{equation}
    We denote $\hat{G} = \mc{F}_{\rho}(G) ={( \hat{g}_1, \dots,\hat{g}_m)}^T$. The inner product of $\mc{H}(\rho)$ is defined by 
        \begin{equation}
            \langle \widehat{G_1} , \widehat{G_2}\rangle_{\mc{H}(\rho)}  \equiv \langle G_1 , G_2 \rangle_{L^2(\rho)} .
        \end{equation}
    
    Define a map $\Phi: L^2(\Omega, \mfk{S}, \mc{P}) \to C^1(\mb{R}^n , \mb{R}^m)$, $h \mapsto \Phi[h](x)$, and for each $t = 1 , \dots m$, the $t$-th component of $\Phi[h]$, ${(\Phi[h])}_t$, is given by $\mc{E}( h \overline{f_t(x)})$, the covariance of $h$ and $f_t(x)$. 
    Note that although our random variables in $L^2(\Omega, \mfk{S}, \mc{P}) $ and $F$ are real-valued, we still write $\overline{(\cdot)}$ in order to be consistent.
    Then, the image of $\Phi$, $\mc{H}(F) \equiv \Phi(L^2(\Omega, \mfk{S}, \mc{P}))$, is a Hilbert space with the induced inner product
        \begin{equation}
            \langle \Phi[h_1], \Phi[h_2] \rangle_{\mc{H}(F)} \equiv \mc{E}(h_1 \overline{h_2}) = \langle h_1, h_2 \rangle_{L^2(\Omega , \mfk{S}, \mc{P})},
        \end{equation}
    for any $h_1, h_2 \in H(F)$, where $H(F) \subset L^2(\Omega, \mfk{S}, \mc{P})$ is the closed $\mb{R}$-linear span of all $(F(x) \cdot w)$ for $x \in \mb{R}^n$ and $w \in \mb{R}^m$, which is exactly the orthogonal complement of $\mr{Ker}(\Phi)$. Let $\{h_j\}$ be a countable orthonormal basis of $H(F)$ and set $e_j = \Phi[h_j] \in C^1(\mb{R}^n , \mb{R}^m)$, we see that for any fixed $w \in \mb{R}^m$, and fixed $x \in \mb{R}^n$,
        \begin{equation}
            F(x) \cdot w = \sum_{j} \mc{E}\big((F(x) \cdot w) h_j\big)h_j = \sum_{j} \big( \Phi[h_j](x) \cdot w\big) h_j,
        \end{equation} 
    and equivalently,
    \begin{equation}\label{E: APP D Expansion of F}
        F(x)  = \sum_{j}  e_j(x)  h_j.
    \end{equation} 
    For any $\tilde{G} \in \mc{H}(F)$, we can write a convergent series in $\mc{H}(F)$, 
        \begin{equation}\label{E: APP D Expansion of H(F)}
            \tilde{G}(x) = \sum_{j} e_j(x) \langle \tilde{G}, e_j \rangle_{\mc{H}(F)} .
        \end{equation}
    Since we have the assumption $(A2)$ in Definition~\ref{D: Axiom 1} that $F$ is a $C^{3-}$-smooth Gaussian field, the convergence (\ref{E: APP D Expansion of H(F)}) is, at least, in locally $C^1$-sense. Then, we can get that for any $\epsilon > 0$ and any bounded open set $B \subset \mb{R}^n$,
        \begin{equation}\label{E: APP D Positive Approximation}
            \mc{P}(||F-\tilde{G}||_{C^1(B)}>\epsilon)>0.
        \end{equation}
    Hence, we need to show that $\mc{H}(F)$ is $C^1$-dense in $C^1(\mb{R}^n , \mb{R}^m)$ under some assumptions.
    
    On the other hand, there is a way to link the inner products $\langle \cdot , \cdot \rangle_{\mc{H}(\rho)}$ and $\langle \cdot , \cdot \rangle_{\mc{H}(F)}$ by the following isometry map, which will show that these two spaces, $\mc{H}_0(\rho)$ and $\mc{H}(F)$, and their inner products are actually the same. For any $G = {(g_1 , \dots, g_m)}^T \in L^2 _H(\rho)$, let 
        \begin{equation}\label{E: APP D Isometry Map}
            I(G) \equiv \int_{\mb{R}^n} G(\xi) \cdot \ dZ(\xi) = \sum_{i = 1}^m \int_{\mb{R}^n} g_i(\xi) \ dZ_i(\xi),
        \end{equation}
    then
        \begin{equation}
            \begin{split}
                ||I(G)||^2 _{L^2(\Omega , \mfk{S}, \mc{P})} &= \mc{E} \bigg(  \int_{\mb{R}^n} G(\xi_1) \cdot \ dZ(\xi_1) \int_{\mb{R}^n} \overline{G(\xi_2)} \cdot \ d\overline{Z(\xi_2)} \bigg) 
                \\  &= \sum_{1 \leq i_1, i_2 \leq m} \int_{\mb{R}^n} g_{i_1}(\xi) \overline{g_{i_2}(\xi)} \ d\rho_{i_1 i_2}(\xi) = ||G||^2 _{L^2(\rho)}.
            \end{split}
        \end{equation}
    Notice that $H(F) \subset I(L^2 _H(\rho))$. Actually, one can even show that $H(F) = I(L^2 _H (\rho))$.  This is because ${\{e^{i \langle x , \xi \rangle}\}}_{x \in \mb{R}^n}$ is a dense family as $\xi$-functions in $L^2_{\rho_{tt}}(\mb{R}^n)$ for each $t = 1, \dots,m$, which means that the linear combinations of functions in this family can approximate smooth functions with compact support. One then compares (\ref{E: APP D Spectral Decomposition}) and (\ref{E: APP D Isometry Map}).
    Also, one can check that 
        \begin{equation}
            \begin{split}
                {(\Phi[I(G)](x))}_{i_2} &= \mc{E}\big(I(G) \overline{f_{i_2}(x)} \big) = \mc{E} \bigg( \int_{\mb{R}^n} G(y) \cdot \ dZ(y) \int_{\mb{R}^n} e^{-i \langle x , \xi \rangle}\ d\overline{Z_{i_2}(\xi)} \bigg) 
            \\ &= \sum_{i_1 = 1}^m \int_{\mb{R}^n} e^{-i \langle x, \xi \rangle} g_{i_1}(\xi) \ d\rho_{i_1 i_2}(\xi) = {(\mc{F}_\rho(G)(x))}_{i_2},
            \end{split}
        \end{equation}
    and then
        \begin{equation}
            \begin{split}
                \langle \widehat{G_1} , \widehat{G_2}\rangle_{\mc{H}(\rho)} &= \langle G_1,G_2 \rangle_{L^2(\rho)} = \langle I(G_1), I(G_2) \rangle_{L^2(\Omega , \mfk{S}, \mc{P})}
                \\  & =  \langle \Phi[I(G_1)], \Phi[I(G_1)] \rangle_{\mc{H}(F)} = \langle \widehat{G_1} , \widehat{G_2}\rangle_{\mc{H}(F)}.
            \end{split}
        \end{equation}
    For similar reasons, one can show that $\mc{H}(F) = \mc{H}_0(\rho) $.

    We define a subset $\Sigma_\rho$ of $C^1(\mb{R}^n, \mb{R}^m)$ as follows. For each element $ H(x) \in \Sigma_\rho$, write $H(x)$ as ${(h_1(x), \dots, h_m(x))}^T $, and for each $t = 1, \dots,m$,  $h_t(x)$ has the form
        \begin{equation}\label{E: APP D Discrete Sum Set Sigma}
            h_t(x) = \sum_{p = 1} ^{q} \sum_{s=1} ^m \bigg( e^{-i \langle x, \xi_p \rangle} g_s(\xi_p) \rho_{st}(V_p )  +  e^{i \langle x, \xi_p \rangle} g_s(-\xi_p) \rho_{st}(-V_p) \bigg),
        \end{equation}
    where $V_p $ are Borel sets, $\xi_p \in V_p $, and $G = {(g_1, \dots,g_m)}^T \in L^2 _H (\rho)$. We emphasize that these $V_p, \xi_p$ are the same for diffrent $t = 1, \dots,m$. This $H(x)$ is actually a Riemann sum of $\mc{F}_\rho (G)$ in (\ref{E: APP D Fourier}) if $G$ is also continuous. Also, notice that $h_t(x)$ is real, because $g_s(-\xi_p ) = \overline{g_s(\xi_p )}$ and $\rho_{st}(-V_p )  = \overline{\rho_{st}(V_p ) }$.
    
    Our aim is to prove that under some assumptions, one can show that $\Sigma_\rho$, and hence $\mc{H}_0(\rho)$, is $C^1$-dense in $C^1(B, \mb{R}^m)$ for any bounded open set $B \subset \mb{R}^n$. Then, we can use (\ref{E: APP D Positive Approximation}) to get an approximation by $F(x)$ with positive probability. This strategy was sketched in Lemma 5.5 and Proposition 5.2 in~\cite{SW19}.
    
    We call that the $\mb{C}^m$-valued orthogonal random measure $Z(\cdot) = {(Z_1(\cdot), \dots , Z_m(\cdot))}^T$ in (\ref{E: APP D Spectral Decomposition}) is \textit{ $m$-balls absolute non-degenerate} if there are $m$ open balls, ${\{ B_{r_t}(y_t) \}}_{t=1} ^m$, such that for each $t = 1, \dots , m$, and for any nonempty open subset $U \subset B_{r_t}(y_t)$, 
        \begin{equation}
            Z_t(U) \notin \mr{Span}_{\mb{C}} \{ Z_1(U), \dots, Z_{t-1}(U), Z_{t+1}(U), \dots , Z_m(U) \},
        \end{equation}
    which is equivalent to saying that there are complex constants $c_{t1} (U), \dots, c_{tm} (U)$, such that
        \begin{equation}\label{E: APP D m nondegenerate constants}
            \mc{E}\bigg( (\sum_{i=1}^m c_{ti} (U) Z_i(U)) \ \overline{Z_s(U)} \bigg) = \sum_{i=1}^m c_{ti} (U) \rho_{is}(U)= \delta_{ts}.
        \end{equation}
    This also implies that $B_{r_t}(y_t)$ is in the interior of the compact support of $\rho_{tt}$.
    Morover, we requires that there are absolute positive constants $C_1(Z)$ and $C_2(Z)$ such that for any open subset $V \subset \cup_i B_{r_i}(y_i)$, $|\rho_{st}|(V) \leq C_1(Z) \cdot |\rho_{st}(V)|$ for any $s,t$; and for any $s,t$ and $U \subset B_{r_t}(y_t)$, $ \sum_{i=1}^m |c_{ti} (U) \rho_{is}(U)| \leq C_2(Z)$.
    
    \begin{example}\label{Example: APP D 1}
        If the random field $F$ has independent components, i.e., $Z_s(\cdot)$ is independent of $Z_t(\cdot)$ when $s \neq t$, this $m$-balls absolute non-degenerate condition is equivalent to saying that for each $t$, the interior of the compact support of $\rho_{tt}$ is nonempty. In this case, each $\rho_{tt}$ is also a nonnegative measure. One can choose $B_{r_t}(y_t)$ contained in the interior of the compact support of $\rho_{tt}$, and let $c_{ti}(U) = \delta_{ti} \cdot {\rho_{tt}(U)}^{-1}$ for all $t$ and $i$.
        Then, choose $C_1(Z)=1$ and $C_2(Z)=1$.

    \end{example}
    
    In general, the constant $C_2(Z)$ is quite tricky to obtain. One may use the following weaker but more explicit assumption.
    
    \begin{example}\label{Example: APP D 2}
        Assume that there is an open ball $B \subset \mb{R}^n$, and a positive constant $\lambda$, such that for any open subset $U \subset B$,
            \begin{equation}
                \lambda^{-1} \cdot |U| \cdot \mr{Id}_m \leq \rho (U) \leq \lambda \cdot |U| \cdot \mr{Id}_m,
            \end{equation}
        where $\rho(U)$ is the $m \times m $ matrix with elements $\rho_{i_1 i_2}(U)$. For the notation $A \geq B$, we mean that the smallest eigenvalue of $A-B$ is nonnegative. And $|U|$ is the volume of $U$ in $\mb{R}^n$. We can also replace $|\cdot|$ with any other measure that is mutually absolutely continuous with respect to $|\cdot|$.
        With this nondegeneracy assumption, we can choose $c(U)$ as the inverse of $\rho(U)$ and set a $C_2(Z)$ only depending on $\lambda$. And if we also have a $C_1(Z)$ satisfying the definition, then this random measure $Z(\cdot)$ satisfies this \textit{ $m$-balls absolute non-degenerate} assumption.	
    \end{example}

    According to the above definitions and examples, assume that we can get two universal positive constants $C_1(Z)$  and $C_2(Z)$ for the \textit{ $m$-balls absolute non-degenerate} condition, we then have the following theorem.

    \begin{theorem}\label{Thm: APP D}
        \textit{
            If $Z(\cdot)$ is \textit{ $m$-balls absolute non-degenerate}, then $\mc{H}_0(\rho)$ is $C^1$-dense in $C^1(B, \mb{R}^m)$ for any bounded open set $B \subset \mb{R}^n$.
        }
    \end{theorem}
    \begin{proof}
        By the assumption, we denote $B_{r_t}(y_t)$ by $V^t$. We assume that ${\{V^t\}}_{t=1} ^m$ are disjoint, otherwise we choose smaller open balls in them.
        Recall that $\rho(x) = \rho^*(x) = {(\rho(-x))}^T$. We can also assume that $V^t \cap ( -V^t) = \emptyset$, otherwise we choose smaller balls again. 
        The following strategies have also been sketched in Lemma 5.5 and Proposition 5.2 in~\cite{SW19} when $m=1$.
    
        Fix any $p(x) = {(p_1(x), \dots, p_m(x))}^T$ such that for each $t = 1, \dots , m$, $p_t(x)$ is a real polynomial in $x_1, \dots, x_n$. Then, for any $\epsilon >0$ and any bounded open set $B$, we need to find a $G \in L^2 _H (\rho)$ such that
            \begin{equation}
                ||{(\mc{F}_\rho(G)(x))}_t - p_t(x) ||_{C^1(B)} <\epsilon.
            \end{equation}
        We fix such $\epsilon$ and $B$ and then proceed.
    
        Let $\phi(\xi)$ be a smooth, nonnegative function supported in $B_1(0)$ such that  $\int_{\mb{R}^n} \phi(\xi) \ d\xi = 1$ and $\phi(\xi_1) = \phi(\xi_2)$ if $||\xi_1|| = ||\xi_2||$. For any $\delta_t \in (0,1)$ and any multi-index $\beta_t$, we see that the function
            \begin{equation}
                \phi_{\delta_t,\beta_t}(\xi) \equiv \frac{\partial^{\beta_t} }{\partial \xi^{\beta_t} } \bigg( \frac{1}{{(r_t \delta_t)}^n} \phi\bigg(\frac{\xi - y_t}{r_t \delta_t}\bigg)\bigg)
            \end{equation}
        is supported in $V^t = B_{r_t}(y_t)$. Now, we do the Fourier transform on ${(-i)}^{\beta_t} \phi_{\delta_t,\beta_t}(\xi)$, and then get a function
        \begin{equation}\label{E: APP D Approximate Polynomials}
            \begin{split}
                \tilde{g}_{\delta_t, \beta_t}(x) &\equiv {(-i)}^{\beta_t} \cdot \int_{\mb{R}^n} (\phi_{\delta_t,\beta_t}(\xi) ) e^{-i\langle x, \xi \rangle} \ d\xi
                \\  &= x^{\beta_t} \cdot \int_{\mb{R}^n } \frac{1}{{(r_t \delta_t)}^n} \phi\bigg(\frac{\xi - y_t}{r_t \delta_t}\bigg)  e^{-i \langle x , \xi \rangle } \ d\xi
                \\  & \to x^{\beta_t} \cdot e^{-i \langle x , y_t \rangle },
            \end{split}
        \end{equation}
        as $\delta_t \to 0^+$. Notice that, for any bounded open set $B \subset \mb{R}^n$, the convergence of (\ref{E: APP D Approximate Polynomials}) is in $C^1(B , \mb{C})$ as functions in $x$.
        Here, for $g \in C^1(B , \mb{C})$, we mean that $g = u + iv$ for $u,v \in C^1(B,\mb{R})$, and the $C^1$-norm is induced by $C^1$-norms on $u,v$ components.
    
        We first consider a polynomial with complex coefficients, which we denote as $q_\epsilon(x)$, such that $|| q_\epsilon(x) - e^{i \langle x,y_t \rangle}||_{C^1(B)} \leq \epsilon$.
        Then, by (\ref{E: APP D Approximate Polynomials}), one can find a smooth complex-valued function $\phi_t(\xi)$ supported in $V^t$ such that the Fourier transform of $\phi_t$, i.e., $\mc{F}(\phi_t)(x) \equiv \int_{\mb{R}^n} \phi_t(\xi) e^{-i\langle x, \xi \rangle} \ d\xi$, satisfies that 
            \begin{equation}
                ||\mc{F}(\phi_t)(x) - p_t(x) q_\epsilon(x)   e^{-i \langle x , y_t \rangle }||_{C^1(B)} \leq \epsilon.
            \end{equation}
        Hence,
            \begin{equation}\label{E: APP D Close to Polynomials}
                ||\mc{F}(\phi_t)(x) - p_t(x)||_{C^1(B)} \leq C(||p_t||_{C^1(B)}, ||y_t|| ) \cdot \epsilon,
            \end{equation}
        where  $C(||p_t||_{C^1(B)}, ||y_t|| )$ is a positive constant depending on the $C^1$-norm of $p_t$ on $B$ and the length of $y_t$.
        We then choose $\varphi_t(\xi) = (\phi_t(\xi) + \overline{\phi_t(-\xi)})/2$, whose Fourier transform is the real part of $\mc{F}(\phi_t)(x)$ and (\ref{E: APP D Close to Polynomials}) still holds for $\mc{F}(\varphi_t)(x)$. Notice that $\varphi_t(-\xi) = \overline{\varphi_t(\xi)}$ for any $\xi \in \mb{R}^n$ and $\varphi_t(\xi)$ is supported in $V^t \cup (-V^t)$.
    
        Now, we approximate $\mc{F}(\varphi_t)(x)$ in $C^1(B)$ by a Riemann sum. That is, for the chosen $\epsilon>0$, we have another function 
            \begin{equation}
                R(\varphi_t)(x) = \frac{1}{2}\sum_{p=1} ^{q_t} \bigg( e^{-i \langle x, \xi_p ^t \rangle} \phi_t(\xi_p ^t) |V_p ^t|  +  e^{i \langle x, \xi_p ^t \rangle} \overline{\phi_t(\xi_p ^t)} |V_p ^t|  \bigg),
            \end{equation}
        such that $||\mc{F}(\varphi_t)(x) - R(\varphi_t)(x)||_{C^1(B)} \leq \epsilon$, where ${\{V_p ^t\}}_{p=1} ^{q_t}$ is a disjoint open subdivision of $V^t$ and each $\xi_p ^t$ is a point in $V_p ^t$. Also, this subdivision is so narrow such that for any $\xi \in V_p ^t$ and any $x \in B$, 
            \begin{equation}\label{E: APP D Small Division 1}
                |e^{-i \langle x , \xi_p ^t \rangle} - e^{-i \langle x , \xi \rangle}| \leq \epsilon \cdot {\bigg( \sum_{s=1}^m ||\phi_s||_{L^1(\mb{R}^n)} \bigg)}^{-1},
            \end{equation}
            \begin{equation}\label{E: APP D Small Division 2}
                \  |\xi_ p e^{-i \langle x , \xi_p ^t \rangle} - \xi e^{-i \langle x , \xi \rangle}| \leq \epsilon \cdot {\bigg( \sum_{s=1}^m ||\phi_s||_{L^1(\mb{R}^n)}  \bigg)}^{-1},
            \end{equation}
        and
            \begin{equation}
                \sum_{l=1}^m \sum_{V_p ^l}   \cdot | \phi_l(\xi_p ^l) | \cdot | V_p ^l| \leq 2 \cdot \bigg(  \sum_{l=1}^m ||\phi_l||_{L^1(\mb{R}^n)} \bigg).
            \end{equation}
        We then collect all $\{V_p ^t\}$ for all $t = 1, \dots ,m$, and get a family of disjoint open sets $\{V_p\}$ and their corresponding $\{\xi_p\}$. Notice that $\phi_t(\xi_p ^s) = 0$ for $s \neq t$. We denote $q = \sum_{t=1} ^m q_t$.
        
        Compare $R(\varphi_t)(x)$ and functions in $\Sigma_\rho$ of the form (\ref{E: APP D Discrete Sum Set Sigma}), we consider a function $G = {(g_1, \dots, g_m)}^T \in L^2 _H (\rho)$ such that $G(x) = 0$ if $x$ is not contained in any $V_p$ or $-V_p$.
        And for each $g_s$ and each $V_p ^t$, let $g_s$ be the constant $c_{ts}(V_p ^t) (\phi_t(\xi_p ^t) |V_p ^t| )$ on $V_p ^t$ and  $\overline{c_{ts} (V_p ^t) (\phi_t(\xi_p ^t) |V_p ^t| )}$ on $(-V_p ^t)$, where the constants $c_{ts} (V_p ^t)$ are from (\ref{E: APP D m nondegenerate constants}). This $G$ is well-defined and in $ L^2 _H (\rho)$.
    
        By the transform (\ref{E: APP D Discrete Sum Set Sigma}), we see that for each $t = 1, \dots, m$,
            \begin{equation}
                \begin{split}
                    &{(H(G)(x))}_t \equiv \sum_{p = 1} ^{q} \sum_{s=1} ^m \bigg( e^{-i \langle x, \xi_p \rangle} g_s(\xi_p) \rho_{st}(V_p )  +  e^{i \langle x, \xi_p \rangle} g_s(-\xi_p) \rho_{st}(-V_p) \bigg)
                    \\  & = 2 \mr{Re} \sum_{l=1}^m \sum_{V_p ^l} \sum_{s=1} ^m    e^{-i \langle x, \xi_p ^l \rangle} g_s(\xi_p ^l) \rho_{st}(V_p ^l) 
                    \\  & = 2 \mr{Re} \sum_{l=1}^m \sum_{V_p ^l} \sum_{s=1} ^m    e^{-i \langle x, \xi_p ^l \rangle} c_{ls}(V_p ^l) \rho_{st}(V_p ^l) (\phi_l(\xi_p ^l) |V_p ^l| ) 
                    \\  & = 2 \mr{Re} \sum_{l=1}^m \sum_{V_p ^l}    e^{-i \langle x, \xi_p ^l \rangle} \delta_{lt} (\phi_l(\xi_p ^l) |V_p ^l| ) 
                    \\  & = 2 \mr{Re} \sum_{V_p ^t}    e^{-i \langle x, \xi_p ^t \rangle}  (\phi_t(\xi_p ^t) |V_p ^t| ) 
                       = 2 R(\varphi_t)(x),
                \end{split}
            \end{equation}
        where $\mr{Re}(\cdot)$ means the real part of a complex number. Hence, we see that $R(\varphi_t)(x)$ is actually expressed by $G/2$ in the transform (\ref{E: APP D Discrete Sum Set Sigma}). 
        
        Then, we compare $H(G)$ and the $\rho$-Fourier transform of $G$. Recall that the subdivision $\{V_p\}$ satisfies that for any $\xi \in V_p$, both $|e^{-i \langle x , \xi_p  \rangle} - e^{-i \langle x , \xi \rangle}|$ and $|\xi_p e^{-i \langle x , \xi_p \rangle} - \xi e^{-i \langle x , \xi \rangle}|$ are bounded by the chosen $\epsilon$ and the constant related to $\phi_l$ in (\ref{E: APP D Small Division 1}) and (\ref{E: APP D Small Division 2}).
        The absolute non-degeneracy of $Z(\cdot)$ assumption is that $|\rho_{st}|(U) \leq C_1(Z)\cdot |\rho_{st}(U)|$ for any $s,t$. Hence, by this choice of the subdivision, and notice that $G$ is a constant on each $V_p$, we then have that
            \begin{equation}
                \begin{split}
                    & \qquad \frac{1}{2}\big| {(\mc{F}_\rho(G)(x))}_{t} - {(H(G)(x) )}_t \big| 
                    =\bigg| \mr{Re}\sum_{p=1}^q \sum_{s = 1}^m \int_{V_p} e^{-i \langle x, \xi \rangle} g_{s}(\xi) -  e^{-i \langle x, \xi_p \rangle} g_{s}(\xi_p)\ d\rho_{st}(\xi) \bigg|
                    \\  &= \bigg| \mr{Re}\sum_{p=1}^q \sum_{s = 1}^m g_s(\xi_p) \int_{V_p} e^{-i \langle x, \xi \rangle} -  e^{-i \langle x, \xi_p \rangle} \ d\rho_{st}(\xi) \bigg|
                    \\  &\leq  \epsilon \cdot {\bigg( \sum_{s=1}^m ||\phi_s||_{L^1(\mb{R}^n)} \bigg)}^{-1} \cdot\sum_{p=1}^q \sum_{s = 1}^m |g_s(\xi_p)| |\rho_{st}|(V_p)
                    \\  &\leq  \epsilon \cdot {\bigg( \sum_{s=1}^m ||\phi_s||_{L^1(\mb{R}^n)} \bigg)}^{-1}\cdot C_1(Z) \sum_{p=1}^q \sum_{s = 1}^m |g_s(\xi_p) \rho_{st}(V_p) | 
                    \\  &= \epsilon \cdot {\bigg( \sum_{s=1}^m ||\phi_s||_{L^1(\mb{R}^n)} \bigg)}^{-1}\cdot C_1(Z) \sum_{l=1}^m \sum_{V_p ^l} \sum_{s = 1}^m |c_{ls}(V_p ^l) \rho_{st}(V_p ^l) |  \cdot | \phi_l(\xi_p ^l) | \cdot | V_p ^l| 
                    \\  &\leq \epsilon \cdot {\bigg( \sum_{s=1}^m ||\phi_s||_{L^1(\mb{R}^n)} \bigg)}^{-1} \cdot  C_1(Z)C_2(Z) \sum_{l=1}^m \sum_{V_p ^l}   \cdot | \phi_l(\xi_p ^l) | \cdot | V_p ^l| 
                    \\  &\leq 2\epsilon \cdot C_1(Z)C_2(Z).
                \end{split}
            \end{equation}
        One can do similar estimates on $\big|\nabla[{(\mc{F}_\rho(G)(x))}_{t} - {(H(G)(x) )}_t] \big|$. Hence, there is a positive constant $C(C_1(Z),C_2(Z))$, such that
            \begin{equation}
                ||{(\mc{F}_\rho(G)(x))}_t - {(H(G)(x) )}_t  ||_{C^1(B)} \leq C\big( C_1(Z),C_2(Z) \big) \cdot \epsilon.
            \end{equation}        
        Then,
            \begin{equation}
                ||{(\mc{F}_\rho(G)(x))}_t - p_t(x) ||_{C^1(B)} \leq C\big( ||p_t||_{C^1(B)} , |y_t|,C_1(Z),C_2(Z) \big) \cdot \epsilon.
            \end{equation}
        Since any map in $C^1(B,\mb{R}^m)$ can be $C^1$-approximated by polynomial vectors $p(x)$, we finish the proof.
    \end{proof}

    Now, we can prove the positivity of the limiting constant $\nu_F$ in Theorem~\ref{Thm: Euclidean Random Field}. Notice that by Lemma~\ref{L: Integral Sandwich}, we have that for all $R >0$, 
        \begin{equation}
            \frac{\mc{E}(N(R;F))}{|C_R|} \leq \nu_F.
        \end{equation}
    Hence, we only need to show that there is an $R = R_0$ such that $\mc{P}(N(R_0;F)>0)>0$.
    
    According to (\ref{E: APP D Positive Approximation}) (also see Proposition 5.2 of~\cite{SW19}), together with Theorem~\ref{Thm: APP D} and the fact that $\mc{H}_0(\rho) = \mc{H}(F)$, we see that for any $\epsilon > 0 $ and any $G \in C^1(\mb{R}^n, \mb{R}^m)$,
        \begin{equation}
            \mc{P}(||F - G||_{C^1(B_{10}(0))} < \epsilon ) >0,
        \end{equation}
    where $B_{10}(0)$ is the ball with center $0$ and radius $10$ in $\mb{R}^n$.
    We consider 
    \begin{equation}
        G = (g_1 , \dots,g_m) \equiv (x_1, \dots,x_{m-1},1- |x_m|^2 -|x_{m+1}|^2 - \dots |x_n|^2),
    \end{equation}
    so that $Z(G)$ is a unit $\mb{S}^{n-m}$ contained in the linear subspace of coordinates $x_m, \dots,x_n$.
    Hence, there is an $\epsilon = \epsilon(n,m)$ chosen in Lemma~\ref{L: Continuous Stability} such that
        \begin{equation}
            \mc{P}(N(10;F) \geq 1) \geq \mc{P}(||F - G||_{C^1(C_{B_{10}(0)})} < \epsilon )>0.
        \end{equation}
    This shows the positivity of $\nu_F$. 
    
    Moreover, under our assumptions, by Lemma~\ref{L: Continuous Stability}, there is a component $\gamma$ of $Z(F)$ in $B_{10}(0)$ that is $C^1$-isotopic to the component of $Z(G)$, which is $\mb{S}^{n-m}$ for this $G$. But actually, one can replace $G$ here with any $G \in C^1(B_{10}(0), \mb{R}^m)$ such that $Z(G) \cap B_1(0) \neq \emptyset$ and $0$ is a regular value for $G$.
    For similar reasons, one can show that for any $C^1$-isotopic type $c$ which can be realized as a regular connected component of some $Z(G)$ in $B_1(0)$, we always have that
        \begin{equation}
            \mc{P}(N(10;F, H) \geq 1) \geq \mc{P}(||F - G||_{C^1(C_{B_{10}(0)})} < \epsilon )>0.
        \end{equation}
    
    \begin{remark}
        We can also show that $\nu_F>0$ and $\nu_{F,c}>0$ when $F$ consists of independent $f_i$, $i =1 ,\dots,m$, and each $f_i$ is a frequency $1$ random wave on $\mb{R}^n$, i.e.,
            \begin{equation}
                \mc{E}(f_{i_1}(u) f_{i_2}(v)) = \delta_{i_1 i_2} \cdot \int_{\mb{S}^{n-1}} e^{i \langle u-v , \xi \rangle } \ d\mc{H}^{n-1}(\xi),
            \end{equation}
        and solves
            \begin{equation}
                \Delta f_i + f_i = 0.
            \end{equation}
        In this case, one can use a similar proof for Theorem 1 in~\cite{CS19}, which is in their section 3. The rough idea is as follows. We consider an arbitrary $(n-m)$-dimensional connected closed submanifold $\gamma$ of $\mb{R}^{n+q}$ fully contained in $B_1(0)$, and this $\gamma$ is also realized as a regular connected component of zero sets of some $C^1$-smooth function $G$, i.e., $\gamma \subset Z(G) \cap B_1(0)$, and $\nabla G$ is of full rank on $\gamma$.
        Our goal is to modify this $\gamma$ such that it is $C^1$-isotopic to another $\tilde{\gamma}$ which can be expressed as the transversal intersection of $m$ closed analytic hypersurfaces, and these $m$ closed hypersurfaces also enclose $m$ bounded domains with the same first Dirichlet eigenvalue.
    
        Denote that $G = (g_1 , \dots g_m)$ and assume that $0$ is a regular value of each $g_i$, otherwise one just replaces each $g_i$ with some $g_i+t_i$ with $t_i$ small, and get a new component $C^1$-isotopic to the original $\gamma$. We do not know whether each $Z(g_i)$ is a closed hypersurface yet. But since each $Z(g_i)$ is a regular hypersurface, one can replace each $g_i$ with $\tilde{g}_i \equiv {g_i(x)}^2 + \epsilon_i |x|^2 - \delta_i ^2$ for very small $\epsilon_i$ and very small $\delta_i$, such that $0$ is still a regular value of $\tilde{g}_i$ by Sard's theorem. 
        Notice that $Z(\tilde{g}_i) \subset B_{\delta_i/\epsilon_i}(0)$. Hence, we get a new component $\tilde{\gamma}$ that can be realized as a part of the intersection of those $Z(\tilde{g}_i)$ with each $Z(\tilde{g}_i)$ bounded. This $\tilde{\gamma}$ is also $C^1$-isotopic to $\gamma$ since $\epsilon_i$ and $\delta_i$ are small. 
        We can also assume that each $\tilde{g}_i$ is $C^\infty$-smooth.

        Assume that $\tilde{\gamma} = \cap_i c_i$, where each $c_i$ is a bounded connected component of $Z(\tilde{g}_i)$ and is also a closed hypersurface. Let $A_i$ be the bounded component of $\mb{R}^n \backslash c_i$. Now,  for each $A_i$, one can make smooth connected sum with long and thin necks to some large balls $B_i$ and avoid touching $\tilde{\gamma}$.
        We do these connected sums and get new bounded sets $\mc{A}_i$, such that all these ${\{\mc{A}_i\}}_{i=1} ^m$ have the same first Dirichlet eigenvalue $\lambda$. 
        This same first Dirichlet eigenvalue property is possible, because the first Dirichlet eigenvalue for a given domain is always larger than the first eigenvalues of domains which contain it, and we can adjust the size of each ball $B_i$ to let all $\mc{A}_i$ have the same first eigenvalue.
        Now, $\tilde{\gamma}$ becomes a part of the transversal intersection of those $\pa \mc{A}_i$.
        
        For each $\mc{A}_i$, we assume that $\partial \mc{A}_i$ is realized as a component of the regular zero set of another smooth function $\bar{g}_i$. Approximate each $\bar{g_i}$ by another analytic function $\mathbf{g}_i$ defined in a neighborhood of $\partial \mc{A}_i$, we see that the set $\{\mathbf{g}_i = 0\}$ has a bounded component $\widetilde{c_i}$ which is very close to $\partial \mc{A}_i$, and these $\widetilde{c_i}$ form another transversal intersection $\tilde{\gamma}_1$, which is $C^1$-isotopic to $\tilde{\gamma}$. 
        The existence of such an analytic $\mathbf{g}_i$ is by the same arguments as proving Theorem 1 on page 7 of~\cite{CS19}, where they used the Whitney approximation Theorem.
    
        The new problem here is that $\widetilde{\mc{A}_i}$, the bounded component of $\mb{R}^n \backslash \widetilde{c_i}$, may not have the same first Dirichlet eigenvalue. To fix this, since $\mc{A}_i$ have the same first Dirichlet eigenvalue, and each $\mc{A}_i$ is very close to $\widetilde{\mc{A}_i}$, we may scale $\widetilde{\mc{A}_i}$ a little, or equivalently, replace $\mathbf{g}_i(x)$ with $\mathbf{g}_i(\eta_i x)$ for some $\eta_i$ close to $1$, and this small scaling will not affect the transversal intersection $\tilde{\gamma}_1$ up to another $C^1$-isotopy.

        To summarize, we get a $\tilde{\gamma}_1$ in the same $C^1$-isotopy class as the original $\gamma$, which can also be expressed as the transversal intersection of analytic hypersurfaces $\widetilde{c_i} = \partial \widetilde{\mc{A}_i}$, where all these $\widetilde{\mc{A}_i}$ are bounded and have the same first Dirichlet eigenvalue.
    
        Then, the remaining arguments follow easily from the proof on page 8 of~\cite{CS19}.

    \end{remark}


    \section{Assumptions (A3) in Definition~\ref{D: Axiom 1} and (B4) in Definition~\ref{D: Axiom 3}}\label{APP: 2 Point}
    
    As mentioned previously, we can use $(A1)$ and $(A2)$ to deduce $(A3)$ in Definition~\ref{D: Axiom 1}. Let us explain it more. 
    The two-point  $m \times m$ covariance matrix $\mr{Cov}(x,y)$ is defined by
        \begin{equation}
            {(\mr{Cov}(x,y))}_{i_1 i_2} \equiv \ \mc{E}(f_{i_1}(x)f_{i_2}(y)), \ i_1,i_2 = 1, \dots,m.
        \end{equation}
    \begin{theorem}\label{Thm: APP C 1}
        \textit{
        Assume that $F : C_{R+1} \to \mb{R}^m$ satisfies \textit{ $(R; M,k_1)$-assumptions} in Definition~\ref{D: Axiom 1}. Then, there is a $\delta = \delta(n,m,M,k_1) \in (0,1)$ such that when $x \in C_R$ and $||x-y|| \leq \delta$, we have that for any $\eta \in \mb{S}^{m-1} \subset \mb{R}^m$,
            \begin{equation}\label{E: APP C Smallest Eigenvalue}
                \eta^T \cdot  \big(\mr{Cov}(y,y) - \mr{Cov}(y,x) {\mr{Cov}(x,x)}^{-1} \mr{Cov}(x,y) \big) \cdot \eta \geq \frac{k_1}{2}  \cdot ||y-x||^2 .
            \end{equation}
        In particular, there is a $k_2 = k_2(n,m,M,k_1)>0$ such that when $x\in C_R$ and $||x-y|| <\delta$, then
            \begin{equation}\label{E: APP C Local (A3) Condition}
                p_{(F(x),F(y))}(0,0) \leq \frac{k_2}{||x-y||^m} ,
            \end{equation}
        which also implies that the joint distribution of $(F(x),F(y))$ is non-degenerate if $||x-y|| \leq \delta$.
        }
    \end{theorem}
    \begin{proof}
        Notice that (\ref{E: APP C Local (A3) Condition}) follows from (\ref{E: APP C Smallest Eigenvalue}) by elementary row operations in calculating the determinant of the covariance kernel for the joint distribution of $(F(x), F(y))$.
        To prove (\ref{E: APP C Smallest Eigenvalue}), we fix an arbitrary $x\in C_R$. For any $w = {(w_1, \dots, w_n)}^T \in \mb{S}^{n-1}$, we define an $m \times m$ symmetric matrix-valued function on $[-1,1] \subset \mb{R}$:
            \begin{equation}
                k_w(t) \equiv \mr{Cov}(x +t \cdot w,x + t \cdot w) - \mr{Cov}(x+ t \cdot w,x) {\mr{Cov}(x,x)}^{-1} \mr{Cov}(x,x+t \cdot w) .
            \end{equation}
        Since $F$ satisfies $(A1)$ in Definition~\ref{D: Axiom 1}, $\mr{Cov}(x,x)$ is invertible and hence $k_w(t)$ is well defined on $[-1,1]$.
        It is easy to see that $k_w(0) = 0$ and $k_w '(0) = 0$. Then,
            \begin{equation}\label{E: APP C Second Derivative}
                \begin{split}
                    & \quad {(k_w ''(t) )}_{i_1 i_2}
                \\     &= \mc{E}((\nabla^2 _{w,w}f_{i_1})(x+tw) f_{i_2}(x+tw)) + \mc{E}(f_{i_1}(x+tw) (\nabla^2 _{w,w}f_{i_2})(x+tw)) 
                \\	& \quad +  2 \mc{E}((\nabla_w f_{i_1})(x+tw) (\nabla_w f_{i_2})(x+tw)) 
                \\	& \quad -\sum_{p,q} \mc{E}((\nabla^2 _{w,w}f_{i_1})(x+tw)f_{p}(x) ) {({\mr{Cov}(x,x)}^{-1})}_{pq} {(\mr{Cov}(x,x+t \cdot w) )}_{q i_2}
                \\	& \quad -\sum_{p,q} {(\mr{Cov}(x+ t \cdot w,x) )}_{i_1 p} {({\mr{Cov}(x,x)}^{-1})}_{pq}  \mc{E}(f_{q}(x)(\nabla^2 _{w,w}f_{i_2})(x+tw) )
                \\	& \quad -  2\sum_{p,q} \mc{E}((\nabla_w f_{i_1})(x+tw)f_{p}(x) ) \cdot  {( {\mr{Cov}(x,x)}^{-1} )}_{pq} \cdot  \mc{E}(f_{q}(x)(\nabla _{w}f_{i_2})(x+tw) ).
                \end{split}
            \end{equation}
        By Cauchy's mean value theorem, we have that for any $i_1, i_2 = 1, \dots , m$, there is a $\xi_{i_1 i_2}$ between $0$ and $t$, such that
            \begin{equation}
                {(k_w(t))}_{i_1 i_2} = \frac{t^2}{2}  \cdot {(k_w ''(\xi_{i_1 i_2}))}_{i_1 i_2} . 
            \end{equation}
        Furthermore, notice that $F$ is $C^{3-}$-smooth, by assumption $(A2)$, we have that
            \begin{equation}
                | {(k_w ''(\xi_{i_1 i_2}))}_{i_1 i_2} - {(k_w '' (0) )}_{i_1 i_2}| \leq \sqrt{\xi_{i_1i_2}} \cdot C_1 \leq \sqrt{t} \cdot C_1 
            \end{equation}
        for some positive constant $C_1 = C_1(n,m,M)$. We set the remaining term as the symmetric matrix $R$ with elements $ R_{i_1i_2}\equiv t^{-1/2} \cdot [{(k_w ''(\xi_{i_1 i_2}))}_{i_1i_2} - {( k_w '' (0))}_{i_1 i_2}]$, so that 
            \begin{equation}\label{E: APP C Expansion}
                k_w(t) = \frac{t^2}{2} \big(k_w ''(0) + \sqrt{t} R \big) .
            \end{equation}
        For the expression of $k_w '' (0)$, we set $t = 0$ in (\ref{E: APP C Second Derivative}) and then get that
            \begin{equation}
                \begin{split}
                    {(k_w ''(0))}_{i_1 i_2} &= 2 \big[ \mc{E}((\nabla_w f_{i_1})(x) (\nabla_w f_{i_2})(x)) 
                \\ &\quad -  \sum_{p,q} \mc{E}((\nabla_w f_{i_1})(x)f_{p}(x) ) \cdot {( {\mr{Cov}(x,x)}^{-1} )}_{pq}  \cdot  \mc{E}(f_{q}(x)(\nabla _{w}f_{i_2})(x) ) \big].
                \end{split}
            \end{equation}
        For any $\eta = {(\eta_1, \dots ,\eta_m)}^T\in \mb{S}^{m-1}$, we let
            \begin{equation}
                v_\eta = \sum_{i=1}^m \eta_i  ((\nabla_w f_i)(x)) = \sum_{i=1}^m \sum_{j=1} ^n \eta_i w_j ((\pa_j f_i)(x)).
            \end{equation}
        Notice that $\mr{Cov}(x,x)$ is a positive definite matrix, so it can be written as the square 
        of another positive definite matrix, i.e., $\mr{Cov}(x,x) = B^2$ and $B$ is positive definite. Set the Gaussian vector $G = (g_1 ,\dots ,g_m) \equiv {F(x)}^T \cdot B^{-1}$ so that $\mc{E}(g_{i_1} g_{i_2}) = \delta_{i_1i_2}$.
        Then,
            \begin{equation}
                \begin{split}
                    \frac{1}{2}\eta^T \cdot k_w ''(0) \cdot \eta &= \mc{E}(v_\eta^2) - \sum_p \mc{E}(v_\eta g_p) \cdot \mc{E}(g_p v_\eta)
                \\	&= \mc{E}\big[{\big(v_\eta - \sum_p g_p  \cdot \mc{E}(v_\eta g_p)\big)}^2 \big].
                \end{split}
            \end{equation}
        Notice that $v_\eta$ is a linear combination of $\nabla F(x)$ terms while $\sum_p g_p  \cdot \mc{E}(v_\eta g_p)$ is a linear combination of $F(x)$ terms. Since the joint distribution of $(F(x),\nabla F(x))$ is non-degenerate by $(A1)$ of Definition~\ref{D: Axiom 1}, we see that
            \begin{equation}
                \mc{E}\big[{\big(v_\eta - \sum_p g_p  \cdot \mc{E}(v_\eta g_p)\big)}^2 \big] \geq k_1 \cdot \bigg|\sum_{i =1 }^m \sum_{j=1}^n {(\eta_i w_j)}^2 \bigg| = k_1,
            \end{equation}
        where for the inequality, we only collect the coefficients for $\nabla F(x)$ terms.
        Hence, by (\ref{E: APP C Expansion}), there is a $\delta = \delta(n,m,M,k_1) \in (0,1)$ such that 
        when $|t| \leq \delta$, we have that 
            \begin{equation}
                \eta^T \cdot k_w (t) \cdot \eta = \frac{t^2}{2}\eta^T \cdot (k_w ''(0) + \sqrt{t}R) \cdot \eta \geq \frac{t^2}{2} \cdot k_1.
            \end{equation}

    \end{proof}
    
    \begin{remark}
        Theorem~\ref{Thm: APP C 1} gives an upper bound for the density $p_{(F(x),F(y))}(0,0)$ when $0<||x-y|| < \delta$. If we know that $(F(x),F(y))$ is always non-degenerate when $x\neq y$ and $x,y \in C_{R+1}$, then condition $(A3)$ in Definition~\ref{D: Axiom 1} will hold true automatically. Then, $k_2$ depends on 
    the covariance kernel fo the joint distribution of $(F(x),F(y))$ when $||x-y|| > \delta$.
    
    For condition $(B4)$ in Definition~\ref{D: Axiom 3}, in most of applications and concrete examples like Kostlan's ensemble or complex arithmetic random waves, first, we will know that the convergence for kernels ${\{K_{x,L}\}}_{L \in \mc{L}}$ is in $C^k$, i.e., for all $R > 0$,
        \begin{equation}
            \lim_{L \to \infty} ||K_{x,L}(u,v) - K_x(u-v)||_{C^k(C_{R+1} \times C_{R+1})} = 0.
        \end{equation}
    In this case, conditions $(B1)$, $(B2)$ and $(B4)$ will hold true if for all compact set $Q \subset U$, and $x \in Q$, $K_x(u-v)$ satisfies a uniform \textit{ $(R; M(Q),k_1(Q),k_2(x))$-assumptions}. In particular, for $(B4)$, if we know that for all $x \in U$ and $R > 0$,
        \begin{equation}
            \lim_{L \to \infty} \sup_{u,v \in C_{R+1}}||K_{x,L}(u,v) - K_x(u-v)|| = 0,
        \end{equation}
    and $K_x$, or equivalently $F_x$, satisfies $(A3)$ in Definition~\ref{D: Axiom 1} for some $k_2 = k_2(x)$, then when $||u-v|| > \delta$, the quantitative two-point nondegeneracy of $K_{x,L}$, or equivalently $F_{x,L}$, holds true from this uniform convergence.
    When $||u-v|| \leq \delta$, one can just apply Theorem~\ref{Thm: APP C 1} because if $K_{x,L}$ converges to $K_x$ in $C^k$, then $K_{x,L}$ satisfies \textit{ $(R; 2M(Q),k_1(Q)/2)$-assumptions} since $K_x$ satisfies \textit{ $(R; M(Q),k_1(Q))$-assumptions}.
    \end{remark}



\end{appendix}

\section*{Declarations}
\bigskip

{\bf Acknowledgements} The author would like to thank his advisor, Professor Fang-Hua Lin, for his continuous support and encouragement. 
The author would also like to thank Professor Paul Bourgade and Professor Robert V.~Kohn for very helpful discussions and suggestions in the finalizing stage.
Professor Yuri Bakhtin brought~\cite{GS04} and Professor Ao Sun brought~\cite{CL58} to the author's attention, and the author wants to thank them. 
The author also wants to thank Professor Ao Sun and Professor Chao Li for their encouragement in 2021.
This work was completed when the author studied at NYU as a PhD student, partially supported by the NSF grants DMS2247773, DMS2055686, and DMS2009746.


\bibliographystyle{acm} 
\bibliography{RandomZeroSetsNewIntroduction3.bib}

\end{document}